\documentclass[a4paper,12pt]{article}
\usepackage{latexsym}
\usepackage{mathrsfs}
\usepackage{amsfonts,amsmath,amssymb}
\usepackage{indentfirst}
\usepackage{subeqnarray}
\usepackage{verbatim}
\usepackage{graphicx}
\usepackage{epstopdf}
\usepackage{algorithm}
\usepackage{algorithmic}
\usepackage{subfigure}
\usepackage{color}
\usepackage{multirow}
\usepackage{float}

\allowdisplaybreaks[4]

\newtheorem{theorem}{Theorem}[section]
\newtheorem{lemma}[theorem]{Lemma}
\newtheorem{assumption}[theorem]{Assumption}

\newtheorem{definition}[theorem]{Definition}

\newtheorem{proposition}[theorem]{Proposition}
\newtheorem{remark}[theorem]{Remark}
\newtheorem{condition}[theorem]{Condition}
\numberwithin{equation}{section}
\newenvironment{proof}[1][Proof]{\textbf{#1.} }
{\ \rule{0.75em}{0.75em}\smallskip}

\begin{document}

\begin{center}
\Large\bf $\ell_{1}^{2}-\eta\ell_{2}^{2}$ regularization for sparse recovery
\end{center}

\begin{center}
Long Li\footnotemark[1] \quad and \quad Liang Ding\footnotemark[2]$^{,*}$

\footnotetext[1]{Department of Mathematics, Northeast Forestry University, Harbin 150040, China;
e-mail: {\tt 15146259835@nefu.edu.cn}.}
\footnotetext[2]{Department of Mathematics, Northeast Forestry University, Harbin 150040, China;
e-mail: {\tt dl@nefu.edu.cn}. The work of this author was supported by the Fundamental Research Funds for the Central Universities (no.\ 2572021DJ03).} 

\renewcommand{\thefootnote}{\fnsymbol{footnote}}
\footnotetext[1]{Corresponding author.}
\renewcommand{\thefootnote}{\arabic{footnote}}
\end{center}

\medskip
\begin{quote}
{\bf Abstract.}
This paper presents a regularization technique incorporating a non-convex and non-smooth term, $\ell_{1}^{2}-\eta\ell_{2}^{2}$, with parameters $0<\eta\leq 1$ designed to address ill-posed linear problems that yield sparse solutions. We explore the existence, stability, and convergence of the regularized solution, demonstrating that the $\ell_{1}^{2}-\eta\ell_{2}^{2}$ regularization is well-posed and results in sparse solutions. Under suitable source conditions, we establish a convergence rate of $\mathcal{O}\left(\delta\right)$ in the $\ell_{2}$-norm for both a priori and a posteriori parameter choice rules. Additionally, we propose and analyze a numerical algorithm based on a half-variation iterative strategy combined with the proximal gradient method. We prove convergence despite the regularization term being non-smooth and non-convex. The algorithm features a straightforward structure, facilitating implementation. Furthermore, we propose a projected gradient iterative strategy base on surrogate function approach to achieve faster solving. Experimentally, we demonstrate visible improvements of $\ell_{1}^{2}-\eta\ell_{2}^{2}$ over $\ell_{1}$, $\ell_{1}-\eta\ell_{2}$, and other nonconvex regularizations for compressive sensing and image deblurring problems. All the numerical results show the efficiency of our proposed approach.
\end{quote}


\section{Introduction}\label{sec1}

\par This paper addresses the solution of an ill-posed operator equation of the form
\begin{equation}\label{equ1.1}
    Ax=y,
\end{equation}
where $A:\ell_{2}\rightarrow Y$ is a linear and bounded operator, $x$ is sparse and $Y$ is a Hilbert space with norm $\left\|\cdot\right\|_{Y}$.\ Throughout this paper, $\left<\cdot,\cdot\right>$ denotes the inner product in the $\ell_{2}$ space. In practical applications, the data $y$ is not provided exactly. Instead, we have an approximation $y^{\delta}$ such that $\left\|y-y^{\delta}\right\|_{Y}\leq\delta$ for a small $\delta >0$.\ The most widely adopted approach to solve the ill-posed operator equation \eqref{equ1.1} is through sparsity regularization, formulated as
\begin{equation}\label{equ1.2}
    \min_{x}\left\{\mathcal{J}_{\alpha}\left(x\right)=\frac{1}{q}\left\|Ax-y^{\delta}\right\|_{Y}^{q}+ \alpha\left\|x\right\|_{w,p}^{p}\right\},
\end{equation}
where $1\leq p<2$, $1\leq q\leq 2$, and $\alpha>0$. The term $\frac{1}{q}\left\|Ax-y^{\delta}\right\|_{Y}^{q}$ represents the fidelity term that quantifies the discrepancy between the data $y^{\delta}$ and the model, while $\left\|x\right\|_{w,p}^{p}=\sum_{i}w_{i}\left|\left<\phi_{i},x\right>\right|^{p}$ accounts for the weights, with $\left\{\phi_{i}\right\}$  being an orthonormal basis. In \cite{DDD04}, Daubechies employed a wavelet as an orthonormal basis and introduced the weighting $w=\mu w_{0}$, where $\mu>0$ is a constant and $w_{0}$ is a sequence with all entries equal to 1. Over the past two decades, the concept of sparsity has gained considerable popularity, leading to extensive research into well-posedness issues and the development of algorithms for tackling sparsity regularization problems (see \cite{F10,GL12,JM12} and the respective references therein).

\subsection{Some iterative algorithms for sparse regularization}

\par Sparse regularization stems from the introduction of a priori information through the $\ell_{1}$-norm, which serves as a convex relaxation of the $\ell_{0}$-norm to encourage sparsity in the regularization functional. In \cite{DDD04}, Daubechies introduced an iterative soft threshold algorithm (ISTA) designed for solving linear inverse problems that impose sparsity constraints and demonstrated the convergence of this algorithm. Beck and Teboulle proposed the fast iterative shrinkage threshold algorithm (FISTA) in \cite{BT09}. This method not only maintains the simplicity of the ISTA but also achieves improved global convergence speed, both theoretically and in practice. In \cite{DDFG10}, an iteratively reweighted least squares (IRLS) algorithm was introduced for recovering sparse signals from underdetermined linear systems. Voronin enhanced this IRLS algorithm for the regularization of linear inverse problems with sparse constraints and provided convergence proof for the improved algorithm in \cite{VD15}. In \cite{BLR15}, Bredies analyzed a general framework for non-smooth and non-convex regularization using the generalized gradient projection method. The alternating direction method of multipliers (ADMM) was utilized in \cite{WYZ19} to tackle non-convex and non-smooth optimization problems, along with its corresponding theoretical convergence analysis and results. Finally, Doe and Smith improved the classical ISTA in \cite{DS24} by incorporating entropy regularization, enhancing algorithm performance in addressing ill-posed linear inverse problems.

\par Theoretically, the $\ell_{1}$ norm is not the most effective method for achieving sparsity in solutions. Consequently, research shifted towards exploring non-convex sparse regularization techniques, as these non-convex methods have attracted considerable scholarly attention and investigation. A non-convex $\ell_{p}$-norm sparse regularization with $0<p<1$ has been proposed as an alternative to the approach in \eqref{equ1.2}, as seen in \cite{BL09} and \cite{G10}. However, due to the non-convexity and non-differentiability of the non-convex regularization problem, its analysis and solution tend to be more complex, limiting the exploration of its regularization properties, particularly concerning the convergence rate. To analyze $\ell_{p}$-norm sparse regularization with $0<p<1$, specialized regularization techniques are required. In \cite{N13}, $\ell_{0}$-norm regularization problems are examined within finite-dimensional spaces. Meanwhile,  \cite{IK14} investigates sparsity optimization within infinite-dimensional sequence spaces $\ell_{p}$ with $0\leq p\leq 1$. Additionally, \cite{YLHX15} introduces a continuous regularization term that lies between the $\ell_{1}$-norm and $\ell_{2}$-norm, focusing on minimizing the regularization functional in finite-dimensional spaces to address the compressed sensing problem. In \cite{HM10}, a novel regularization term called \verb+"+sort $\ell_{1}$\verb+"+ is presented, accompanied by theoretical analysis, including considerations of convergence speed and stability estimates. For insights into the non-convex sparsity regularization of nonlinear ill-posed problems, one can refer to \cite{RZ12,Z09} and their corresponding references.

\par $\ell_{p}\ (0<p<1)$ sparse regularization represents a viable non-convex approximation of $\ell_{0}$ sparse regularization. The specific choice of the parameter $p$ is pivotal in determining the effectiveness of this approximation. However, to date, only the case where $p=1/2$ has been effectively solvable, and there is currently a lack of numerical algorithms capable of addressing other values of $p$. Furthermore, no evidence suggests that $p=1/2$ is the optimal non-convex approximation of $\ell_{p}(0<p<1)$. In \cite{DH19}, a non-convex $\ell_{1}-\eta\ell_{2}$ sparse regularization problem is introduced
\begin{equation}\label{equ1.3}
    \min\left\{\mathcal{J}_{\alpha,\beta}^{\delta}(x)= \frac{1}{q}\left\|Ax-y^{\delta}\right\|_{Y}^{q}+\alpha\left\|x\right\|_{\ell_{1}} -\beta\left\|x\right\|_{\ell_{2}}\right\},\ \left(\alpha\geq\beta>0\right).
\end{equation}
Additionally, the ST-$\left(\alpha\ell_{1}-\beta\ell_{2}\right)$ iterative algorithm for solving problem \eqref{equ1.3} when $q=2$ is proposed
\begin{align*}
    \left\{             
        \begin{array}{ll}
        z^{k}=\mathbb{S}_{\alpha/\gamma}\left(x^{k}+\frac{\beta}{\gamma\left\|x^{k}\right\| _{\ell_{2}}}x^{k}+\frac{1}{\gamma}A^{*}\left(Ax^{k}-y^{\delta}\right)\right), \\
        x^{k+1}=x^{k}+s^{k}\left(z^{k}-x^{k}\right), \\
        \end{array}
    \right.
\end{align*}
where $\mathbb{S}_{\alpha/\gamma}$ represents the soft-thresholding operator, $s^{k}$ is the step size and $\gamma>0$. However, the term $\beta\left\|x\right\|_{\ell_{2}}$ in \eqref{equ1.3} is non-differentiable, complicating the convergence analysis of the iterative algorithm and necessitating additional conditions to prevent the iteration value from being zero. Additionally, the term $x^{k}/\left\|x^{k}\right\|_{\ell_{2}}$ in the iterative algorithm unites $x^{k}$, which may result in slower updates for the algorithm when $\left\|x^{k}\right\|_{\ell_{2}}$ is large. A natural approach is to consider the regularization term $\alpha\left\|x\right\|_{\ell_{1}}-\beta\left\|x\right\|_{\ell_{2}}^{2}$ regularization term. However, it is essential to ensure that $\left\|x\right\|_{\ell_{2}}<1$, as violating this condition can lead to a negative regularization term, undermining the functional's regularization properties. While it is possible to scale $x$ to meet the requirement $\left\|x\right\|_{\ell_{2}}<1$, this remains a strict condition since the value of $\left\|x\right\|_{\ell_{2}}$ is unknown prior to solving the problem.

\subsection{Contribution and organization}

In this paper, by the proposal of $\left\|x\right\|_{\ell_{1}}^{2}$ proximal operator in \cite{BA17}, we consider a non-convex $\ell_{1}^{2}-\eta\ell_{2}^{2}$ sparse regularization to address these challenges. This reformulation allows us to avoid the complexities associated with non-differentiable terms by substituting $\left\|x\right\|_{\ell_{2}}$ with $\left\|x\right\|_{\ell_{2}}^{2}$, enabling the proposal of an algorithm that leverages the proximal gradient method to achieve improved solution results.

\begin{figure}[htbp]
    \centering
    \subfigure[$\ell_{1}$]{\includegraphics[width=0.25\columnwidth,height=0.2\linewidth]{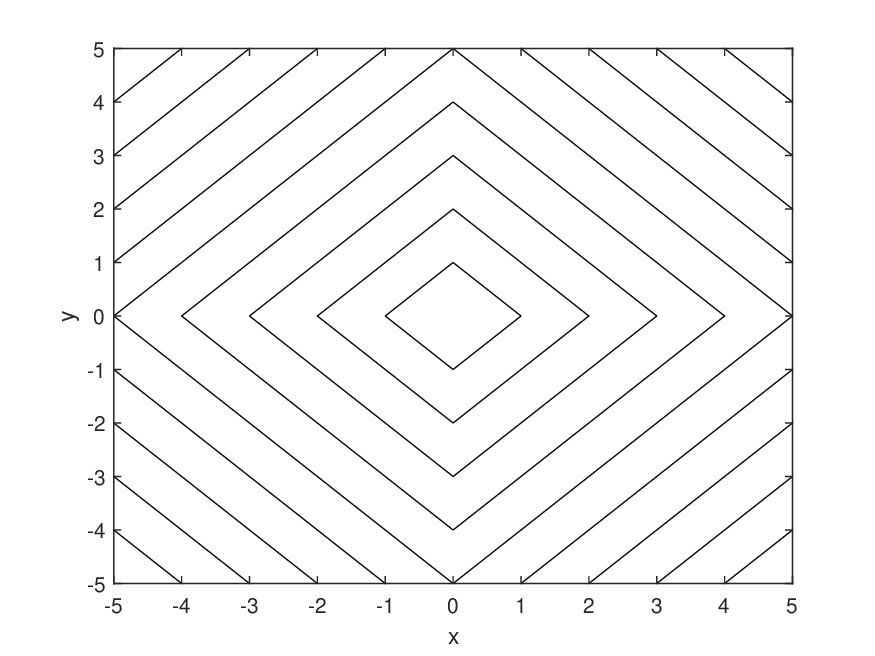}}
    \subfigure[$\ell_{1/2}$]{\includegraphics[width=0.25\columnwidth,height=0.2\linewidth]{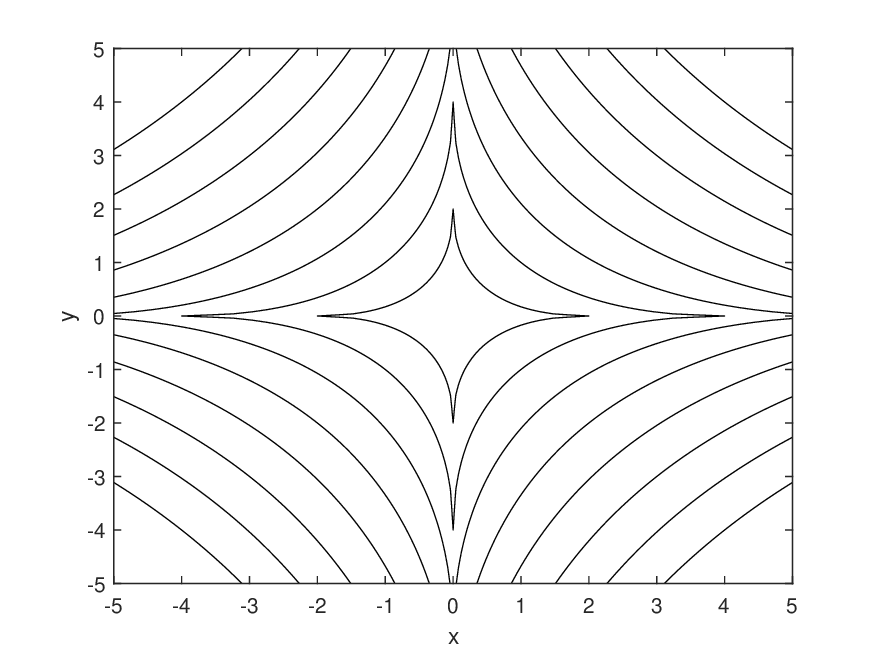}}
    \subfigure[$\ell_{1}^{2}-\eta\ell_{2}^{2}$ with $\eta=0.1$]{\includegraphics[width=0.25\columnwidth,height=0.2\linewidth]{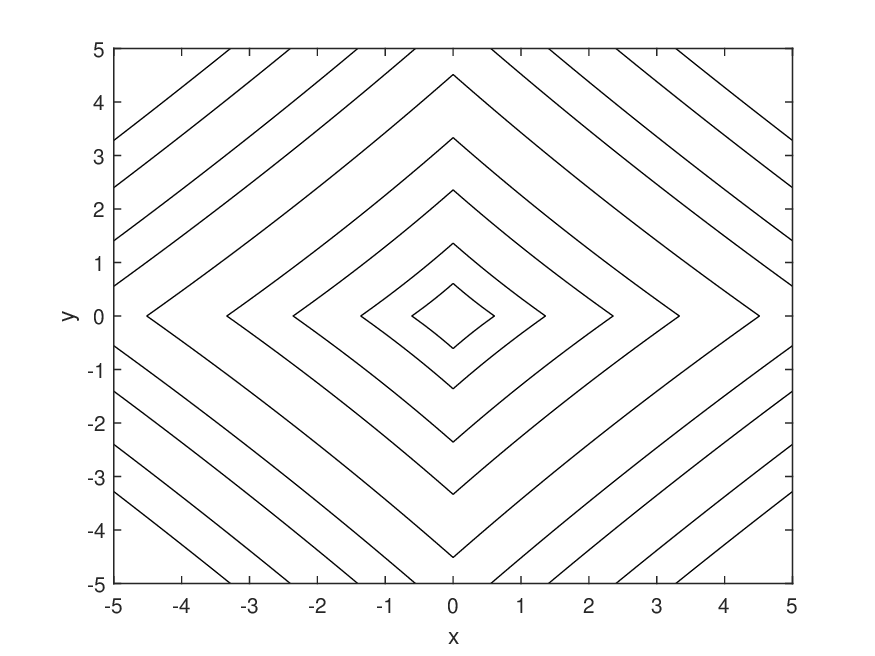}}\\
    \subfigure[$\ell_{1}^{2}-\eta\ell_{2}^{2}$ with $\eta=0.5$]{\includegraphics[width=0.25\columnwidth,height=0.2\linewidth]{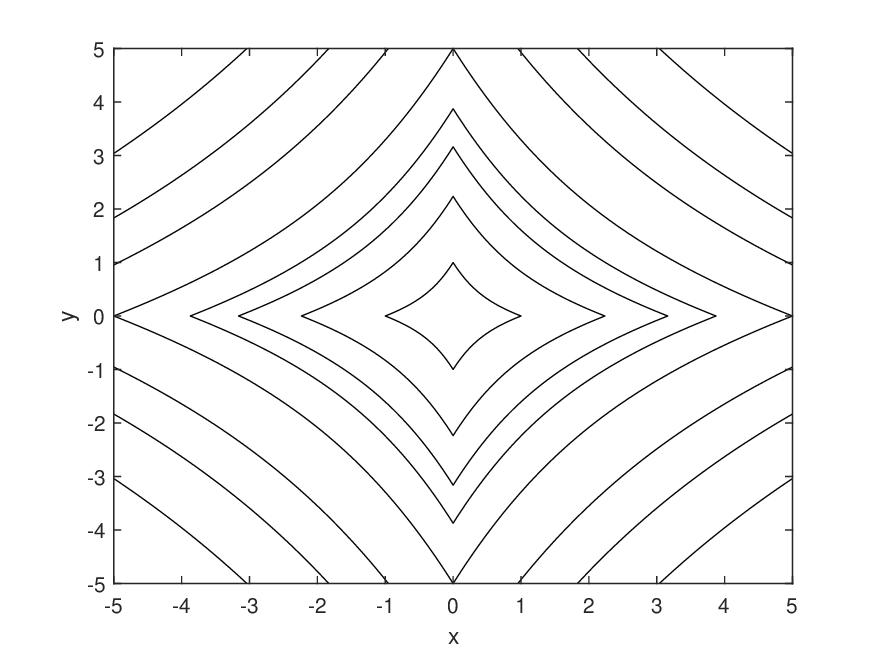}}
    \subfigure[$\ell_{1}^{2}-\eta\ell_{2}^{2}$ with $\eta=0.9$]{\includegraphics[width=0.25\columnwidth,height=0.2\linewidth]{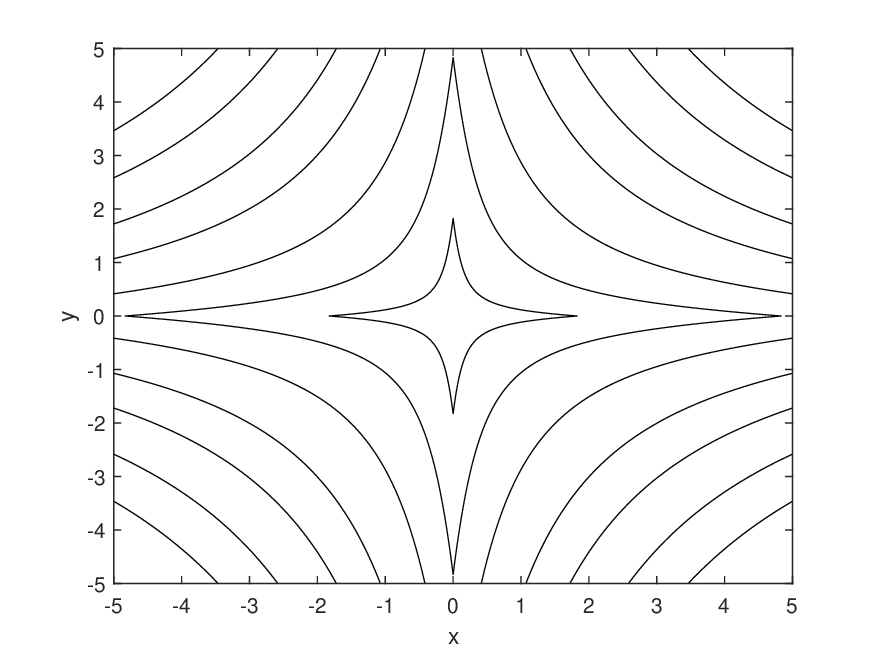}}
    \subfigure[$\ell_{1}^{2}-\eta\ell_{2}^{2}$ with $\eta=1$]{\includegraphics[width=0.25\columnwidth,height=0.2\linewidth]{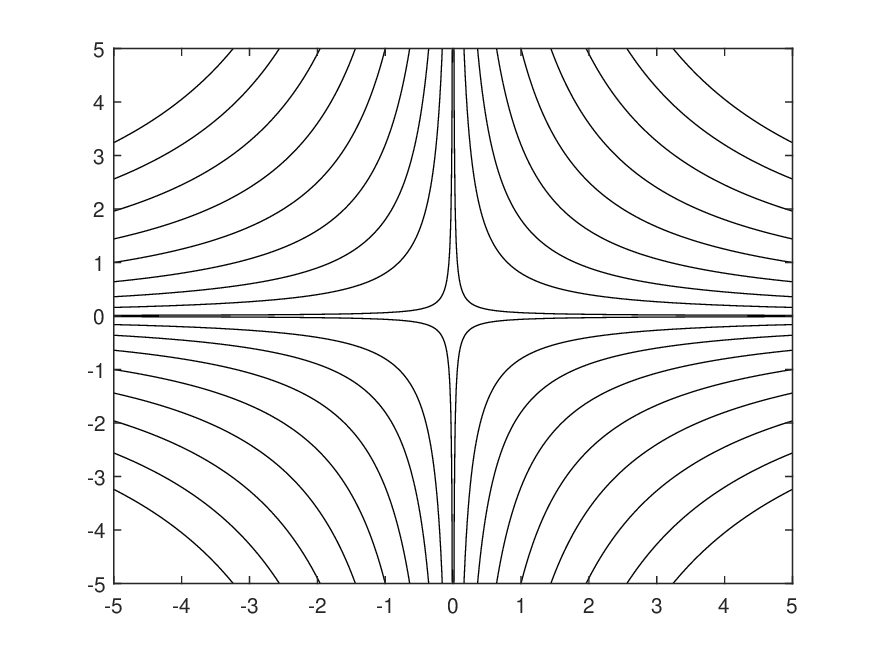}}
    \caption{Contour plots of $\ell_{1}$, $\ell_{1/2}$ and $\ell_{1}^{2}-\eta\ell_{2}^{2}$ with different $\eta$.}
    \label{figure:1}
\end{figure}

\par The objective of this paper is to explore the following regularization method for solving the ill-posed linear equation \eqref{equ1.1}:
\begin{equation}\label{equ1.4}
    \min\left\{\mathcal{J}_{\alpha,\beta}^{\delta}(x)= \frac{1}{q}\left\|Ax-y^{\delta}\right\|_{Y}^{q}+\mathcal{R}_{\alpha,\beta}(x)\right\}
\end{equation}
in $\ell_{2}$ space, utilizing the standard $\ell_{2}$-norm $\left\|\cdot\right\|_{\ell_{2}}$, where $1\leq q\leq 2$. We define the regularization term as follows
\begin{equation}\label{equ1.5}
    \mathcal{R}_{\alpha,\beta}\left(x\right):=\alpha\left\|x\right\|_{\ell_{1}}^{2}-\beta\left\|x\right\|_{\ell_{2}}^{2},\ \alpha\geq\beta >0.
\end{equation}
By setting $\eta=\beta/\alpha$, we can equivalently express the functional in \eqref{equ1.5} as
\begin{equation*}
    \mathcal{R}_{\alpha,\beta}(x)=\alpha\mathcal{R}_{\eta}(x):= \alpha\scalebox{1.2}{(} \left\|x\right\|_{\ell_{1}}^{2}-\eta\left\|x\right\|_{\ell_{2}}^{2}\scalebox{1.2}{)},
\end{equation*}
where $\alpha>0$ and $1\geq\eta> 0$. In Fig.\ \ref{figure:1}, we present contour plots of $\ell_{1}$, $\ell_{1/2}$, and $\ell_{1}^{2}-\eta\ell_{2}^{2}$ with different $\eta=\beta/\alpha$. It is observed that $\mathcal{R}_{\eta}\left(x\right)$ exhibits behavior akin to the $\ell_{0}$-norm as $\eta\rightarrow 1$, and converges to a constant multiple of the $\ell_{1}$-norm as $\eta\rightarrow 0$. When $\eta=1$, $\mathcal{R}_{\eta}\left(x\right)$ serves as a good approximation of a constant multiple of $\left\|x\right\|_{\ell_{0}}$. However, the contour of $\mathcal{R}_{\eta}\left(x\right)$ does not intersect the coordinate axes. The primary motivation for investigating minimization with the regularization term \eqref{equ1.5} is that $\left\|x\right\|_{\ell_{1}}^{2}-\eta\left\|x\right\|_{\ell_{2}}^{2}$ can be perceived as an approximation of $\left\|x\right\|_{\ell_{0}}$. This approach presents a simpler structure compared to regularization techniques using the $\ell_{0}$-norm and $\ell_{p}$-norms for $0<p <1$. Commonly applied norms, such as the $\ell_{1}$-norm and $\ell_{2}$-norm, along with their derivatives, can be easily computed, facilitating a numerical solution to problem \eqref{equ1.4} through an iterative half-variation algorithm.

\par This paper investigates the regularizing properties and numerical algorithms associated with problem \eqref{equ1.4}. It presents four key contributions:

\textbullet\ We provide a comprehensive proof of the coercivity, weak lower semi-continuity, and Radon-Riesz property of the regularization term $\mathcal{R}_{\eta}(x)=\left\|x\right\|_{\ell_{1}}^{2}-\eta\left\|x\right\|_{\ell_{2}}^{2}$. Building upon classical regularization techniques, we further investigate the well-posedness of the regularization method, which includes analyses of existence, stability, convergence, and the sparsity of the solution. Given the non-convex nature of the regularization term $\mathcal{R}_{\eta}(x)$, additional considerations are necessary. Under a supplementary source condition, we derive an inequality that facilitates our analysis. Utilizing this inequality, we establish that the convergence rate in the $\ell_{2}$-norm is $\mathcal{O}(\delta)$.

\textbullet\ For the numerical method, we adopt a technique presented in \cite{BA17} to express the functional $\mathcal{J}_{\alpha,\beta}^{\delta}$ in \eqref{equ1.4} in the form
\begin{equation*}
  \mathcal{J}_{\alpha,\beta}^{\delta}\left(x\right)=f\left(x\right)+g\left(x\right),
\end{equation*}
where $f\left(x\right)=\frac{1}{2}\left\|Ax-y^{\delta}\right\|_{Y}^{2} -\beta\left\|x\right\|_{\ell_{2}}^{2}$ and $g\left(x\right)=\alpha\left\|x\right\|_{\ell_{1}}^{2}$. We establish that $f\left(x\right)$ and $g\left(x\right)$  satisfy the requisite smoothness and convexity conditions necessary for the application of the proximal gradient method, as delineated in \cite{BA17}. We further propose a half variation (HV) iterative algorithm for addressing problem \eqref{equ1.4}, articulated as
\begin{equation*}
    x^{k+1}=\mathbb{H}_{\alpha}\left(x^{k}+ \frac{2\beta}{L_{k}} x^{k}- \frac{1}{L_{k}}A^{*}\left(Ax^{k}-y^{\delta}\right)\right),
\end{equation*}
where $L_{k}\in (\frac{L}{2},\infty)$ and
\begin{align*}
    \mathbb{H}_{\alpha}(x)=
    \left\{           
             \begin{array}{ll}
             \sum_{i=1}^{n}\frac{\lambda_{i}}{2\alpha+\lambda_{i}}\left<x, e_{i}\right>e_{i}, & if \ x\neq 0, \\
             0, & if \ x= 0. \\
             \end{array}
    \right.
\end{align*}
where $L$ denote the Lipschitz constant of $f'\left(x\right)$, $\lambda_{i}$ is defined in Lemma \ref{lem4.4}, and $e_{i}$ is defined
in Remark \ref{rem2.4}. Moreover, we provide a convergence proof for the proposed algorithm.

\textbullet\ While the HV-$\left(\alpha\ell_{1}^{2}-\beta\ell^{2}_{2}\right)$ algorithm yields superior results, it necessitates considerable computational time to solve the multivariate linear equation at each iteration due to the parameter $\lambda_{i}$. To circumvent this computational bottleneck, we substitute the operator $\mathbb{H}_{\alpha}$ with the projection operator $\mathbb{P}_{R}$, thereby proposing the PG-$\left(\alpha\ell_{1}^{2}-\beta\ell_{2}^{2}\right)$ algorithm based on the surrogate function method formulated as
\begin{equation*}
     x^{k+1}=\mathbb{P}_{R}\left(\frac{\gamma}{\gamma-2\beta}x^{k} -\frac{1}{\gamma-2\beta}A^{*}\left(Ax^{k}-y^{\delta}\right)\right)
\end{equation*}
where $\gamma>2\beta$ adheres to specific conditions. In contrast to the implicit fixed point projection algorithms PG-GCGM and PG-SF as discussed in \cite{DH20}, the PG-$\left(\alpha\ell_{1}^{2}-\beta\ell_{2}^{2}\right)$ algorithm is characterized by its fully explicit nature. Morozov's discrepancy principle determines the constraint radius $R$. Furthermore, we substantiate the convergence.

\textbullet\ The experimental results consistently highlight the competitiveness of our proposed method compared to various traditional iterative algorithms. The HV-$\left(\ell_{1}^{2}-\eta\ell^{2}_{2}\right)$ algorithm achieves superior reconstruction outcomes compared to ISTA, FISTA, and ST-$\left(\alpha\ell_{1}-\beta\ell_{2}\right)$ algorithm. Furthermore, the PG-$\left(\ell_{1}^{2}-\eta\ell^{2}_{2}\right)$ algorithm demonstrates enhanced reconstruction performance while significantly reducing solution time. Overall, the experimental findings validate the efficacy of the proposed algorithms in addressing sparse inverse problems.

\par The organization of this paper is as follows: Section \ref{sec2} proves the coercivity, weak lower semi-continuity and Radon-Riesz property of the regularization term. Section \ref{sec3} addresses the well-posedness and convergence rate results of regularized solution correlated with the $\ell_{1}^{2}-\eta\ell^{2}_{2}$ regularization in $\ell_{2}$ space. Section \ref{sec4} introduces a novel half-variation iterative algorithm inspired by the proximal gradient method. Section \ref{sec5} presents a PG-$\left(\ell_{1}^{2}-\eta\ell_{2}^{2}\right)$ algorithm that integrates the projected gradient and surrogate function methodologies. Lastly, Section \ref{sec6} discusses the outcomes of our numerical experiments.

\section{Preliminaries}\label{sec2}

\par We define a minimizer of the regularization functional $\mathcal{J}_{\alpha,\beta}^{\delta}(x)$ in \eqref{equ1.4} as
\begin{equation}\label{equ2.1}
    x_{\alpha,\beta}^{\delta}:=\arg\min_{x\in\ell_{2}}\left\{\mathcal{J}_{\alpha,\beta}^{\delta}(x)= \frac{1}{q}\left\|Ax-y^{\delta}\right\|_{Y}^{q}+\mathcal{R}_{\alpha,\beta}(x)\right\}
\end{equation}
for every $\alpha\geq\beta> 0$. We adopt the following definition of the $\mathcal{R}_{\eta}$-minimum solution \cite{EHN96}.

\begin{definition}\label{def2.1}
    An element $x^{\dagger}\in\ell_{2}$ is referred to as an $\mathcal{R}_{\eta}$-minimum solution for the linear problem $Ax=y$ if
    \begin{equation*}
        Ax^{\dagger}=y\ and\ \mathcal{R}_{\eta}\left(x^{\dagger}\right)=\min_{x\in\ell_{2}}\left\{\mathcal{R}_{\eta}(x)\ \mid\ Ax=y \right\}.
    \end{equation*}
\end{definition}

\par We will also revisit the definition of sparsity \cite{DDD04}.

\begin{definition}\label{def2.2}
     A vector $x\in\ell_{2}$ is termed sparse if ${\rm supp}(x):=\left\{i\in\mathbb{N} \mid x_{i}\neq 0\right\}$ is finite, where $x_{i}$ denotes the $i{\rm th}$ component of $x$. The notation $\left\|x\right\|_{\ell_{0}}:={\rm supp}(x)$ represents the cardinality of ${\rm supp}(x)$. If $\left\|x\right\|_{\ell_{0}}=s$ for some $s\in\mathcal{N}$, we classify $x\in\ell_{2}$ as $s$-sparse.
\end{definition}

\begin{definition}\label{def2.3}
    We define
    \begin{equation*}
        \mathcal{I}(x^{\dagger})=\left\{i\in\mathbb{N}\ \mid\ x_{i}^{\dagger}\neq 0\right\},
    \end{equation*}
    where $x_{i}^{\dagger}$ is the $i{\rm th}$ component of the $\mathcal{R}_{\eta}$-minimum solution $x^{\dagger}$ of the linear problem $Ax=y$.
\end{definition}

\begin{remark}\label{rem2.4}
    If $x^{\dagger}$ is sparse, i.e., $\mathcal{I}(x^{\dagger})$ is finite, then there exists a constant $m>0$ such that
    \begin{equation*}
        \min_{i\in\mathcal{I}(x^{\dagger})}\left|x_{i}^{\dagger}\right|=m,
    \end{equation*}
    where $e_{i}=(\underbrace{0,\cdots,0,1}_{i},0,\cdots)$.
\end{remark}

\begin{lemma}\label{lem2.5}
    \textbf{(Coercivity)} Assume $\alpha>\beta\geq 0$ and $\eta=\beta/\alpha$. The functional $\mathcal{R}_{\eta}:\ell_{2}\rightarrow \left[0, +\infty\right]$ is coercive, meaning that if $\left\|x\right\|_{\ell_{2}}\rightarrow +\infty$, then $\mathcal{R}_{\eta}(x)\rightarrow +\infty$.
\end{lemma}

\begin{proof}
    We note that $\left\|x\right\|_{\ell_{1}}\geq\left\|x\right\|_{\ell_{2}}$. Consequently, we have
    \begin{equation*}
        \mathcal{R}_{\eta}\left(x\right)= \left(\left\|x\right\|_{\ell_{1}}^{2}-\left\|x\right\|_{\ell_{2}}^{2}\right)+ \left(1-\eta\right)\left\|x\right\|_{\ell_{2}}^{2} \geq \left(1-\eta\right)\left\|x\right\|_{\ell_{2}}^{2},
    \end{equation*}
    which clearly implies that $\left\|x\right\|_{\ell_{2}}\rightarrow +\infty$ results in $\mathcal{R}_{\eta}(x)\rightarrow +\infty$.
\end{proof}

\par It is important to note that $\mathcal{R}_{\eta}(x)$ is not coercive when $\eta=1$. For instance, let $x=(\underbrace{0,\cdots,0,x_{i}}_{i},0,\cdots)$, in this case, $\left\|x\right\|_{\ell_{2}}\rightarrow +\infty$ as $\left|x_{i}\right|\rightarrow +\infty$. Nevertheless, $\mathcal{R}_{\eta}(x)\equiv 0$ remains constant for any choice of $x_{i}$.

\par Next, we revisit an extension of Fatou's lemma \cite{C00}

\begin{lemma}\label{lem2.6}
    Let $f_{1}, f_{2}, \cdots$ be a sequence of real-valued measurable functions defined on a measure space $(S, \Sigma, \mu)$. If there exists an integrable function $g$ on $S$ such that $f_{n}\geq -g$ for all $n$, then we have
    \begin{equation*}
        \int_{S}\liminf_{n}f_{n}d\mu \leq \liminf_{n}\int_{S}f_{n}d\mu.
    \end{equation*} 
\end{lemma}

\begin{lemma}\label{lem2.7}
    \textbf{(Weak lower semi-continuity)} Let $M>0$ be given. Then, for any sequence $x_{n}\in\ell_{2}$ such that $\mathcal{R}_{\eta}(x_{n})\leq M$, if $\left\{x_{n}\right\}$ weakly converging to $x\in\ell_{2}$, it follows that
    \begin{equation*}
        \liminf_{n}\mathcal{R}_{\eta}(x_{n}) \geq \mathcal{R}_{\eta}(x).
    \end{equation*}
\end{lemma}

\begin{proof}
    By the definition of $\mathcal{R}_{\eta}$ as presented in \eqref{equ1.5}, we obtain
    \begin{align}\label{equ2.2}
        \mathcal{R}_{\eta}(x_{n})-\mathcal{R}_{\eta}(x)
        & = \left(\left\|x_{n}\right\|_{\ell_{1}}^{2}-\left\|x\right\|_{\ell_{1}}^{2}\right)- \eta\left(\left\|x_{n}\right\|_{\ell_{2}}^{2}-\left\|x\right\|_{\ell_{2}}^{2}\right) \nonumber\\
        & =\sum_{i}\left(\left\|x_{n}\right\|_{\ell_{1}}+\left\|x\right\|_{\ell_{1}}\right) \left(\left|x_{n}^{i}\right|-\left|x^{i}\right|\right)- \eta\sum_{i}\left(\left|x_{n}^{i}\right|+ \left|x^{i}\right|\right)\left(\left|x_{n}^{i}\right|-\left|x^{i}\right|\right)  \nonumber\\
        & =\sum_{i}\left[\left(\left\|x_{n}\right\|_{\ell_{1}}+\left\|x\right\|_{\ell_{1}}\right) -\eta\left(\left|x_{n}^{i}\right|+\left|x^{i}\right|\right)\right]\left(\left|x_{n}^{i}\right|-\left|x^{i}\right|\right), 
    \end{align}
    where $x_{n}^{i}$ and $x^{i}$ denote the $i$th components of $x$ and $x_{n}$, respectively. If $x_{n}\neq 0$ or $x\neq 0$, we define $c_{n}^{i}:=1 -\frac{\eta\left(\left|x_{n}^{i}\right|+\left|x^{i}\right|\right)} {\left\|x_{n}\right\|_{\ell_{1}}+\left\|x\right\|_{\ell_{1}}}$, resulting in $0<1-\eta\leq c_{n}^{i}\leq 1$. If $x_{n}=0$ and $x=0$, we let $c_{n}^{i}=0$. From \eqref{equ2.2}, we have
    \begin{equation}\label{equ2.3}
        \liminf_{n}\left[\mathcal{R}_{\eta}(x_{n})-\mathcal{R}_{\eta}(x)\right] =\liminf_{n}\left[\sum_{i}c_{n}^{i}\left(\left|x_{n}^{i}\right|-\left|x^{i}\right|\right) \left(\left\|x_{n}\right\|_{\ell_{1}}+\left\|x\right\|_{\ell_{1}}\right)\right].
    \end{equation}
    By the definition of $c_{n}^{i}$, it follows that
    \begin{equation}\label{equ2.4}
        c_{n}^{i}\left(\left|x_{n}^{i}\right|-\left|x^{i}\right|\right)\geq -c_{n}^{i}\left|x^{i}\right| \geq -\left|x^{i}\right|.
    \end{equation}
    Meanwhile, since $\mathcal{R}_{\eta}(x_{n})\leq M$, we can conclude that $\left\{\left\| x_{n}\right\|_{\ell_{1}}\right\}$ is bounded. Consequently, from $\left\|x\right\|_{\ell_{1}}\leq\liminf_{n}\left\|x_{n}\right\|_{\ell_{1}}$, it follows that $\left\|x\right\|_{\ell_{1}}$ must be finite. Thus, we obtain
    \begin{equation}\label{equ2.5}
        \sum_{i}\left|x^{i}\right|\neq\infty.
    \end{equation}
    Using \eqref{equ2.4} and \eqref{equ2.5}, we can apply Lemma \ref{lem2.6} to derive
    \begin{equation}\label{equ2.6}
        \liminf_{n}\left[\sum_{i}c_{n}^{i}\left(\left|x_{n}^{i}\right|-\left|x^{i}\right|\right)\right] \geq\sum_{i}\liminf_{n}\left(c_{n}^{i}\left|x_{n}^{i}\right|-c_{n}^{i}\left|x^{i}\right|\right).
    \end{equation}
    From the weak convergence of $x_{n}$ to $x$, we can deduce that $\left|x_{n}^{i}\right|\rightarrow \left|x^{i}\right|$ for all $i\in\mathbb{N}$. Given that $0<c_{n}^{i}\leq\alpha$, it is clear that $c_{n}^{i}\left|x_{n}^{i}\right|-c_{n}^{i}\left|x^{i}\right|\rightarrow 0$. Therefore, we have
    \begin{equation*}
        \liminf_{n}\left(c_{n}^{i}\left|x_{n}^{i}\right|-c_{n}^{i}\left|x^{i}\right|\right)=0.
    \end{equation*} 
    Hence,
    \begin{equation}\label{equ2.7}
        \sum_{i}\liminf_{n} \left[\left(c_{n}^{i}\left|x_{n}^{i}\right|-c_{n}^{i}\left|x^{i}\right|\right) \left(\left\|x_{n}\right\|_{\ell_{1}}+\left\|x\right\|_{\ell_{1}}\right)\right]=0.
    \end{equation}
    Combining \eqref{equ2.3}, \eqref{equ2.6}, and \eqref{equ2.7}, we conclude that
    \begin{equation*}
        \liminf_{n}\left(\mathcal{R}_{\eta}(x_{n})-\mathcal{R}_{\eta}(x)\right)\geq 0,
    \end{equation*}
    which completes the proof of the lemma.
\end{proof}

\begin{remark}\label{rem2.8}
    It is important to note that Lemma \ref{lem2.7} continues to hold when $\alpha=\beta$, specifically, the expression $\left\|\cdot\right\|_{\ell_{1}}^{2}-\left\|\cdot\right\|_{\ell_{2}}^{2}$ remains weakly lower semi-continuous. In the case where $\alpha=\beta$ and $0\leq c_{n}^{i}\leq\alpha$, the above proof remains valid.
\end{remark}

\begin{lemma}\label{lem2.9}
    \textbf{(Radon-Riesz property)} Let $M>0$ be specified. Then, for any sequence $x_{n}\in\ell_{2}$ satisfying $\mathcal{R}_{\eta}\left(x_{n}\right)\leq M$, if $x_{n}$ converges weakly to $x\in\ell_{2}$ and $\mathcal{R}_{\eta}\left(x_{n}\right)\rightarrow \mathcal{R}_{\eta}\left(x\right)$, it follows that $x_{n}$ converges strongly to $x\in\ell_{2}$.  
\end{lemma}

\begin{proof}
    Given the assumption that $\mathcal{R}_{\eta}(x_{n})\rightarrow \mathcal{R}_{\eta}(x)$, we can write
    \begin{equation*}
        \left\|x_{n}\right\|_{\ell_{1}}^{2}-\eta\left\|x_{n}\right\|_{\ell_{2}}^{2} \rightarrow\left\|x\right\|_{\ell_{1}}^{2}-\eta\left\|x\right\|_{\ell_{2}}^{2}.
    \end{equation*}
    This implies
    \begin{equation}\label{equ2.8}
        \left(\left\|x_{n}\right\|_{\ell_{1}}^{2}-\left\|x_{n}\right\|_{\ell_{2}}^{2} \right)+\left(1-\eta\right)\left\|x_{n}\right\|_{\ell_{2}}^{2} \rightarrow \left(\left\|x\right\|_{\ell_{1}}^{2}-\left\|x\right\|_{\ell_{2}}^{2} \right)+\left(1-\eta\right)\left\|x\right\|_{\ell_{2}}^{2}.
    \end{equation}
    \par Next, we will demonstrate that $\left\|x_{n}\right\|_{\ell_{2}}\rightarrow \left\|x\right\|_{\ell_{2}}$ by means of contradiction. Assume, for the sake of contradiction, that $\left\|x_{n}\right\|_{\ell_{2}}\nrightarrow \left\|x\right\|_{\ell_{2}}$. Since $x_{n}\rightharpoonup x$ in $\ell_{2}$, it follows that $\left\|x\right\|_{\ell_{2}}\leq \liminf_{n}\left\|x_{n}\right\|_{\ell_{2}}$, due to the weak lower semi-continuity of the norm. Consequently, there exists a constant $c>0$ such that $c=\liminf_{n}\left\|x_{n}\right\|_{\ell_{2}}>\left\|x\right\|_{\ell_{2}}$. Therefore, we can find a subsequence $\left\{x_{m}\right\}$ from $\left\{x_{n}\right\}$ such that
    \begin{equation*}
        \lim_{m}\left\|x_{m}\right\|_{\ell_{2}}=c>\left\|x\right\|_{\ell_{2}}.
    \end{equation*}
    This leads to
    \begin{equation}\label{equ2.9}
        \lim_{m}\left(1-\eta\right)\left\|x_{m}\right\|_{\ell_{2}}^{2} = c^{2}\left(1-\eta\right)>\left(1-\eta\right)\left\|x\right\|_{\ell_{2}}^{2}.
    \end{equation}
    From \eqref{equ2.8}, we can infer that
    \begin{equation}\label{equ2.10}
        \left(\left\|x_{m}\right\|_{\ell_{1}}^{2}-\left\|x_{m}\right\|_{\ell_{2}}^{2} \right)+\left(1-\eta\right)\left\|x_{m}\right\|_{\ell_{2}}^{2} \rightarrow \left(\left\|x\right\|_{\ell_{1}}^{2}-\left\|x\right\|_{\ell_{2}}^{2} \right)+\left(1-\eta\right)\left\|x\right\|_{\ell_{2}}^{2}.
    \end{equation}
    Combining \eqref{equ2.9} and \eqref{equ2.10} yields
    \begin{equation*}
        \lim_{m}\left(\left\|x_{m}\right\|_{\ell_{1}}^{2}-\left\|x_{m}\right\|_{\ell_{2}}^{2} \right)<\left\|x\right\|_{\ell_{1}}^{2}-\left\|x\right\|_{\ell_{2}}^{2}.
    \end{equation*}
    Thus, we have
    \begin{equation}\label{equ2.11}
        \liminf_{n}\left(\left\|x_{n}\right\|_{\ell_{1}}^{2}-\left\|x_{n}\right\|_{\ell_{2}}^{2} \right) \leq
        \liminf_{m}\left(\left\|x_{m}\right\|_{\ell_{1}}^{2}-\left\|x_{m}\right\|_{\ell_{2}}^{2} \right)<\left\|x\right\|_{\ell_{1}}^{2}-\left\|x\right\|_{\ell_{2}}^{2}.
    \end{equation}
    This contradicts the established fact that $\left\|\cdot\right\|_{\ell_{1}}^{2}- \left\|\cdot\right\|_{\ell_{2}}^{2}$ is weakly lower semi-continuous (refer to Remark \ref{rem2.8}). This argument demonstrates that $\left\|x_{n}\right\|_{\ell_{2}} \rightarrow\left\|x\right\|_{\ell_{2}}$. Since $x_{n}\rightharpoonup x$ in $\ell_{2}$ by assumption, we conclude that $x_{n}\rightarrow x$ in $\ell_{2}$.
\end{proof}

\begin{remark}\label{rem2.10}
     Let $x_{n}=(\underbrace{0,\cdots,0,1}_{n},0,\cdots)$ and $x=0$. Then $x_{n}\rightharpoonup x$ in $\ell_{2}$. If $\eta=1$, we have
    \begin{equation*}
        \mathcal{R}_{\eta}\left(x_{n}\right)=\left\|x_{n}\right\|_{\ell_{1}}^{2}- \left\|x_{n}\right\|_{\ell_{2}}^{2}=0 \quad and \quad \mathcal{R}_{\eta}\left(x\right)=0.
    \end{equation*}
    So $\mathcal{R}_{\eta}\left(x_{n}\right)\rightarrow \mathcal{R}_{\eta}\left(x\right)$. However, $\left\|x_{n}-x\right\|_{\ell_{2}}=1$, which implies that $x_{n}$ does not converge strongly to $x$. We note that the Radon-Riesz property is invalid if $\eta=1$.
\end{remark}

\section{Properties of regularized solution}\label{sec3}

\par In this section, we present a comprehensive analysis of the well-posedness of the regularization method, encompassing aspects such as existence, stability, convergence, and sparsity of solutions. Given that the well-posedness of regularization methods is a classical result established through the conditions of the coercivity, weak lower semi-continuity, and Radon-Riesz property of the penalty term, we will succinctly state the relevant theorems without delving into extensive detail. For a thorough examination of the proofs, we direct the reader to references \cite{DH19,EHN96}.

\begin{theorem}\label{the3.1}
    \textbf{(Existence)} For all $\alpha>\beta> 0$ and $y^{\delta}\in Y$, problem \eqref{equ1.4} has a solution.
\end{theorem}

\begin{theorem}\label{the3.2}
    \textbf{(Stability)} Let $\alpha>\beta> 0$, and let $\left\{y_{n}\right\}$ be a sequence with $\lim_{n\rightarrow +\infty}\left\|y_{n}-y^{\delta}\right\|=0$. Furthermore, let $\left\{x_{n}\right\}$ be a sequence such that $x_{n}$ is a minimizer of $\mathcal{J}_{\alpha_{n},\beta_{n}}^{\delta_{n}}\left(x\right)$, with $\alpha_{n}>\beta_{n}\geq 0$ and $\alpha_{n}\rightarrow \alpha$, $\beta_{n}\rightarrow \beta$ as $n\rightarrow +\infty$. Then, the sequence $\left\{x_{n}\right\}$ contains a convergent subsequence $\left\{x_{n_{k}}\right\}$, and the limit $x_{\alpha,\beta}^{\delta}$ of every convergent subsequence is a minimizer of $\mathcal{J}_{\alpha,\beta}^{\delta}(x)$. If the minimizer of $\mathcal{J}_{\alpha,\beta}^{\delta}(x)$ is unique, then it follows that $\lim_{n\rightarrow +\infty}\left\|x_{n}- x_{\alpha,\beta}^{\delta}\right\|_{\ell_{2}}=0$.
\end{theorem}

\begin{theorem}\label{the3.3}
    \textbf{(Convergence)} Let $x_{\alpha_{n},\beta_{n}}^{\delta_{n}}$ be a minimizer of $\mathcal{J}_{\alpha_{n},\beta_{n}}^{\delta_{n}}\left(x\right)$ defined by \eqref{equ2.1} with the data $y^{\delta_{n}}$ satisfying $\left\|y-y^{\delta_{n}}\right\|\leq\delta_{n}$, where $\delta_{n}\rightarrow 0$ as $n\rightarrow +\infty$ and $y^{\delta_{n}}$ resides in the range of $A$. Assume $\alpha_{n}:=\alpha\left(\delta_{n}\right)$, $\beta_{n}:=\beta\left(\delta_{n}\right)$, where $\alpha_{n}>\beta_{n}> 0$, and the following conditions hold
    \begin{equation*}
        \lim_{n\rightarrow +\infty}\alpha_{n}=0,\ \lim_{n\rightarrow +\infty}\beta_{n}=0\ and\ \lim_{n\rightarrow +\infty}\frac{\delta_{n}^{p}}{\alpha_{n}}=0.
    \end{equation*}
    Additionally, assume $\eta=\lim_{n\rightarrow +\infty}\eta_{n}\in\left[0,1\right)$ exists, where $\eta_{n}=\beta_{n}/\alpha_{n}$. Then there exists a subsequence of $\left\{x_{\alpha_{n},\beta_{n}}^{\delta_{n}}\right\}$, denoted by $\left\{x_{\alpha_{n_{k}},\beta_{n_{k}}}^{\delta_{n_{k}}}\right\}$, such that $x_{\alpha_{n},\beta_{n}}^{\delta_{n}}$ converges to an $\mathcal{R}_{\eta}$-minimizing solution $x^{\dagger}\in\ell_{2}$. Furthermore, if the $\mathcal{R}_{\eta}$-minimizing solution $x^{\dagger}$ is unique, then
    \begin{equation*}
        \lim_{n\rightarrow +\infty}\left\|x_{\alpha_{n},\beta_{n}}^{\delta_{n}}- x^{\dagger}\right\|_{\ell_{2}}=0.
    \end{equation*}
\end{theorem}

\begin{remark}   
    Note that we require $\alpha>\beta$ in the above theorems. In instances where $\alpha=\beta$, we need to require $\mathcal{J}_{\alpha,\beta}\left(x\right)$ is coercivity. Additionally, in the absence of the Radon-Riesz property, we can only establish weak convergence.
\end{remark}

\begin{proposition}\label{pro2.13}
    \textbf{(Sparsity)} Every minimizer $x$ of $\mathcal{J}_{\alpha,\beta}^{\delta}$ in \eqref{equ1.4} is sparse when $q=2$.
\end{proposition}

\begin{proof}
    The proof parallels the approach taken in Proposition 4.5 of \cite{G10}. We define the sequence $\overline{x}:=x-x_{i}e_{i}$ for $i\in\mathbb{N}$, where $e_{i}=(\underbrace{0,\cdots,0,1}_{i},0,\cdots)$ and $x_{i}$ is the $i$th component of $x$. Due to the minimizing property of $x$, we have
    \begin{equation}\label{equ3.1}
        \frac{1}{2}\left\|Ax-y^{\delta}\right\|_{Y}^{2}+\mathcal{R}_{\alpha,\beta}\left(x\right) \leq \frac{1}{2}\left\|A\left(x-x_{i}e_{i}\right)-y^{\delta}\right\|_{Y}^{2}+ \mathcal{R}_{\alpha,\beta}\left(x-x_{i}e_{i}\right).
    \end{equation} 
    If $x=0$, then $x$ is sparse. If $x\neq 0$, from \eqref{equ3.1}, we derive
    \begin{align}\label{equ3.2}
        \left[\alpha\left|x_{i}\right|-\beta\frac{\left|x_{i}\right|^{2}}{\left\|x\right\|_{\ell_{1}} +\left\|\overline{x}\right\|_{\ell_{1}}}\right]\left(\left\|x\right\|_{\ell_{1}} +\left\|\overline{x}\right\|_{\ell_{1}}\right)
        & = \mathcal{R}_{\alpha,\beta}\left(x\right)-\mathcal{R}_{\alpha,\beta}\left(\overline{x}\right) \nonumber\\
        & \leq \frac{1}{2}x_{i}^{2}\left\|Ae_{i}\right\|_{Y}^{2}-x_{i}\left<Ae_{i},Ax-y^{\delta}\right> \nonumber\\
        & \leq \frac{1}{2}x_{i}^{2}\left\|A\right\|_{Y}^{2}-x_{i}\left<e_{i},A^{*}\left(Ax-y^{\delta} \right)\right>
    \end{align}
    for every $i\in\mathbb{N}$. Meanwhile, for any constant $0<c\leq 1-\frac{\beta}{\alpha}$, we can show
    \begin{equation}\label{equ3.3}
        \alpha c\frac{\left|x_{i}\right|}{1+\left|x_{i}\right|}\leq \alpha\left|x_{i}\right|-\beta\left|x_{i}\right|\leq \alpha\left|x_{i}\right|-\beta\frac{\left|x_{i}\right|^{2}}{\left\|x\right\|_{\ell_{1}} +\left\|\overline{x}\right\|_{\ell_{1}}}.
    \end{equation}
    Denote
    \begin{equation*}
        K_{i}:=\frac{\left(1+\left\|x\right\|_{\ell_{1}}\right)\left( \frac{1}{2}x_{i}\left\|A\right\|_{Y}^{2}-\left<e_{i},A^{*}\left(Ax-y^{\delta} \right)\right>\right)}{\alpha c\left(\left\|x\right\|_{\ell_{1}}+ \left\|\overline{x}\right\|_{\ell_{1}}\right)}.
    \end{equation*}
    Combining \eqref{equ3.2} and \eqref{equ3.3} leads us to
    \begin{equation*}
        K_{i}x_{i}\geq \left|x_{i}\right|,\ i\in\mathbb{N}.
    \end{equation*}
    Since $x\in\ell_{2}$, it follows that $x_{i}\rightarrow 0$ as $i\rightarrow \infty$. Additionally, $\left\|A\right\|$ is finite due to $A$ being linear and bounded. Notably, $\left<e_{i}, A^{*}\left(Ax-y^{\delta}\right)\right>= \left(A^{*}\left(Ax-y^{\delta}\right)\right)_{i}$, where $\left(A^{*}\left(Ax-y^{\delta}\right)\right)_{i}$ is the $i$th component of $A^{*}\left(Ax-y^{\delta}\right)$. Since $A^{*}\left(Ax-y^{\delta}\right)\in\ell_{2}$, we have  $\left(A^{*}\left(Ax-y^{\delta}\right)\right)_{i}\rightarrow 0$ as $i\rightarrow \infty$. Thus, $K_{i}\rightarrow 0$ as $i\rightarrow \infty$, which implies that $\Lambda:=\left\{i\in\mathbb{N}\ \mid\ \left|K_{i}\right|\geq 1\right\}$ is finite.\ Consequently, $x_{i}=0$ for $i\notin\Lambda$, proving the sparsity of $x$.
\end{proof}

\par Subsequently, we present the results regarding the convergence rates of both a priori and a posteriori parameter choice rules. We derive an inequality under a source condition, which leads us to the convergence rate $\mathcal{O}\left(\delta\right)$ in the $\ell_{2}$-norm based on this inequality. The source condition is outlined below. 

\begin{assumption}\label{ass3.5}
    Let $x^{\dagger}\neq 0$ be a sparse $\mathcal{R}_{\eta}$-minimizing solution to the problem $Ax=y$. We assume that
    \begin{equation*}
        e_{i}\in\mathcal{R}\left(A^{*}\right)\ \forall i\in\mathcal{I}(x^{\dagger}),
    \end{equation*}
    where $e_{i}=(\underbrace{0,\cdots,0,1}_{i},0,\cdots)$ and $\mathcal{I}\left(x^{\dagger}\right)$ is defined in Definition \ref{def2.3}. In other words, for each $i\in\mathcal{I}\left(x^{\dagger}\right)$, there exists an element $\omega_{i}\in\mathcal{D}\left(A^{*}\right)$ such that $e_{i}=A^{*}\omega_{i}$.
\end{assumption}
\par Assumption \ref{ass3.5} and its modified version were introduced in \cite{BL09,G09}. This assumption acts as a source condition and implies that the operator $A$ satisfies a form of the "finite basis injectivity condition", which is frequently utilized in sparsity regularization.
\par Next, we will present an inequality that follows from this source condition. The linear convergence rate $\mathcal{O}\left(\delta\right)$ can be derived from this inequality.

\begin{lemma}\label{lem3.6}
    Assume that Assumption \ref{ass3.5} holds and that $\mathcal{R}_{\alpha,\beta}\left(x\right)\leq M$ $\left(\alpha>\beta>0\right)$for some constant $M>0$. Then there exist constants $c_{1}>c_{2}$ with $c_{1}>0$ and $c_{3}>0$ such that
    \begin{equation}\label{equ3.4}
        c_{3}\left(\alpha-\beta\right)\left\|x-x^{\dagger}\right\|_{\ell_{1}} \leq\mathcal{R}_{\alpha,\beta}\left(x\right)-\mathcal{R}_{\alpha,\beta}\left(x^{\dagger}\right) +\left(c_{1}\alpha-c_{2}\beta\right)\|Ax-Ax^{\dagger}\|_{Y}
    \end{equation}
\end{lemma}

\begin{proof}
    From the definition of index set $\mathcal{I}\left(x^{\dagger}\right)$, let $\gamma\left(x\right):= \left\|x \right\|_{\ell_{1}}+\left\|x^{\dagger}\right\|_{\ell_{1}}$ and we have
    \begin{equation*}
        \left(\alpha-\beta\right)\left\|x-x^{\dagger}\right\|_{\ell_{1}}\left(\left\|x \right\|_{\ell_{1}}+\left\|x^{\dagger}\right\|_{\ell_{1}}\right)
        =\left(\alpha-\beta\right)\gamma\left(x\right)\left(\sum_{i\in\mathcal{I}\left(x^{\dagger}\right)} \left|x_{i}-x_{i}^{\dagger}\right|^{2}+\sum_{i\notin\mathcal{I}\left(x^{\dagger} \right)}\left|x_{i}\right|^{2}\right).
    \end{equation*}
    Consequently,
    \begin{align}\label{equ3.5}
        &  \quad \left(\alpha-\beta\right)\gamma\left(x\right)\left\|x-x^{\dagger}\right\|_{\ell_{1}}- \left(\mathcal{R}_{\alpha,\beta}\left(x\right)-\mathcal{R}_{\alpha,\beta}\left(x^{\dagger}\right)\right) \nonumber\\
        &  = -\alpha\gamma\left(x\right)\left(\sum_{i\in\mathcal{I}\left(x^{\dagger}\right)}\left|x_{i}\right| -\sum_{i\in\mathcal{I}\left(x^{\dagger}\right)}\left|x_{i}^{\dagger}\right|\right) +\left(\alpha-\beta\right)\gamma\left(x\right)\left(\sum_{i\in\mathcal{I}\left(x^{\dagger}\right)} \left|x_{i}- x_{i}^{\dagger}\right|\right) +\beta\left(T_{1}-T_{2}\right),
    \end{align}
    where
    \begin{align*}
        & T_{1}= \sum_{i}\left|x_{i}\right|^{2} -\sum_{i\notin\mathcal{I}\left(x^{\dagger}\right)}\left|x_{i}\right|^{2} -\sum_{i\in\mathcal{I}\left(x^{\dagger}\right)}\left|x_{i}^{\dagger}\right|^{2} =\sum_{i\in\mathcal{I}\left(x^{\dagger}\right)}\left|x_{i}\right|^{2}- \sum_{i\in\mathcal{I}\left(x^{\dagger}\right)}\left|x_{i}^{\dagger}\right|^{2}, \\
        & T_{2}= \gamma\left(x\right) \left(\sum_{i\notin\mathcal{I}\left(x^{\dagger}\right)}\left|x_{i}\right|\right) -\sum_{i\notin\mathcal{I}\left(x^{\dagger}\right)}\left|x_{i}\right|^{2} =\sum_{i\notin\mathcal{I}\left(x^{\dagger}\right)}\left|x_{i}\right| \left(\gamma\left(x\right)-\left|x_{i} \right|\right).
    \end{align*}
    Note that $T_{2}\geq 0$. Thus, from \eqref{equ3.5}, we derive
    \begin{align}\label{equ3.6}
        \left(\alpha-\beta\right)\gamma\left(x\right)\left\|x-x^{\dagger}\right\|_{\ell_{1}} 
        & \leq \left(\mathcal{R}_{\alpha,\beta}\left(x\right) -\mathcal{R}_{\alpha,\beta}\left(x^{\dagger}\right)\right) +\alpha\gamma\left(x\right)\left(\sum_{i\in\mathcal{I}\left(x^{\dagger}\right)}\left|x_{i}-x_{i}^{\dagger}\right| \right) \nonumber\\
        & \quad +\left(\alpha-\beta\right)\gamma\left(x\right)\left(\sum_{i\in\mathcal{I}\left(x^{\dagger}\right)} \left|x_{i}- x_{i}^{\dagger}\right|\right)+\beta T_{1}.
    \end{align}
    Let $m_{1}$ be a constant upper bound for the terms of the form $\left|x_{i}\right|+ \left|x_{i}^{\dagger}\right|$, and let $m_{2}=\sum_{i}\left|x_{i}^{\dagger}\right|$. Then $0<m_{2}\leq\gamma\left(x\right)$, and we can express
    \begin{equation}\label{equ3.7}
        T_{1}=\sum_{i\in\mathcal{I}\left(x^{\dagger}\right)}\left(\left|x_{i}\right|+ \left|x_{i}^{\dagger}\right|\right)\left(\left|x_{i}\right|- \left|x_{i}^{\dagger}\right|\right) \leq m_{1}\sum_{i\in\mathcal{I}\left(x^{\dagger}\right)}\left|x_{i}- x_{i}^{\dagger}\right|
    \end{equation}
    A combination of \eqref{equ3.6} an \eqref{equ3.7} leads to the following result
    \begin{align}\label{equ3.8}
        \left(\alpha-\beta\right)\left\|x-x^{\dagger}\right\|_{\ell_{1}}
        & \leq \frac{\left(\mathcal{R}_{\alpha,\beta}\left(x\right) -\mathcal{R}_{\alpha,\beta}\left(x^{\dagger}\right)\right)}{\gamma\left(x\right)}+ \left[2\alpha-\left(1-\frac{m_{1}}{\gamma\left(x\right)}\right)\beta\right]\left(\sum_{i\in\mathcal{I} \left(x^{\dagger}\right)}\left|x_{i}-x_{i}^{\dagger}\right|\right) \nonumber\\
        & \leq \frac{\left(\mathcal{R}_{\alpha,\beta}\left(x\right) -\mathcal{R}_{\alpha,\beta}\left(x^{\dagger}\right)\right)}{m_{2}}+ \left[2\alpha-\left(1-\frac{m_{1}}{m_{2}}\right)\beta\right]\left(\sum_{i\in\mathcal{I} \left(x^{\dagger}\right)}\left|x_{i}-x_{i}^{\dagger}\right|\right).
    \end{align}
    Furthermore, based on Assumption \ref{ass3.5}, we have
    \begin{equation*}
        \left|x_{i}-x_{i}^{\dagger}\right|=\left|\left<e_{i},x-x^{\dagger}\right>\right| =\left|\left<\omega_{i},Ax-Ax^{\dagger}\right>\right| \leq \max_{i\in\mathcal{I}\left(x^{\dagger}\right)}\left\|\omega_{i}\right\|_{Y}\left\|Ax-Ax^{\dagger}\right\|_{Y}
    \end{equation*}
    for all $i\in\mathcal{I}\left(x^{\dagger}\right)$. Hence, we obtain
    \begin{equation}\label{equ3.9}
        \sum_{i\in\mathcal{I}\left(x^{\dagger}\right)}\left|x_{i}-x_{i}^{\dagger}\right| \leq \left|\mathcal{I}\left(x^{\dagger}\right)\right| \max_{i\in\mathcal{I}\left(x^{\dagger}\right)}\left\|\omega_{i}\right\|_{Y}\left\|Ax-Ax^{\dagger}\right\|_{Y},
    \end{equation}
    where $\left|\mathcal{I}\left(x^{\dagger}\right)\right|$ denotes the size of the index set $\mathcal{I}\left(x^{\dagger}\right)$. Combining equations \eqref{equ3.8} and\eqref{equ3.9} leads us to the conclusion that
    \begin{align*}
        \left(\alpha-\beta\right)\left\|x-x^{\dagger}\right\|_{\ell_{1}} 
        & \leq \frac{\left(\mathcal{R}_{\alpha,\beta}\left(x\right) -\mathcal{R}_{\alpha,\beta}\left(x^{\dagger}\right)\right)}{m_{2}} \\
        & \quad + \left[2\alpha-\left(1-\frac{m_{1}}{m_{2}}\right)\beta\right]\left|\mathcal{I}\left(x^{\dagger}\right)\right| \max_{i\in\mathcal{I}\left(x^{\dagger}\right)}\left\|\omega_{i}\right\|_{Y}\left\|Ax-Ax^{\dagger}\right\|_{Y},
    \end{align*}
    i.e.,
    \begin{equation*}
        c_{3}\left(\alpha-\beta\right)\left\|x-x^{\dagger}\right\|_{\ell_{1}}\leq \left(\mathcal{R}_{\alpha,\beta}\left(x\right) -\mathcal{R}_{\alpha,\beta}\left(x^{\dagger}\right)\right) + \left(c_{1}\alpha-c_{2}\beta\right)\left\|Ax-Ax^{\dagger}\right\|_{Y},
    \end{equation*}
    where $c_{1}=2m_{2}\left|\mathcal{I}\left(x^{\dagger}\right)\right| \max_{i\in\mathcal{I}\left(x^{\dagger}\right)}\left\|\omega_{i}\right\|_{Y}$, $c_{2}=\left(m_{2}-m_{1}\right)\left|\mathcal{I}\left(x^{\dagger}\right)\right| \max_{i\in\mathcal{I}\left(x^{\dagger}\right)}\left\|\omega_{i}\right\|_{Y}$, and $c_{3}=m_{2}$ with the condition $c_{1}\alpha-c_{2}\beta>0$.
\end{proof}

\begin{theorem}\label{the3.7}
    Under Assumption \ref{ass3.5}, let $x_{\alpha,\beta}^{\delta}$ be defined as in \eqref{equ2.1} with $\alpha>\beta>0$, and the constants $c_{1}>c_{2}$ be as stated in Lemma \ref{lem3.6}. \\
    Case 1, if $q=1$ and $1-\left(c_{1}\alpha-c_{2}\beta\right)>0$, then
    \begin{equation}\label{equ3.10}
        \left\|x_{\alpha,\beta}^{\delta}-x^{\dagger}\right\|_{\ell_{1}}\leq \frac{1+\left(c_{1}\alpha-c_{2}\beta\right)}{c_{3}\left(\alpha-\beta\right)}\delta, \ \  \left\|Ax_{\alpha,\beta}^{\delta}-y^{\delta}\right\|_{Y}\leq \frac{1+\left(c_{1}\alpha-c_{2}\beta\right)}{1-\left(c_{1}\alpha-c_{2}\beta\right)}\delta,
    \end{equation}
    Case 2, if $q>1$, then
    \begin{align}\label{equ3.11}
        & \left\|x_{\alpha,\beta}^{\delta}-x^{\dagger}\right\|_{\ell_{1}}\leq \frac{1}{c_{3}\left(\alpha-\beta\right)}\left[\frac{\delta^{q}}{q}+ \left(c_{1}\alpha-c_{2}\beta\right)\delta+ \frac{\left(q-1\right)2^{\frac{1}{q-1}}\left(c_{1}\alpha-c_{2}\beta\right)^{\frac{q}{q-1}}}{q}\right], \nonumber\\
        & \left\|Ax_{\alpha,\beta}^{\delta}-y^{\delta}\right\|_{Y}\leq q\left[\frac{\delta^{q}}{q}+ \left(c_{1}\alpha-c_{2}\beta\right)\delta+ \frac{\left(q-1\right)2^{\frac{1}{q-1}}\left(c_{1}\alpha-c_{2}\beta\right)^{\frac{q}{q-1}}}{q}\right].
    \end{align}
\end{theorem}

\begin{proof}
    Owing to the minimization property of $x_{\alpha,\beta}^{\delta}$, it follows that
    \begin{align*}
        \frac{1}{q}\left\|Ax_{\alpha,\beta}^{\delta}-y^{\delta}\right\|_{Y}^{q}+ \mathcal{R}_{\alpha,\beta}\left(x_{\alpha,\beta}^{\delta}\right) 
        & \leq \frac{1}{q}\left\|Ax^{\dagger}-y^{\delta}\right\|_{Y}^{q}+ \mathcal{R}_{\alpha,\beta}\left(x^{\dagger}\right) \\
        & \leq \frac{\delta^{q}}{q}+\mathcal{R}_{\alpha,\beta}\left(x^{\dagger}\right).
    \end{align*}
    Consequently, $\mathcal{R}_{\alpha,\beta}\left(x_{\alpha,\beta}^{\delta}\right)$ is bounded. From Lemma \ref{lem3.6}, we can conclude that
    \begin{align}\label{equ3.12}
        \frac{\delta^{q}}{q}
        & \geq \mathcal{R}_{\alpha,\beta}\left(x_{\alpha,\beta}^{\delta}\right)- \mathcal{R}_{\alpha,\beta}\left(x^{\dagger}\right)+ \frac{1}{q}\left\|Ax_{\alpha,\beta}^{\delta}-y^{\delta}\right\|_{Y}^{q}  \nonumber\\
        & \geq  c_{3}\left(\alpha-\beta\right)\left\|x_{\alpha,\beta}^{\delta}-x^{\dagger}\right\|_{\ell_{1}}- \left(c_{1}\alpha-c_{2}\beta\right)\left\|Ax_{\alpha,\beta}^{\delta}-Ax^{\dagger}\right\|_{Y}+ \frac{1}{q}\left\|Ax_{\alpha,\beta}^{\delta}-y^{\delta}\right\|_{Y}^{q}  \nonumber\\
        & \geq c_{3}\left(\alpha-\beta\right)\left\|x_{\alpha,\beta}^{\delta}-x^{\dagger}\right\|_{\ell_{1}}- \left(c_{1}\alpha-c_{2}\beta\right)\left\|Ax_{\alpha,\beta}^{\delta}-y^{\delta}\right\|_{Y}- \left(c_{1}\alpha-c_{2}\beta\right)\delta+ \frac{1}{q}\left\|Ax_{\alpha,\beta}^{\delta}-y^{\delta}\right\|_{Y}^{q}.
    \end{align}
    Thus, if $q=1$ and $1-\left(c_{1}\alpha-c_{2}\beta\right) >0$, then \eqref{equ3.10} holds. For the case $q>1$, we apply Young's inequality $ab \leq \frac{a^{q}}{q}+\frac{b^{q^{*}}}{q^{*}}$ to obtain
    \begin{align}\label{equ3.13}
        \left(c_{1}\alpha-c_{2}\beta\right)\left\|Ax_{\alpha,\beta}^{\delta}-y^{\delta}\right\|_{Y}
        & =2^{\frac{1}{q}}\left(c_{1}\alpha-c_{2}\beta\right)2^{-\frac{1}{q}} \left\|Ax_{\alpha,\beta}^{\delta}-y^{\delta}\right\|_{Y} \nonumber\\
        & \leq \frac{1}{2q}\left\|Ax_{\alpha,\beta}^{\delta}-y^{\delta}\right\|_{Y}^{q}+ \frac{\left(q-1\right)2^{\frac{1}{q-1}}\left(c_{1}\alpha-c_{2}\beta\right)^{\frac{q}{q-1}}}{q}.
    \end{align}
    A combination of \eqref{equ3.12} and \eqref{equ3.13} implies \eqref{equ3.11}.
\end{proof}

\begin{remark}\label{rem3.8}
    \textbf{(A priori estimation)} Assume that $\beta=\eta\alpha$ for a constant $\eta>0$. If $\alpha\sim\delta^{q-1}(q>1)$, then $\left\|x_{\alpha,\beta}^{\delta}-x^{\dagger}\right\|_{\ell_{1}}\leq c\delta$ for some constant $c>0$. Moreover, it can also be established that $\left\|x_{\alpha,\beta}^{\delta}-x^{\dagger}\right\|_{\ell_{2}}\leq c\delta$.
\end{remark}

\par We will now present a convergence rate result based on the discrepancy principle.

\begin{theorem}\label{the3.9}
    \textbf{(Discrepancy principle)} Under the assumption of the Lemma \ref{lem3.6}, let $x_{\alpha,\beta}^{\delta}$ be defined by \eqref{equ2.1} with $\alpha>\beta>0$, where the parameters $\alpha$ and $\beta$ $\left(\beta=\eta\alpha\right)$ are determined via the discrepancy principle
    \begin{equation*}
        \delta\leq\left\|Ax_{\alpha,\beta}^{\delta}-y^{\delta}\right\|_{Y}\leq\tau\delta,\ \tau\geq 1.
    \end{equation*}
    Then, we have
    \begin{equation*}
        \left\|x_{\alpha,\beta}^{\delta}-x^{\dagger}\right\|_{\ell_{2}}\leq \frac{\left(c_{1}-c_{2}\eta\right)\left(\tau+1\right)\delta}{1-\eta}.
    \end{equation*}
\end{theorem}
    
\begin{proof}
    By the definition of $x_{\alpha,\beta}^{\delta}$, $\alpha$ and $\beta$, we can derive that 
    \begin{equation}\label{equ3.14}
        \frac{1}{q}\delta^{q}+\mathcal{R}_{\alpha,\beta}\left(x_{\alpha,\beta}^\delta\right) \leq \frac{1}{q}\left\|Ax_{\alpha,\beta}^{\delta}-y^{\delta}\right\|_{Y}^{q}+ \mathcal{R}_{\alpha,\beta}\left(x_{\alpha,\beta}^\delta\right) \leq  \frac{1}{q}\left\|Ax^{\dagger}-y^{\delta}\right\|_{Y}^{q}+ \mathcal{R}_{\alpha,\beta}\left(x^{\dagger}\right).
    \end{equation}
    Consequently, we obtain
    \begin{equation*}
        \mathcal{R}_{\alpha,\beta}\left(x_{\alpha,\beta}^\delta\right)\leq \mathcal{R}_{\alpha,\beta}\left(x^{\dagger}\right) 
    \end{equation*}
    From Lemma \ref{lem3.6}, it follows that 
    \begin{align}\label{equ3.15}
        0\geq \mathcal{R}_{\alpha,\beta}\left(x_{\alpha,\beta}^\delta\right)- \mathcal{R}_{\alpha,\beta}\left(x^{\dagger}\right)
        & \geq  c_{3}\left(\alpha-\beta\right)\left\|x_{\alpha,\beta}^\delta-x^{\dagger}\right\|_{\ell_{1}}- \left(c_{1}\alpha-c_{2}\beta\right)\left\|Ax_{\alpha,\beta}^\delta-Ax^{\dagger}\right\|_{Y}\nonumber\\
        & \geq  c_{3}\left(\alpha-\beta\right)\left\|x_{\alpha,\beta}^\delta-x^{\dagger}\right\|_{\ell_{1}}- \left(c_{1}\alpha-c_{2}\beta\right)\left(\tau+1\right)\delta.
    \end{align}
    Therefore, we deduce that
    \begin{equation*}
        \left\|x_{\alpha,\beta}^{\delta}-x^{\dagger}\right\|_{\ell_{2}}\leq \left\|x_{\alpha,\beta}^{\delta}-x^{\dagger}\right\|_{\ell_{1}}\leq \frac{\left(c_{1}\alpha-c_{2}\beta\right)\left(\tau+1\right)\delta} {c_{3}\left(\alpha-\beta\right)}.
    \end{equation*}
    The theorem is thereby established with $\beta=\eta\alpha$.
\end{proof}

\par Note that the proof of convergence is not established when $\alpha=\beta$.

\section{Proximal gradient method via half variation}\label{sec4}

\par In this section, we introduce an algorithm to tackle problem \eqref{equ1.4} and analyze its convergence properties when $q=2$. We will adapt the proximal gradient method \cite{BA17} and demonstrate that this algorithm can effectively minimize the functional incorporating the non-convex and non-smooth regularization term $\alpha\left\|x\right\|_{\ell_{1}}^{2}-\beta\left\|x\right\|_{\ell_{2}}^{2}$, $\alpha\geq\beta> 0$.

\subsection{Proximal gradient method}\label{sec4.1}

For clarity, we begin with a brief overview of the proximal gradient method. The foundation for this discussion is based on \cite[Chapter 10]{BA17}, which proposes a proximal gradient method for solving minimization problems of the form
\begin{equation}\label{equ4.1}
    \min_{x\in X}\left\{F\left(x\right)= f\left(x\right)+g\left(x\right)\right\},
\end{equation}
where $X$ represents a Hilbert space. \cite[Assumption 10.1]{BA17} requires that the functional $f\left(x\right): X\rightarrow (-\infty,+\infty]$ and $g\left(x\right): X\rightarrow (-\infty,+\infty]$ satisfy the following conditions
\begin{condition}\label{con4.1}
    (a). $g\left(x\right)<+\infty$ is proper closed and convex. \\
    (b). $f\left(x\right)<+\infty$ is proper and closed. ${\rm dom}\left(g\right)\subset {\rm int}\left({\rm dom}\left(f\right)\right)$, and $f\left(x\right)$ has a L-Lipschits continuous gradient over ${\rm int}\left({\rm dom}\left(f\right)\right)$.\\
    (c). The optimal set of problem \eqref{equ4.1} is nonempty and denoted by $X^{*}$. An optimal value of the problem is denoted by $F_{opt}$.
\end{condition}
\par For the solution of problem \eqref{equ4.1}, it is natural to generalize the above idea and to define the next iterate as the minimizer of the sum of the linearization of $f$ around $x^k$, the non-smooth function $g$, and a quadratic prox term:
\begin{equation}\label{equ4.2}
    x^{k+1}=\arg\min_{x\in X}\left\{f\left(x^{k}\right)+\left<\nabla f\left(x^{k}\right), x-x^{k}\right>+g\left(x^{k}\right)+\frac{1}{2t_{k}}\left\|x-x^{k}\right\|^{2}\right\},
\end{equation}
where $t^{k}$ is the stepsize at iteration $k$. After some simple algebraic manipulation and cancellation of constant terms, \eqref{equ4.2} can be rewritten as
\begin{equation*}
    x^{k+1}=\arg\min_{x\in X}\left\{t_{k}g\left(x^{k}\right)+ \frac{1}{2}\left\|x-\left(x^{k}-t_{k}\nabla f\left(x^{k}\right)\right)\right\|^{2}\right\},
\end{equation*}
which by the definition of the proximal operator is the same as
\begin{equation}\label{equ4.3}
    x^{k+1}={\rm prox}_{t_{k}g}\left(x^{k}-t_{k}\nabla f\left(x^{k}\right)\right).
\end{equation}

\par The proximal gradient method with taking the step-sizes $t_{k}=\frac{1}{L_{k}}$ with $L_{k}\in\left(\frac{L}{2}, +\infty\right)$ is initiated in the form of Algorithm \ref{alg1}.  We next recall a convergence result on Algorithm \ref{alg1} in \cite[Theorem 10.15]{BA17}, where $f\left(x\right)$ is non-convex.

\begin{algorithm}
\caption{Proximal gradient method}
\begin{algorithmic}\label{alg1}
\STATE{\textbf{Initialization}: Set $k=0$, $x_{0}\in {\rm int}\left(X\right)$.}
\STATE{\textbf{General step}: for any $k=0,1,2,\cdots$ execute the following steps:}
\STATE{\qquad 1: pick $L_{k}\in\left(\frac{L}{2}, +\infty\right)$ for $L>0$;}
\STATE{\qquad 2: set 
\begin{equation*}
    x^{k+1}={\rm prox}_{t_{k}g}\left(x^{k}-\frac{1}{L_{k}}\nabla f\left(x^{k}\right)\right);
\end{equation*}
\STATE{\qquad 3: check the stopping criterion and return $x^{k+1}$ as a solution.}}
\end{algorithmic}
\end{algorithm}

\begin{theorem}\label{the4.2}
    \textbf{(Convergence of the proximal gradient method)}
    Suppose that Condition \ref{con4.1} holds, and let $\left\{x^{k}\right\}$ be the sequence generated by Algorithm \ref{alg1} for solving problem \eqref{equ4.1} with a constant step-size defined by $L_{k}=\overline{L}\in (\frac{L}{2},\infty)$. Then \\
    (1) the sequence $\left\{F\left(x^{k}\right)\right\}$ is non-increasing. In addition, $F\left(x^{k+1}\right)<F\left(x^{k}\right)$ if and only if $x^{k}$ is not a stationary point of \eqref{equ4.1}; \\
    (2) all limit points of the sequence $\left\{x^{k}\right\}$ are stationary points of problem \eqref{equ4.1}.
\end{theorem}

\subsection{The proximal operator of the $\left\|\cdot\right\|_{\ell_{1}}^{2}$}\label{sec4.2}

To enhance the clarity of the article, we introduce the proximal operator of the squared $\ell_{1}$-norm. We will reference the following lemma, which expresses $\left\|x\right\|_{\ell_{1}}^{2}$ as the optimal value of an optimization problem framed in terms of the
\begin{align}\label{equ4.4}
    \varphi\left(s,t\right)=
    \left\{             
             \begin{array}{ll}
             \frac{s^{2}}{t}, & if \ t>0, \\
             0, & if s=0,t=0, \\
             \infty, & else. 
             \end{array}
    \right.
\end{align}
We recommend referring to \cite[Section 6.8.2]{BA17} for readers seeking detailed proof of the following lemmas.

\begin{lemma}\label{lem4.3}
    \textbf{(Variational representation of $\left\|\cdot\right\|_{\ell_{1}}^{2}$)} For a Hilbert space $X$ and any $x\in X$, the following holds:
    \begin{equation}\label{equ4.5}
        \min_{\lambda}\left\{\sum_{i=1}^{n}\varphi\left(x_{i},\lambda_{i} \right)=\left\|x\right\|_{\ell_{1}}^{2}\right\},
    \end{equation}
    where $\varphi$ is defined in \eqref{equ4.4}. An optimal solution $\overline{\lambda}$ to the minimization problem outlined in \eqref{equ4.5} is given by
    \begin{align*}
        \overline{\lambda}_{i}=
        \left\{             
             \begin{array}{ll}
             \frac{\left|x_{i}\right|}{\left\|x\right\|_{\ell_{1}}}, & if \ x\neq 0, \\
             \frac{1}{n}, & if \ x=0,
             \end{array}
        \right.
        i=1,2,\cdots,n. 
    \end{align*}
\end{lemma}

\begin{lemma}\label{lem4.4}
    \textbf{(Proximal operator of $\left\|\cdot\right\|_{\ell_{1}}^{2}$)} Let $g:X\rightarrow \mathcal{R}$ be given by $g\left(x\right)=\left\|x\right\|_{\ell_{1}}^{2}$, and let $\alpha>0$. Then we have 
    \begin{equation*}
        {\rm prox}_{\alpha g}\left(x\right)=
        \left\{             
             \begin{array}{ll}
             \left(\frac{\lambda_{i}x_{i}}{\lambda_{i}+2\alpha}\right)_{i=1}^{n}, & if \ x\neq 0, \\
             0, & if \ x=0,
             \end{array}
        \right.
    \end{equation*}
    where $\lambda_{i}=\left[\frac{\sqrt{\alpha}\left|x_{i}\right|}{\sqrt{\mu^{*}}}-2\alpha\right]_{+}$ with $\mu^{*}$ being any positive root of the nonincreasing function 
    \begin{equation*}
        \psi\left(\mu\right)=\sum_{i=1}^{n}\left[\frac{\sqrt{\alpha}\left|x_{i}\right|} {\sqrt{\mu}}-2\alpha\right]_{+}-1.
    \end{equation*}
\end{lemma}

\par Note that for $a\in\mathbb{R}$, if $a>0$, $\left[a\right]_{+}=a$; if $a\leq 0$, $\left[a\right]_{+}=0$ in Lemma \ref{lem4.4}. Based on Lemma \ref{lem4.3} and Lemma \ref{lem4.4}, we give the definition of the half variation function in the following.

\begin{definition}\label{def4.5}
    For $x\in\ell_{2}$ and $\alpha>0$, the function $\mathbb{H}_{\alpha}$ is defined as follows
\begin{align}\label{equ4.6}
    \mathbb{H}_{\alpha}(x)=
    \left\{             
             \begin{array}{ll}
             \sum_{i=1}^{n}\frac{\lambda_{i}}{2\alpha+\lambda_{i}}\left<x, e_{i}\right>e_{i}, & if \ x\neq 0, \\
             0, & if \ x= 0. \\
             \end{array}
    \right.
\end{align}
where $e_{i}=(\underbrace{0,\cdots,0,1}_{i},0,\cdots)$ and $\lambda_{i}$ is defined in Lemma \ref{lem4.4}. Thus, $\mathbb{H}_{\alpha}$ is referred to as a half variation function mapping $\ell_{2}$ to $\ell_{2}$.
\end{definition}

\subsection{Proximal gradient method for $\ell_{1}^{2}-\eta\ell_{2}^{2}$ sparsity regularization}\label{sec4.3}

\par Since $\mathcal{R}_{\alpha,\beta}\left(x\right)=\alpha\left\|x\right\|_{\ell_{1}}^{2}- \beta\left\|x\right\|_{\ell_{2}}^{2}$, where $\alpha\geq\beta> 0$, is non-convex, the proximal gradient method cannot be directly applied to problem \eqref{equ1.4}. We can rewrite $\mathcal{J}_{\alpha,\beta}^{\delta}\left(x\right)$ in \eqref{equ1.4} as 
\begin{equation}\label{equ4.7}
    \mathcal{J}_{\alpha,\beta}^{\delta}\left(x\right)=f\left(x\right)+g\left(x\right),
\end{equation}
where $f\left(x\right)=\frac{1}{2}\left\|Ax-y^{\delta}\right\|_{Y}^{2} -\beta\left\|x\right\|_{\ell_{2}}^{2}$ and $g\left(x\right)=\alpha\left\|x\right\|_{\ell_{1}}^{2}$. It is evident that $g\left(x\right)$ is proper and convex. Since $f\left(x\right)$ is continuous, it is also proper and closed. Next, we verify the smoothness of $f\left(x\right)$ and the existence of optimal solution of problem \eqref{equ4.7}.

\begin{definition}\label{def4.6}
    A differentiable function $f\left(x\right)$ is called smooth if it has a Lipschitz continuous gradient, i.e., $\exists L<\infty$ such that
    \begin{equation*}
        \left\|\nabla f\left(x\right)-\nabla f\left(z\right)\right\|_{L\left(X,Y\right)}\leq L\left\|x-z\right\|_{X}
    \end{equation*}
    for $\forall x,z\in X$.
\end{definition}

\begin{lemma}\label{lem4.7}
    $f\left(x\right)$ in \eqref{equ4.7} is $L$-smooth.
\end{lemma}
\begin{proof}
    We note that $f\left(x\right)$ is differentiable and
    \begin{equation*}
        \nabla f\left(x\right)=A^{*}\left(Ax-y^{\delta}\right)-2\beta x.
    \end{equation*}
    By the Definition \ref{def4.6}, for $\forall x,z\in X$, we have
    \begin{align*}
        \left\|\nabla f\left(x\right)-\nabla f\left(z\right)\right\|_{L\left(X,Y\right)}
        & =\left\|\left(A^{*}A-2\beta I\right)\left(x-z\right)\right\|_{X} \nonumber\\
        & \leq C\left\|x-z\right\|_{X},
    \end{align*}
    where $C:=\left\|A^{*}A\right\|+\left\|2\beta I\right\|$. Let $L=C$ and the lemma is proved.
\end{proof}

\par Thus, the minimization problem in \eqref{equ4.2} can be rewritten as
\begin{equation}\label{equ4.8}
    x^{k+1}=\arg\min_{x\in X}\left\{\frac{1}{2}\left\|x-x^{k}+t^{k}\left(A^{*} \left(Ax^{k}-y^{\delta}\right)-2\beta x^{k}\right)\right\|_{\ell_{2}}^{2}+ \alpha\left\|x\right\|_{\ell_{1}}^{2}\right\}.
\end{equation}
By Lemma \ref{lem4.3}, \eqref{equ4.8} is equivalent to the variational form
\begin{equation}\label{equ4.9}
    x^{k+1}=\arg\min_{x\in X}\left\{\frac{1}{2}\left\|x-x^{k}+t^{k}\left(A^{*} \left(Ax^{k}-y^{\delta}\right)-2\beta x^{k}\right)\right\|_{\ell_{2}}^{2}+ \alpha\sum_{i=1}^{n}\varphi\left(x_{i},\lambda_{i}\right)\right\}.
\end{equation}
where $\varphi$ is defined in \eqref{equ4.4} and $\lambda_{i}$ is specified in Lemma \ref{lem4.3}. According to Lemma \ref{lem4.4}, a minimizer of \eqref{equ4.9} can be explicitly computed by
\begin{equation*}
    x^{k+1}={\rm prox}_{t_{k}g}\left(x^{k}-t_{k}\left(A^{*} \left(Ax^{k}-y^{\delta}\right)-2\beta x^{k}\right)\right), 
\end{equation*}
i.e.,
\begin{equation*}
    x^{k+1}=\mathbb{H}_{\alpha}\left(x^{k}-\frac{1}{L_{k}} \left(A^{*}\left(Ax^{k}-y^{\delta}\right)-2\beta x^{k}\right)\right), 
\end{equation*}
where $t^{k}$ is the step-size at iteration $k$ and $L_{k}=\frac{1}{t_{k}}$.

\begin{lemma}\label{lem4.8}
    A minimizer of problem \eqref{equ4.9} is expressed as
    \begin{equation}\label{equ4.10}
        x^{k+1}=\mathbb{H}_{\alpha}\left(x^{k}+\frac{2\beta}{L_{k}} x^{k}-\frac{1}{L_{k}}A^{*}\left(Ax^{k}-y^{\delta}\right)\right).
    \end{equation}
\end{lemma}
\begin{proof}
    The proof closely follows that of \cite[Lemma 6.70]{BA17}. Problem \eqref{equ4.9} can be reformulated as the problem
    \begin{equation}\label{equ4.11}
        x^{k+1}=\arg\min_{x\in X}\left\{\frac{1}{2}\left|x-u\right|^{2}+\alpha\sum_{i}\varphi(x_{i}, \lambda_{i})\right\},
    \end{equation}
    where $u=x^{k}-\frac{1}{L_{k}}\left(A^{*}\left(Ax^{k}-y^{\delta}\right)-2\beta x^{k}\right)$ for a fixed $x^{k}$. Minimizing first with respect to $x$, we obtain that $x_{i}=\frac{\lambda_{i}u_{i}}{\lambda_{i}+2\alpha}$ so that
    \begin{equation*}
        x^{k+1}_{i}=\sum_{i=1}^{n}\frac{\lambda_{i}}{2\alpha+\lambda_{i}}\left[x^{k}+\frac{2\beta}{L_{k}} x^{k}-\frac{1}{L_{k}}A^{*}\left(Ax^{k}-y^{\delta}\right)\right]_{i}.
    \end{equation*}
    and we can obtain \eqref{equ4.10}.
\end{proof}

\par We summarize this strategy, referred to as the half-variation $\ell_{1}^{2}-\eta\ell_{2}^{2}$ $\left({\rm HV}\text{-}(\ell_{1}^{2}-\eta\ell_{2}^{2})\right)$ algorithm, in Algorithm \ref{alg2}. Finally, we address the convergence properties of the HV-($\ell_{1}^{2}-\eta\ell_{2}^{2}$) algorithm for problem \eqref{equ1.4}.

\begin{algorithm}
\caption{HV-($\ell_{1}^{2}-\eta\ell_{2}^{2}$) algorithm for problem \eqref{equ1.4}}
\begin{algorithmic}\label{alg2}
\STATE{\textbf{Initialization}: Set $k=0$, $x_{0}\in {\rm int}\left(X\right)$ such that $g\left(x_{0}\right)<+\infty$.}
\STATE{\textbf{General step}: for any $k=0,1,2,\cdots$ execute the following steps:}
\STATE{\qquad 1: pick $L_{k}\in\left(\frac{L}{2}, +\infty\right)$ for $L>0$;}
\STATE{\qquad 2: set
\begin{equation*}
    x^{k+1}=\mathbb{H}_{\alpha}\left(x^{k}+ \frac{2\beta}{L_{k}} x^{k}- \frac{1}{L_{k}}A^{*}\left(Ax^{k}-y^{\delta}\right)\right);
\end{equation*}
\STATE{\qquad 3: check the stopping criterion and return $x^{k+1}$ as a solution.}}
\end{algorithmic}
\end{algorithm}

\begin{theorem}\label{the4.9}
    Let $\left\{x^{k}\right\}$ denote the sequence generated by Algorithm \ref{alg2}. Then it holds that all limit points of the sequence $\left\{x^{k}\right\}$ are stationary points of the functional $\mathcal{J}_{\alpha,\beta}^{\delta}\left(x\right)$.
\end{theorem}

\begin{proof}
    It is evident that $g\left(x\right)$ is proper and convex, $f\left(x\right)$ is proper and closed. By Lemma \ref{lem4.7}, $f\left(x\right)$ is $L_{f}$-smooth over ${\rm int}\left({\rm dom}\left(f\right)\right)$. By the existence of the solution of $\mathcal{J}_{\alpha,\beta}^{\delta}$, the optimal set of problem \eqref{equ4.1} is nonempty. Therefore, Condition \ref{con4.1} are satisfied from above discussion. By Theorem \ref{the4.2}, we conclude that all limit points of the sequence $\left\{x^{k}\right\}$ are stationary points of the functional $\mathcal{J}_{\alpha,\beta}^{\delta}\left(x\right)$.
\end{proof}

\section{Projected gradient method via surrogate function}\label{sec5}

\par In Section \ref{sec4}, we present an HV-$\left(\ell_{1}^{2}-\eta\ell_{2}^{2}\right)$ algorithm for addressing the challenges associated with $\ell_{1}^{2}-\eta\ell_{2}^{2}$ sparsity regularization. While the formulation of this algorithm is straightforward and lends itself to ease of implementation, it is noteworthy that the  HV-$\left(\ell_{1}^{2}-\eta\ell_{2}^{2}\right)$ algorithm can exhibit considerable computational inefficiency and may be arbitrarily slow. Consequently, the objective of the current section is to leverage the projected gradient method as a means of effectively resolving $\ell_{1}^{2}-\eta\ell_{2}^{2}$ sparsity regularization problems. We propose an accelerated approach to problem \eqref{equ1.5} as an alternative to the method outlined in \eqref{equ4.11}, utilizing the projected gradient method to enhance computational performance.

\subsection{PG algorithm for $\ell_{1}^{2}-\eta\ell_{2}^{2}$ sparsity regularization}\label{sec5.1}

\par In this section, we propose a projected gradient algorithm for problem  \eqref{equ1.4} in $\ell_{2}$ space based on the surrogate function approach. For the case where $q=2$, problem \eqref{equ1.4} can be expressed as
\\
\begin{equation}\label{equ5.1}
    \min\left\{\mathcal{J}_{\alpha,\beta}^{\delta}\left(x\right)= \frac{1}{2}\left\|Ax-y^{\delta}\right\|_{Y}^{2}+ \alpha\left\|x\right\|_{\ell_{1}}^{2}-\beta\left\|x\right\|_{\ell_{2}}^{2}\right\},
\end{equation}
where $A:\ell_{2}\rightarrow Y$ is a linear operator. In what follows, we will eliminate the $\ell_{1}^{2}$ constraint in \eqref{equ5.1} by framing a constrained optimization problem that incorporates an $\ell_{1}$-ball constraint of a certain radius $\mathbf{R}$. We can represent this constrained optimization problem as
\begin{align}\label{equ5.2}
    \left\{             
        \begin{array}{ll}
        \min\left\{\mathcal{D}_{\beta}^{\delta}\left(x\right) =\frac{1}{2} \left\|Ax-y^{\delta}\right\|_{Y}^{2}-\beta\left\|x\right\|_{\ell_{2}}^{2}\right\} ,\ \beta> 0, \vspace{1ex} \\
        {\rm subject\ to}\ \ell_{1}\ {\rm ball}\ B_{\mathbf{R}}:=\left\{x\in\mathbb{R}^{n}\mid\left\|x\right\|_{\ell_{1}}^{2}\leq \mathbf{R}^{2}\right\} \\
        \end{array}
    \right.
\end{align}
for an appropriately chosen $\mathbf{R}$. The following two lemmas will assist in constructing the surrogate function for \eqref{equ5.2}.

\par We will recall definitions of the projection operators (\cite{BA17})

\begin{definition}\label{def5.1}
    The projection onto the $\ell_{1}$-ball is defined by
    \begin{equation*}
        \mathbb{P}_{\mathbf{R}}(\overline{x}):=\left\{{\rm arg}\min_{x\in\ell_{2}}\|x-\overline{x}\|_{\ell_{2}}\ subject\ to\ \|x\|_{\ell_{1}}\leq \mathbf{R}\right\},
    \end{equation*}
    which gives the projection of an element $\overline{x}$ onto $\ell_{1}$-norm ball with radius $\mathbf{R}>0$.
\end{definition}

\par Finally, we recall some properties of $\mathbb{P}_{\mathbf{R}}$ (\cite{DFL08}).

\begin{lemma}\label{lem5.2}
    For any $x\in\ell_{2}$, $\mathbb{P}_{\mathbf{R}}(x)$ is characterized as the unique vector in $B_{\mathbf{R}}$ such that 
    \begin{equation*}
        \left<\omega-\mathbb{P}_{\mathbf{R}}(x),x-\mathbb{P}_{\mathbf{R}}(x)\right>\leq 0
    \end{equation*}
    for all $\omega\in B_{\mathbf{R}}$. Moreover the projection $\mathbb{P}_{\mathbf{R}}$ is non-expansive:
    \begin{equation*}
        \|\mathbb{P}_{\mathbf{R}}(x)-\mathbb{P}_{\mathbf{R}}(x')\|\leq\|x-x'\|
    \end{equation*}
    for all $x,x'\in\ell_{2}$.
\end{lemma}
\par The proof of the Lemma \ref{lem5.2} is a trivial generalization of Lemma 3 in \cite{DFL08}. By the preceding discussion, we focus on the optimization problem denoted as \eqref{equ5.2}. The following result provides a first order optimality condition for the optimization problem \eqref{equ5.2}

\begin{lemma}\label{lem5.3}
    Let $\overline{\omega}\in\ell_{2}$ be a minimizer of the optimization problem \eqref{equ5.2}. Then the following condition holds
    \begin{equation}\label{equ5.3}
        \mathbb{P}_{\mathbf{R}}\left(\overline{\omega}+\frac{2\beta}{\gamma-2\beta}\overline{\omega}- \frac{1}{\gamma-2\beta}A^{*}\left(A\overline{\omega}-y^{\delta}\right)\right) =\overline{\omega}
    \end{equation}
    for any $\gamma>2\beta$. Equivalently, we have
    \begin{equation}\label{equ5.4}
        \left<2\beta\overline{\omega} -A^{*}\left(A\overline{\omega}-y^{\delta}\right), \omega-\overline{\omega}\right>\leq 0 
    \end{equation}
    for any $x\in B_{\mathbf{R}}$.
\end{lemma}

\begin{proof}
    By the definition of $\overline{\omega}$, for any $\omega\in B_{\mathbf{R}}$ and $t\in[0,1]$, we can express the function $f\left(t\right)$ as
    \begin{align*}
        f\left(t\right)
    & =\frac{1}{2}\left\|A\left(\left(1-t\right)\overline{\omega} +t\omega\right)-y^{\delta}\right\|_{Y}^{2}-\beta\left\|\left(1-t\right)\overline{\omega} +t\omega\right\|_{\ell_{2}}^{2} \nonumber\\
    & =\frac{1}{2}\left\|tA\left(\omega-\overline{\omega}\right)+A\overline{\omega} -y^{\delta}\right\|_{Y}^{2}-\beta\left\|t\left(\omega-\overline{\omega}\right) +\overline{\omega} \right\|_{\ell_{2}}^{2} \nonumber\\
    & =\frac{t^{2}}{2}\left\|A\left(\omega-\overline{\omega}\right)\right\|_{Y}^{2} -\beta t^{2}\left\|\omega-\overline{\omega}\right\|_{\ell_{2}}^{2}+ t\left<\omega-\overline{\omega},A^{*}\left(A\overline{\omega}-y^{\delta}\right)\right> \nonumber\\
    & \quad -2\beta t\left<\omega-\overline{\omega}, \overline{\omega}\right> +\frac{1}{2}\left\|A\overline{\omega}-y^{\delta}\right\|_{Y}^{2}- \beta\left\|\overline{\omega}\right\|_{\ell_{2}}^{2}.
    \end{align*} 
    It follows that $f\left(t\right)$ achieves its minimum at $t=0$. Thus, we obtain
    \begin{equation*}
        f'\left(0+\right)=\left<\omega-\overline{\omega},A^{*}\left(A\overline{\omega} -y^{\delta}\right)-2\beta\overline{\omega}\right>\geq 0.
    \end{equation*}
    Consequently, it is evident that both conditions in \eqref{equ5.3} and \eqref{equ5.4} are satisfied, as verified by Lemma \ref{lem5.3}.
\end{proof}

\par Due to the non-convexity of $\mathcal{D}_{\beta}^{\delta}\left(x\right)$, \eqref{equ5.4} is only a necessary condition of the optimization problem \eqref{equ5.2}.

\begin{lemma}\label{lem5.4}
    For a given $\beta\geq 0$ and arbitrary $\gamma>2\beta$, we define
    \begin{equation}\label{equ5.5}
        \mathcal{S}_{\beta,\gamma}(\omega,x):=\frac{1}{2}\left\|A\omega-y^{\delta}\right\|_{Y}^{2}- \beta\left\|\omega\right\|_{\ell_{2}}^{2} +\frac{\gamma}{2}\left\|x-\omega\right\|_{\ell_{2}}^{2} -\frac{1}{2}\left\|Ax-A\omega\right\|_{Y}^{2}
    \end{equation}
    for $\omega, x\in B_{R}$. Then $\mathcal{S}_{\beta,\gamma}(\omega,x)$ is strictly convex. Furthermore, for any fixed $x\in B_{R}$, there exists a unique global minimizer $\overline{\omega}$ of $\mathcal{S}_{\beta,\gamma}(\omega,x)$ on $B_{R}$.
\end{lemma}

\begin{proof}
    For $\omega\in\ell_{2}$, we denote
    \begin{equation*}
        a_{i,j}\left(\omega\right)=\frac{\partial^{2} \left\|\omega\right\|_{\ell_{2}}^{2}} {\partial \omega_{i}\partial \omega_{j}}=
        \left\{             
            \begin{array}{ll}
            2, \quad i=j \\
            0, \quad i\neq j\\
            \end{array}
        \right.
        \quad i,j=1,2,\cdots.
    \end{equation*}
    Since $\omega\rightarrow \left\|\omega\right\|_{\ell_{2}}^{2}$ is convex, the matrix $\left(a_{i,j}\left(\omega\right)\right)$ is a double identity matrix, which is positive definite. By the definition of $\mathcal{S}_{\beta,L}(\omega,x)$, we obtain
    \begin{equation*}
        \frac{\partial^{2} \mathcal{S}_{\beta,L}(\omega,x)} {\partial \omega_{i}\partial \omega_{j}}= \gamma a_{i,j}\left(\omega\right)-2\beta a_{i,j}\left(\omega\right), \quad i,j=1,2,\cdots.
    \end{equation*}
    Given the assumption that $\gamma> 2\beta$, the Hessian matrix $\left(\frac{\partial^{2} \mathcal{S}_{\beta,L}(\omega,x)} {\partial \omega_{i}\partial \omega_{j}}\mid_{\omega=\overline{\omega}}\right)$ is positive definite. This condition implies that the functional  $\mathcal{S}_{\beta,\gamma}(\omega,x)$ is strictly convex. 
    Consequently, there exists a unique global minimizer $\overline{\omega}$ of $\mathcal{S}_{\beta,\gamma}(\omega,x)$ on the set $B_{R}$.
\end{proof}

\begin{lemma}\label{lem5.5}
    Let $\overline{\omega}\in B_{\mathbf{R}}$. Then $\overline{\omega}$ is a minimizer of $\mathcal{S}_{\beta,\gamma}\left(\omega,x\right)$ on $B_{\mathbf{R}}$ if and only if
    \begin{equation}\label{equ5.6}
        \overline{\omega}=\mathbb{P}_{\mathbf{R}}\left(x+\frac{2\beta}{\gamma-2\beta}x -\frac{1}{\gamma-2\beta}A^{*}\left(Ax-y^{\delta}\right)\right).
    \end{equation}
\end{lemma}

\begin{proof}
    By the definition of $\overline{\omega}$, for any $\omega\in B_{\mathbf{R}}$, the function
    \begin{align*}
        f\left(t\right)
        & =\frac{1}{2}\left\|A\left(\left(1-t\right)\overline{\omega}+t\omega\right)- y^{\delta}\right\|_{Y}^{2}- \beta\left\|\left(1-t\right)\overline{\omega}+t\omega\right\|_{\ell_{2}}^{2}\\
        & \quad +\frac{\gamma}{2}\left\|\left(1-t\right)\overline{\omega}+t\omega -x\right\|_{\ell_{2}}^{2}-\frac{1}{2}\left\|A\left(\left(1-t\right)\overline{\omega} +t\omega-x\right)\right\|_{Y}^{2} \\
        & =\frac{\gamma-2\beta}{2}t^{2} \left\|\omega-\overline{\omega}\right\|_{\ell_{2}}^{2}+ t\left<\omega-\overline{\omega},\left(\gamma-2\beta\right)\overline{\omega}-Lx+ A^{*}\left(Ax-y^{\delta}\right)\right> \\
        & \quad +\frac{1}{2}\left\|y^{\delta}\right\|_{Y}^{2}+\frac{\gamma}{2}\left\|x\right\|_{\ell_{2}}^{2} +\frac{1}{2}\left\|Ax\right\|_{Y}^{2}+\left<\overline{\omega}, \frac{\gamma-2\beta}{2}\overline{\omega}-\gamma x+ A^{*}\left(Ax-y^{\delta}\right)\right>
    \end{align*}
    for $t\in [0,1]$ attains its minimum at $t=0$. Thus,
    \begin{equation*}
        f'\left(0+\right)=\left<\omega-\overline{\omega},\left(\gamma-2\beta\right)\overline{\omega}-\gamma x+ A^{*}\left(Ax-y^{\delta}\right)\right>\geq 0,
    \end{equation*}
     which leads to
    \begin{equation*}
        \left(\gamma-2\beta\right)\left<\omega-\overline{\omega}, \frac{\gamma}{\gamma-2\beta}x-\frac{1}{\gamma-2\beta} A^{*}\left(Ax-y^{\delta}\right)-\overline{\omega}\right>\leq 0.
    \end{equation*}
    By Lemma \ref{lem5.2}, this implies \eqref{equ5.6} holds.
    \par Conversely, assume $\overline{\omega}\in B_{R}$ satisfies \eqref{equ5.6}. By Lemma \ref{lem5.2}, we have
    \begin{equation*}
        \left<\omega-\overline{\omega},\frac{\gamma}{\gamma-2\beta}x -\frac{1}{\gamma-2\beta}A^{*}\left(Ax-y^{\delta}\right)-\overline{\omega}\right>\leq 0.
    \end{equation*}
    Define
    \begin{equation*}
        J\left(x\right):= \mathcal{S}_{\beta,\gamma}(\omega,x)= \frac{1}{2}\left\|A\omega-y^{\delta}\right\|_{Y}^{2}- \beta\left\|\omega\right\|_{\ell_{2}}^{2} +\frac{\gamma}{2}\left\|x-\omega\right\|_{\ell_{2}}^{2} -\frac{1}{2}\left\|Ax-A\omega\right\|_{Y}^{2}.
    \end{equation*}
    We find that
    \begin{equation*}
        J'\left(\omega\right)=\gamma \left(\omega-x\right)+A^{*}\left(Ax-y^{\delta}\right)-2\beta\omega.
    \end{equation*}
    Consequently,
    \begin{equation*}
        0\leq\left<J'\left(\overline{\omega}\right),\omega-\overline{\omega}\right> =\lim_{t\rightarrow 0^{+}} \frac{J\left(\overline{\omega}+ t\left(\omega-\overline{\omega}\right)\right) -J\left(\overline{\omega}\right)}{t}.
    \end{equation*}
    By assumption and Lemma \ref{lem5.4}, $\mathcal{S}_{\beta,\gamma}(\omega,x)$ is locally convex at $\overline{\omega}$. This implies that
    \begin{align*}
        J\left(x\right)-J\left(\overline{x}\right)
        & =\lim_{t\rightarrow 0^{+}} \frac{tJ\left(x\right)+\left(1-t\right) J\left(\overline{x}\right)-J\left(\overline{x}\right)}{t} \\
        & \geq \lim_{t\rightarrow 0^{+}} \frac{J\left(\overline{x}+ t\left(x-\overline{x}\right)\right) -J\left(\overline{x}\right)}{t} \\
        & =\left<J'\left(\overline{x}\right),x-\overline{x}\right>\geq 0
    \end{align*}
    for all $x\in B_{R}$. This concludes the proof of the lemma.
\end{proof}

Furthermore, the PG algorithm for problem \eqref{equ1.2} based on the surrogate function is detailed in Algorithm \ref{alg3}. To establish the convergence of Algorithm \ref{alg4}, we introduce specific conditions on the operator $A$ and the parameter $\gamma$.

\begin{algorithm}
\caption{PG-$\left(\ell_{1}^{2}-\eta\ell_{2}^{2}\right)$ algorithm based on SF for problem \eqref{equ1.4}}
\begin{algorithmic}\label{alg3}
\STATE{\textbf{Initialization}: Set $k=0$, $\mathbf{R}\in\mathbb{R}^{+}$, $x_{0}\in \ell_{2}$ such that $\Phi\left(x_{0}\right)<+\infty$.}
\STATE{\textbf{General step}: for any $k=0,1,2,\cdots$ execute the following steps:}
\STATE{\qquad 1: pick $\gamma>0$}
\STATE{\qquad 2: set
\begin{equation*}
    x^{k+1}=\mathbb{P}_{\mathbf{R}}\left(x^{k}+\frac{2\beta}{\gamma-2\beta}x^{k} -\frac{1}{\gamma-2\beta}A^{*}\left( Ax^{k}-y^{\delta}\right)\right)
\end{equation*}}
\STATE{\qquad 3: check the stopping criterion and return $x^{k+1}$ as a solution.}
\end{algorithmic}
\end{algorithm}

\begin{assumption}\label{ass5.6}
    Let $r:=\left\|A^{*}A\right\|_{L(\ell_{2},\ell_{2})}<1$. Assume that\\
    (a) $\left\|Ax\right\|_{Y}^{2}\leq\frac{\gamma r}{2}\left\|x\right\|_{\ell_{2}}^{2}$ for all $x\in\ell_{2}$.\\
    (b) $\gamma> 2\beta$.
\end{assumption}

\par In Assumption \ref{ass5.6}, we let $r:=\left\|A^{*}A\right\|_{L[\ell_{2},\ell_{2}]}<1$. In the classical theory of sparsity regularization, the value of $\left\|A^{*}A\right\|_{L[\ell_{2},\ell_{2}]}$ is assumed to be less than 1 (\cite{DDD04}). This requirement is still needed in this paper. If $r>1$ we need to re-scale the original ill-posed problem by $Ax=y \rightarrow \left(\frac{1}{c}A\right)x=\frac{1}{c}y$ so that $\frac{1}{c^{2}}\left\|A^{*}A\right\|_{L[\ell_{2},\ell_{2}]}<1$, where $c>1$.

\begin{lemma}\label{lem5.7}
    Let Assumption \ref{ass5.6} holds and the sequence $\left\{x^{k+1}\right\}$ is generated through the iterative process described in equation \eqref{equ5.6}. Then, it holds that
    \begin{equation*}
        \mathcal{D}_{\beta}^{\delta}\left(x^{k+1}\right)\leq \mathcal{D}_{\beta}^{\delta}\left(x^{k}\right)
    \end{equation*}
    and
    \begin{equation*}
        \lim_{k\rightarrow\infty}\left\|x^{k+1}-x^{k}\right\|_{\ell_{2}}=0.
    \end{equation*}
\end{lemma}

\begin{proof}
    Utilizing Lemma \ref{lem5.5} in conjunction with the definition of $x^{k+1}$, being a minimizer of $\mathcal{S}_{\beta,\gamma}\left(x,x^{k}\right)$, we derive
    \begin{align*}
        \mathcal{D}_{\beta}^{\delta}\left(x^{k+1}\right)
        & \leq \mathcal{D}_{\beta}^{\delta}\left(x^{k+1}\right) +\frac{2-r}{2r}\left\|A\left(x^{k+1}-x^{k}\right)\right\|_{Y}^{2} \\
        & = \frac{1}{2}\left\|Ax^{k+1}-y^{\delta}\right\|_{Y}^{2}- \beta\left\|x^{k+1}\right\|_{\ell_{2}}^{2}+ \frac{1}{r}\left\|A\left(x^{k+1}-x^{k}\right)\right\|_{Y}^{2}- \frac{1}{2}\left\|A\left(x^{k+1}-x^{k}\right)\right\|_{Y}^{2} \\
        & = \frac{1}{2}\left\|Ax^{k+1}-y^{\delta}\right\|_{Y}^{2}- \beta\left\|x^{k+1}\right\|_{\ell_{2}}^{2}- \frac{1}{2}\left\|A\left(x^{k+1}-x^{k}\right)\right\|_{Y}^{2}+ \frac{\gamma}{2}\left\|x^{k+1}-x^{k}\right\|_{\ell_{2}}^{2} \\
        & = \mathcal{S}_{\beta,\gamma}\left(x^{k+1},x^{k}\right) 
        \leq \mathcal{S}_{\beta,\gamma}\left(x^{k},x^{k}\right)= \mathcal{D}_{\beta}^{\delta}\left(x^{k}\right).
    \end{align*}
    Furthermore, we obtain
    \begin{align*}
        \mathcal{S}_{\beta,\gamma}\left(x^{k+1},x^{k}\right)- \mathcal{S}_{\beta,\gamma}\left(x^{k+1},x^{k+1}\right)
        & = \frac{\gamma}{2}\left\|x^{k+1}-x^{k}\right\|_{\ell_{2}}^{2} -\frac{1}{2}\left\|A\left(x^{k+1}-x^{k}\right)\right\|_{Y}^{2} \\
        & = \frac{\gamma\left(2-r\right)}{4} \left\|x^{k+1}-x^{k}\right\|_{\ell_{2}}^{2}. 
    \end{align*}
    Consequently, we find
    \begin{align*}
        \sum_{k=0}^{N}\left\|x^{k+1}-x^{k}\right\|_{\ell_{2}}^{2}
        & \leq \frac{4}{\gamma\left(2-r\right)}\sum_{k=0}^{N} \left[\mathcal{S}_{\beta,\gamma}\left(x^{k+1},x^{k}\right)- \mathcal{S}_{\beta,\gamma}\left(x^{k+1},x^{k+1}\right)\right]  \\
        & \leq \frac{4}{\gamma\left(2-r\right)}\sum_{k=0}^{N} \left[\mathcal{S}_{\beta,\gamma}\left(x^{k},x^{k}\right)- \mathcal{S}_{\beta,\gamma}\left(x^{k+1},x^{k+1}\right)\right] \\
        & = \frac{4}{\gamma\left(2-r\right)} \left[\mathcal{S}_{\beta,\gamma}\left(x^{0},x^{0}\right)- \mathcal{S}_{\beta,\gamma}\left(x^{N+1},x^{N+1}\right)\right] \\
        & \leq \frac{4}{\gamma\left(2-r\right)} \left[\mathcal{S}_{\beta,\gamma}\left(x^{0},x^{0}\right)+ \beta\left\|x^{0}\right\|_{\ell_{2}}^{2}\right].
    \end{align*}
    Given that $\sum_{k=0}^{N}\left\|x^{k+1}-x^{k}\right\|_{\ell_{2}}^{2}$ is uniformly bounded concerning $N$, we conclude that the series $\sum_{k=0}^{\infty}\left\|x^{k+1}-x^{k}\right\|_{\ell_{2}}^{2}$ converges. This verifies the lemma.
\end{proof}

\begin{theorem}\label{the5.8}
    Let $\{x^{k}\}$ be the sequence generated by Algorithm \ref{alg4}. Then there exists a subsequence of $\{x^{k}\}$ that converges to a stationary point $x^{*}$ of \eqref{equ5.2}, meaning that $x^{*}$ satisfies
    \begin{equation*}
        \left<2\beta x^{*}-A^{*}\left(Ax^{*}-y^{\delta}\right), x-x^{*}\right>\leq 0
    \end{equation*}
    for $\forall x\in B_{\mathbf{R}}$, where $\mathbf{R}$ adheres to Morozov's discrepancy principle \eqref{equ2.2}.
\end{theorem}

\begin{proof}
    We begin by demonstrating the convergence of the inner iteration as follows
    \begin{equation*}
        x^{k+1}=\mathbb{P}_{\mathbf{R}}\left(x^{k}+\frac{2\beta}{\gamma-2\beta}x^{k} -\frac{1}{\gamma-2\beta}A^{*}\left(Ax^{k}-y^{\delta}\right)\right).
    \end{equation*}
    For any fixed $\mathbf{R}>0$, since $\{x^{k}\}\subseteq B_{\mathbf{R}}$ is bounded, there exists a subsequence $\{x^{k_{i}}\}$ of $\{x^{k}\}$ converging to an element $x^{*}\in B_{\mathbf{R}}$, i.e., $x^{k_{i}}\rightarrow x^{*}$ in $B_{\mathbf{R}}$. Given that $A$ is linear and bounded, we also have $Ax^{k_{i}}\rightarrow Ax^{*}$. By applying Lemma \ref{lem5.2} and considering the definition of $x^{k+1}$, for all $\omega\in\ B_{\mathbf{R}}$, we obtain
    \begin{equation*}
        \left<\frac{\gamma}{\gamma-2\beta}x^{k}-\frac{1}{\gamma-2\beta}A^{*}\left(Ax^{k}- y^{\delta}\right)-x^{k+1},\omega-x^{k+1}\right>\leq 0.
    \end{equation*}
    This implies
    \begin{equation}\label{equ5.7}
        \left<\frac{\gamma}{\gamma-2\beta}x^{k_{i}}-\frac{1}{\gamma-2\beta}A^{*}\left(Ax^{k_{i}} -y^{\delta}\right) -x^{k_{i}+1},\omega-x^{k_{i}+1}\right>\leq 0.
    \end{equation}
    Taking the limit as $i\rightarrow\infty$ in \eqref{equ5.7}, we find
    \begin{equation*}
        \lim_{i\rightarrow\infty}\left<\frac{\gamma}{\gamma-2\beta}x^{k_{i}}- \frac{1}{\gamma-2\beta}A^{*}\left(Ax^{k_{i}} -y^{\delta}\right) -x^{k_{i}+1},\omega-x^{k_{i}+1}\right>\leq 0.
    \end{equation*}
    Since $\left\|x^{k_{i}}-x^{k_{i}+1}\right\|_{\ell_{2}}\rightarrow 0$ as $i\rightarrow\infty$ and $\left\{\omega-x^{k_{i}+1}\right\}$ remains uniformly bounded for $\omega,x^{k_{i}+1}\in B_{\mathbf{R}}$, it follows that
    \begin{equation}\label{equ5.8}
        \lim_{i\rightarrow\infty}\left|\left<x^{k_{i}}-x^{k_{i}+1}, \omega-x^{k_{i}+1}\right>\right|=0.
    \end{equation}
    By combining \eqref{equ5.7} with \eqref{equ5.8}, we conclude that
    \begin{equation}\label{equ5.9}
        \lim_{i\rightarrow\infty}\left<\frac{2\beta}{\gamma}x^{k_{i}}- \frac{1}{\gamma} A^{*}\left(Ax^{k_{i}} -y^{\delta}\right),x-x^{k_{i}+1}\right>\leq 0.
    \end{equation}
    As $x^{k_{i}}\rightarrow x^{*}$ for $i\rightarrow\infty$, it follows from \eqref{equ5.9} that 
    \begin{equation*}
        \left<2\beta x^{*}- A^{*}\left(Ax^{*} -y^{\delta}\right),x-x^{*}\right>\leq 0.
    \end{equation*}
    By invoking Lemma \ref{lem5.3}, we establish that $x^{*}$ is a stationary point of $\mathcal{D}_{\beta}^{\delta}\left(x\right)$ on $B_{\mathbf{R}}$ for any $\gamma>0$. Furthermore, $x^{*}$ emerges as a stationary point of $\mathcal{D}_{\beta}^{\delta}\left(x\right)$ on $B_{\mathbf{R}}$, where $\mathbf{R}$ adheres to Morozov’s discrepancy principle. This completes the proof of the theorem.
\end{proof}

\par Finally, we will demonstrate the stability of Algorithm \ref{alg3}.

\begin{theorem}\label{the5.9}
    Let $x_{\delta_{n}}^{k_{i}}\rightarrow x_{\delta_{n}^{*}}$ where $x_{\delta_{n}}^{*}$ denotes a stationary point of the optimization problem \eqref{equ5.2}. In this context, $y^{\delta}$ is substituted by  $y^{\delta_{n}}$, and a subsequence $\left\{x_{\delta_{n}}^{k_{i}}\right\}$ of $\left\{x_{\delta_{n}}^{k}\right\}$ is generated through Algorithm \ref{alg4}. Furthermore, assume that $\delta_{n}\rightarrow\delta$ as $n\rightarrow\infty$. It follows that there exists a subsequence of $\left\{x_{\delta_{n}}^{*}\right\}$, still denoted by $\left\{x_{\delta_{n}}^{*}\right\}$, such that $x_{\delta_{n}}^{*}\rightarrow x^{*}$. If $x^{*}\neq 0$, it qualifies as a stationary point of the problem \eqref{equ5.2}. Moreover, if $x^{*}$ is unique, then $\lim_{n\rightarrow\infty}\left\|x_{\delta_{n}}^{*}-x^{*}\right\|_{\ell_{2}}=0$.
\end{theorem}

\begin{proof}
    By virtue of the aforementioned assumption, $x_{\delta_{n}}^{*}$ is a stationary point associated with the problem \eqref{equ5.2} where $y^{\delta}$ has been substituted by $y^{\delta_{n}}$. Consequently, we can assert
    \begin{equation}\label{equ5.10}
        \left<2\beta x_{\delta_{n}}^{*}-A^{*}\left(Ax_{\delta_{n}}^{*}-y^{\delta}\right), x-x_{\delta_{n}}^{*}\right>\leq 0
    \end{equation}
    for $\forall x\in B_{\mathbf{R}}$. Since the sequence $\left\{x_{\delta_{n}}^{*}\right\}$ is bounded, it follows that a convergent subsequence exists, still denoted by $\left\{x_{\delta_{n}}^{*}\right\}$, converging to some element $x^{*}$, i.e., $x_{\delta_{n}}^{*}\rightarrow x^{*}$. Taking the limit as $n\rightarrow\infty$ in \eqref{equ5.10} leads to
    \begin{equation}\label{equ5.11}
        \left<2\beta x^{*}-A^{*}\left(Ax^{*}-y^{\delta}\right), x-x^{*}\right>\leq 0
    \end{equation}
    for $\forall x\in B_{R}$. By applying Lemma \ref{lem5.2}, this indicates that $x^{*}$ is indeed a stationary point of the problem \eqref{equ5.2}. In cases where $x^{*}$ is unique, it can be inferred that every subsequence $\left\{x_{\delta_{n}}^{*}\right\}$ converges to $x^{*}$. Thus, we establish that $\lim_{n\rightarrow\infty}\left\|x_{\delta_{n}}^{*}-x^{*} \right\|_{\ell_{2}}=0$.
\end{proof}

\subsection{Determination of the radius $\mathbf{R}$}\label{sec5.2}

\par In this section, we examine certain properties of $\ell_{2}$-projections onto the $\ell_{1}$-ball. Specifically, we delve into the relationships with the projection operator and the function defined in Defintion \ref{def4.5}.

\begin{lemma}\label{lem5.10}
    For any fixed $a\in\ell_{2}$ and $\alpha>0$, the quantity $\left\|\mathbb{H}_{\alpha}\left(a\right)\right\|_{\ell_{1}}$ is a continuous and decreasing function of $\alpha$. Furthermore, it holds that $\left\|\mathbb{H}_{\alpha}\left(a\right)\right\|_{\ell_{1}}\rightarrow\left\|a\right\|_{\ell_{1}}$ as $\alpha\rightarrow 0$, and  $\left\|\mathbb{H}_{\alpha}\left(a\right)\right\|_{\ell_{1}}\rightarrow 0$ as $\alpha\rightarrow\infty$.
\end{lemma}

\begin{proof}
    Given that $\lambda_{i}>0$  as stated in Lemma \ref{lem4.3}, we can express
    \begin{equation*}
        \left\|\mathbb{H}_{\alpha}\left(a\right)\right\|_{\ell_{1}}= \sum_{i=1}^{n}\frac{\lambda_{i}}{2\alpha+\lambda_{i}} \left|a_{i}\right|.
    \end{equation*}
    The value of $\left\|\mathbb{H}_{\alpha}\left(a\right)\right\|_{\ell_{1}}$ is contingent upon the positive coefficients of $\left|a_{i}\right|$, i.e.,
    \begin{equation*}
       \lim_{\alpha\rightarrow 0}\frac{\lambda_{i}}{2\alpha+\lambda_{i}}=1 \quad {\rm and} \quad \lim_{\alpha\rightarrow\infty} \frac{\lambda_{i}}{2\alpha+\lambda_{i}}=0.
    \end{equation*}
    It leads to the conclusion that
    \begin{equation*}
        \lim_{\alpha\rightarrow 0} \left\|\mathbb{H}_{\alpha}\left(a\right)\right \|_{\ell_{1}}=\left\|a\right\|_{\ell_{1}} \quad {\rm and} \quad \lim_{\alpha\rightarrow\infty}\left\|\mathbb{H}_{\alpha}\left(a\right)\right \|_{\ell_{1}}=0.
    \end{equation*}
    Thus, the lemma is established.
\end{proof}

\begin{figure}[htbp]
    \centering
    \includegraphics[scale=0.28]{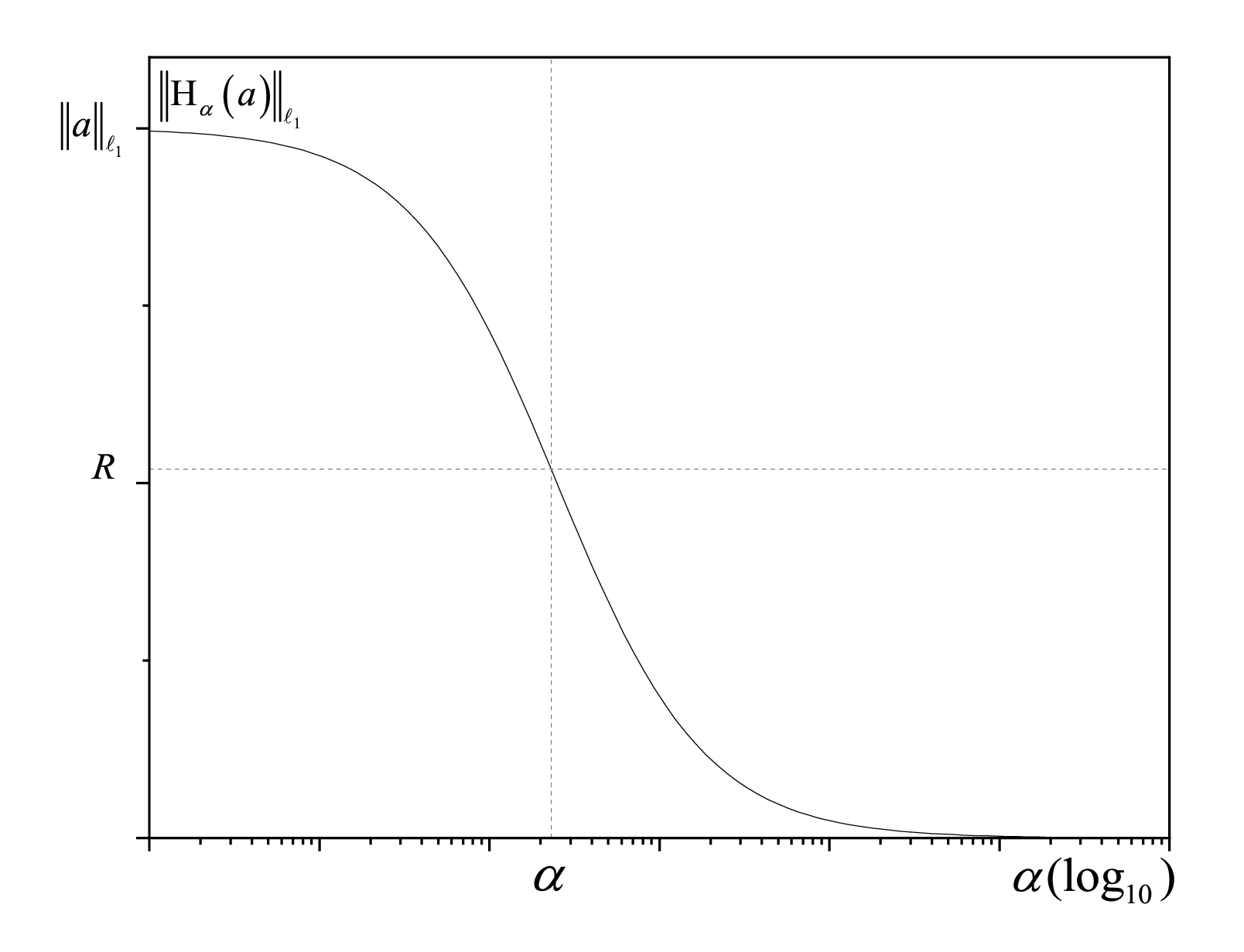}
    \caption{For a given vector $a\in\ell_{2}$, $\left\|\mathbb{H}_{\alpha}\left(a\right)\right\|_{\ell_{2}}$ is a continuous and decreasing function of $\alpha$.\ Finding an $\alpha$ such that $\left\|\mathbb{H}_{\alpha}\left(a\right)\right\|_{\ell_{1}}=\mathbf{R}$ ultimately comes down to a nonlinear interpolation. The figure is made for the finite-dimensional case.}
    \label{figure:2}
\end{figure}

\par A schematic representation is provided in Fig.\ \ref{figure:2}. The subsequent lemma demonstrates that the $\ell_{2}$ projection $\mathbb{P}_{\mathbf{R}}\left(a\right)$ can be effectively obtained by selecting an appropriate $\alpha$. 

\begin{lemma}\label{lem5.11}
     For $a\in\ell_{2}$, if $\left\|a\right\|_{\ell_{1}}>\mathbf{R}$, then the $\ell_{2}$ projection of $a$ onto the $\ell_{1}^{2}$ ball of radius $\mathbf{R}$ is given by $\mathbb{P}_{\mathbf{R}}\left(a\right)=\mathbb{H}_{\alpha}\left(a\right)$, where $\alpha$ is selected such that $\left\|\mathbb{H}_{\alpha}\left(a\right)\right\|_{\ell_{1}}=\mathbf{R}$. Conversely, if $\left\|a\right\|_{\ell_{1}}\leq\mathbf{R}$, then $\mathbb{P}_{\mathbf{R}}\left(a\right)=\lim_{\alpha\rightarrow 0}\mathbb{H}_{\alpha}\left(a\right)=a$.
\end{lemma}

\begin{proof}
    Assume $\left\|a\right\|_{\ell_{1}}>\mathbf{R}$. Notably, since $\left\|\mathbb{H}_{\alpha}\left(a\right)\right\|_{\ell_{1}}$ is continuous in $\alpha$ and it holds that $\lim_{\alpha\rightarrow \infty} \left\|\mathbb{H}_{\alpha}\left(a\right)\right\|_{\ell_{1}}=0$ according to Lemma \ref{lem5.10}, one can select $\alpha$ such that $\left\|\mathbb{H}_{\alpha}\left(a\right)\right\|_{\ell_{1}}=\mathbf{R}$ (refer to Fig.\ \ref{figure:2}). Furthermore, assuming $b=\mathbb{H}_{\alpha}\left(a\right)$ is the unique minimizer of the functional $\frac{1}{2}\left\|x-a\right\|^{2}+\alpha\left\|x\right\|_{\ell_{1}}^{2}$ as established in \cite{BA17}, we have
    \begin{equation*}
        \frac{1}{2}\left\|b-a\right\|^{2}+\alpha\left\|b\right\|_{\ell_{1}}^{2} \leq\frac{1}{2}\left\|x-a\right\|^{2}+\alpha\left\|x\right\|_{\ell_{1}}^{2}
    \end{equation*}
    for all $x\neq b$. Given that $\left\|b\right\|_{\ell_{1}}=\mathbf{R}$,  it follows that for $\forall x\in B_{\mathbf{R}}$ with $x\neq b$
    \begin{equation*}
        \left\|b-a\right\|_{\ell_{1}}^{2}\leq\left\|x-a\right\|_{\ell_{1}}^{2}.
    \end{equation*}
    Thus, we can conclude that $b$ is closer to $a$ than any other $x\in B_{R}$. In summary, this leads us to the conclusion that $\mathbb{P}_{R}\left(a\right)=b=\mathbb{H}_{\alpha}\left(a\right)$.
\end{proof}

\par From the preceding discussion, it is evident that problem \eqref{equ5.1} is equivalent to problem \eqref{equ5.2} for a specific $\mathbf{R}$. In practical applications, selecting an appropriate value for $\mathbf{R}$ is critical for computation. In the subsequent sections, we propose a strategy for determining the radius $\mathbf{R}$ of the $\ell_{1}^{2}$-ball constraint using Morozov's discrepancy principle.

\par According to Lemma \ref{lem5.11}, for a given $\alpha$ in \eqref{equ5.1}, the corresponding $\mathbf{R}$ in \eqref{equ5.2} should be selected such that $\left\|x_{\alpha,\beta}^{\delta}\right\|_{\ell_{1}}^{2}= \mathbf{R}^{2}$. However, the value of $\left\|x_{\alpha,\beta}^{\delta}\right\|_{\ell_{1}}^{2}$ is typically unknown before the initiation of the PG method outlined in \eqref{equ5.6}. Consequently, if an approximation of $x_{\alpha,\beta}^{\delta}$ is acquired through the HV-($\ell_{1}^{2}-\eta\ell_{2}^{2}$) algorithm, the resulting algorithm may no longer retain its accelerated properties. Thus, it becomes essential to determine the appropriate value of $\mathbf{R}$ accurately. Drawing inspiration from \cite{DH20} and \cite{DH23}, we apply Morozov's discrepancy principle to guide the selection of $\mathbf{R}$. 

\begin{algorithm}
\caption{PG-$\left(\ell_{1}^{2}-\eta\ell_{2}^{2}\right)$ algorithm under MDP based on SF}
\begin{algorithmic}\label{alg4}
\STATE{\textbf{Initialization}: Set $\mathbf{R}_{max},\mathbf{R}_{min}\in\mathbb{R}$ with $\mathbf{R}_{max}>\mathbf{R}_{min}$, $\tau_{1},\tau_{2}\in\mathbb{R}^{+}$ with $\tau_{2}>\tau_{1}$, and $x_{0}\in \mathbb{R}^{n}$ such that $\Phi\left(x_{0}\right)<+\infty$.}
\STATE{\textbf{General step}:}
\STATE{for any $j=0,1,2,\cdots$ execute the following steps:}
\STATE{\qquad $\mathbf{R}_{j}=\frac{\mathbf{R}_{max}+\mathbf{R}_{min}}{2}$;}
\STATE{\qquad for any $k=0,1,2,\cdots$ execute the following steps:}
\STATE{\qquad\qquad 1: pick $\gamma>0$ such that $\gamma>2\beta$;}
\STATE{\qquad\qquad 2: set
\begin{equation*}
    x^{k+1}=\mathbb{P}_{\mathbf{R}_{j}}\left(x^{k}+\frac{2\beta}{\gamma-2\beta}x^{k} -\frac{1}{\gamma-2\beta}A^{*}\left(Ax^{k}-y^{\delta}\right)\right);
\end{equation*}}
\STATE{\qquad\qquad 3: check the stopping criterion and return $x_{\alpha,\beta}^{\delta}=x^{k+1}$ as a solution.}
\STATE{\qquad end for}
\STATE{\qquad if $\left\|Ax_ {\alpha,\beta}^{\delta}-y^{\delta}\right\|_{Y}<\tau_{1}\delta$, then $\mathbf{R}_{max}=\mathbf{R}_{j}$;}
\STATE{\qquad if $\left\|Ax_ {\alpha,\beta}^{\delta}-y^{\delta}\right\|_{Y}>\tau_{2}\delta$, then
$\mathbf{R}_{min}=\mathbf{R}_{j}$;}
\STATE{end for}
\STATE{\qquad return $\mathbf{R}=\mathbf{R}_{j}$ and $x_{\alpha,\beta}^{\delta}=x^{k+1}$;}
\end{algorithmic}
\end{algorithm}

\par To facilitate our subsequent discourse, we will adopt a notational convention that involves the slight abuse of the symbol, whereby we denote the set  $\mathbf{R}^{2}$ simply as $\mathbf{R}$. For any given $\mathbf{R}$, it is imperative to verify whether the regularization parameter $\alpha$ fulfills Morozov's discrepancy principle, which requires
\begin{equation*}
    \tau_{1}\delta\leq \left\|Ax_{\alpha,\beta}^{\delta}-y^{\delta}\right\|_{Y}\leq \tau_{2}\delta,\quad \tau_{2}\geq \tau_{1}>1.
\end{equation*}
For any fixed $\mathbf{R}$, we compute $x_{\alpha,\beta}^{\delta}$ via Algorithm \ref{alg3}. As we know, the discrepancy $\left\|Ax_{\alpha,\beta}^{\delta}-y^{\delta}\right\|_{Y}$ behaves as an increasing function of $\alpha$. Since $\alpha$ is inversely related to $\mathbf{R}$, it follows that the discrepancy $\left\|Ax_{\alpha,\beta}^{\delta}-y^{\delta}\right\|_{Y}$ is a decreasing function of $\mathbf{R}$. To determine an appropriate value for the regularization parameter $\mathbf{R}$, we initiate the process by defining a larger upper bound $\mathbf{R}_{max}$ and a smaller lower bound $\mathbf{R}_{min}$. Let $\mathbf{R}=\left(\mathbf{R}_{max}+\mathbf{R}_{min}\right)/2$ and compute the regularization solution $x_{\alpha,\beta}^{\delta}$ utilizing Algorithm \ref{alg3}. If $\left\|Ax_ {\alpha,\beta}^{\delta}-y^{\delta}\right\|_{Y}<\tau_{1}\delta$, we update $\mathbf{R}_{max}=\mathbf{R}$; if $\left\|Ax_ {\alpha,\beta}^{\delta}-y^{\delta}\right\|_{Y}>\tau_{2}\delta$, we update $\mathbf{R}_{min}=\mathbf{R}$. This iterative process continues until the resulting solution satisfies Morozov's discrepancy principle, and the current value of $R$ is deemed the appropriate constraint radius for the regularization method employed. It establishes that the resultant $R$ obtained under these principles is the optimal approximation for the projection radius. Based on Morozov's discrepancy principle, the PG algorithm for problem (1.2) founded on SF is delineated in Algorithm \ref{alg4}. A natural criterion for terminating the iterations in Algorithm \ref{alg4} is based on the change in the iterative solutions. Specifically, if the absolute error between successive iterative solutions meets a predefined accuracy threshold or if the maximum number of iterations is attained, the iterative process may be deemed to have converged, thereby producing a minimizer.

\par Morozov's discrepancy principle is a pivotal methodology for determining the regularization parameter $\alpha$. When selected judiciously, it transforms the problem denoted by \eqref{equ1.4} into a regularization technique. Established literature indicates that Tikhonov functions yield a regularization method when integrated with Morozov's discrepancy principle. However, this assertion is predominantly substantiated when the regularization term is convex (\cite{AR10,R01,TLY98}). Some findings are reported in scenarios involving non-convex regularization terms in \cite{DH19,WLMC13}, wherein Morozov's discrepancy principle is employed to establish the convergence rate. Nevertheless, these results necessitate additional source conditions on the true solution $x^{\dagger}$. To the best of our knowledge, no conclusive evidence has been presented regarding the applicability of Morozov's discrepancy principle in conjunction with problem \eqref{equ1.4} as a regularization method. This paper aims to demonstrate that under specific properties of the non-convex regularization term, e.g., coercivity, weakly lower semi-continuity, and Radon-–Riesz property, the well-posedness of the regularized formulation remains intact.

\section{Numerical experiments}\label{sec6}

\par In this section, we present the results from two numerical experiments designed to demonstrate the effectiveness of the proposed algorithms. We analyze the influence of the parameter $\eta$ on the reconstruction of $x^{*}$ utilizing HV-$\left(\ell_{1}^{2}-\eta\ell_{2}^{2}\right)$ algorithm. Additionally, we assess the reconstruction quality and computational efficiency of the proposed algorithms, the HV-$\left(\ell_{1}^{2}-\eta\ell_{2}^{2}\right)$ and PG-SF algorithms, about other well-known algorithms, including ISTA as outlined in \cite{DDD04}, FISTA from \cite{BT09}, ST-$\left(\alpha\ell_{1}-\beta\ell_{2}\right)$ algorithm in \cite{DH24} and half thresholding algorithm for $\ell_{1/2}$ regularization (HT-$\ell_{1/2}$) discussed in \cite{XCXZ12}. The first example addresses a well-conditioned compressive sensing problem, while the second example focuses on an ill-conditioned image deblurring problem. All numerical experiments were conducted using MATLAB R2024a on a workstation equipped with an AMD Ryzen 5 4500U processor with Radeon Graphics operating at 2.38 GHz and 16GB of RAM.

\subsection{Well-conditioned compressive sensing with random Gaussian matrix}\label{sec6.1}

\par In the first example, we examine a commonly used random Gaussian matrix. The compressive sensing problem is defined as $A_{m\times n}x_{n}=y_{m}$, where $A_{m\times n}$ is a well-conditioned random Gaussian matrix generated by calling $A={\rm randn(m,n)}$ in Matlab. The exact data $y$ is produced by the equation $y=Ax^{\dagger}$, where the exact solution $x^{\dagger}$ represents an $s$-sparse signal supported on a random index set. White Gaussian noise is added to the exact data $y^{\dagger}$ by calling $y^{\delta}={\rm awgn}\left(Ax^{\dagger},\delta\right)$ in Matlab, where $\delta$ denotes the noise level, measured in dB, which quantifies the ratio between the true data (noise free) $y$ and Gaussian noise. The symbol $x^{*}$ signifies the reconstruction computed by the proposed algorithm. To evaluate the performance of the reconstruction $x^{*}$, we utilize the signal-to-noise ratio (SNR) and relative error (Rerror), defined as follows
\begin{equation*}
    {\rm SNR}:=-10\ {\rm log}_{10}\frac{\left\|x^{*}-x^{\dagger}\right\|_{\ell_{2}}^{2}} {\left\|x^{\dagger}\right\|_{\ell_{2}}^{2}}\quad {\rm and} \quad {\rm Rerror}:=\frac{\left\|x^{*}-x^{\dagger}\right\|_{\ell_{2}}} {\left\|x^{\dagger}\right\|_{\ell_{2}}} 
\end{equation*}
We set $n=200$, $m=0.4n$, $s=0.2m$. The value of $\left\|A_{m\times n}\right\|_{\ell_{2}}$ is approximately 22 and its condition numberis around 4. We rescale the matrix by adjusting it to $A_{m\times n}\rightarrow 0.04A_{m\times n}$ resulting in the $\ell_{2}$-norm for the rescale matrix of around 0.9. It is important to note that the condition number remains unchanged following the matrix rescaling. To compare the reconstruction effects of different iterative algorithms fairly, we set the maximum number of iterations to ${\rm maxiter} = 1500$ and $L^{k}=1$ for the HV-$\left(\ell_{1}^{2}-\eta\ell_{2}^{2}\right)$ algorithm, $\gamma=1$ for the PG-$\left(\ell_{1}^{2}-\eta\ell_{2}^{2}\right)$ algorithm, while $\lambda=1$ for the ST-$\left(\alpha\ell_{1}-\beta\ell_{2}\right)$ algorithm. The stopping criterion is that the error of the adjacent two iterative solutions is less than $10^{-5}$ or the maximum number of iterations is reached. The initial value $x^{0}$ is generated by calling $x^{0}=0.01{\rm ones}(n,1)$.

\par In Section \ref{sec3}, we employ the a priori rule or discrepancy principle to select the regularization parameter $\alpha$. However, in practical numerical implementations, accurately estimating the optimal value of $\alpha$ poses significant challenges. To address this issue, we apply the discrepancy principle to identify the regularization parameter $\alpha$ such that the residual norm for the regularized solution meets the condition $\left\|Ax^{*}-y^{\delta}\right\|_{Y}=\delta$. When a reliable estimate for the noise level $\delta$ is available, this approach effectively yields an appropriate regularization parameter.

\begin{table}[H]
\caption{{\rm SNR} of reconstruction $x^{*}$ utilizing HV-$\left(\ell_{1}^{2}-\eta\ell_{2}^{2}\right)$ algorithm for different values of $\eta$ and $\alpha$ with $\sigma=40$dB.}
\label{table:1}
\centering
\tabcolsep=0.35cm
\tabcolsep=0.01\linewidth
\scalebox{0.9}{
\begin{tabular}{ccccccc}
\hline
$\eta$ & \multicolumn{1}{c}{$\alpha=2.0\times10^{-5}$} & \multicolumn{1}{c}{$\alpha=4.0\times10^{-5}$} & \multicolumn{1}{c}{$\alpha=6.0\times10^{-5}$} & \multicolumn{1}{c}{$\alpha=8.0\times10^{-5}$} & \multicolumn{1}{c}{$\alpha=1.0\times10^{-4}$} & \multicolumn{1}{c}{$\alpha=1.2\times10^{-4}$}\\
\hline
0.0  & 19.2647 & 27.1406 & 30.6048 & 29.7776 & 29.3829 & 11.1724 \\
0.1  & 19.3353 & 27.2007 & 30.6362 & 29.8203 & 29.4218 & 11.1797 \\
0.2  & 19.4064 & 27.2636 & 30.6679 & 29.8636 & 29.4068 & 11.1868 \\
0.3  & 19.4778 & 27.3314 & 30.7000 & 29.9075 & 29.5000 & 11.1939 \\
0.4  & 19.5497 & 27.4136 & 30.7327 & 29.9522 & 29.5393 & 11.2010 \\
0.5  & 19.6221 & 27.5826 & 30.7659 & 29.9977 & 29.5788 & 11.2081 \\
0.6  & 19.6946 & 27.6737 & 30.8002 & 30.0445 & 29.6184 & 11.2151 \\
0.7  & 19.7677 & 27.7535 & 30.8358 & 30.0928 & 29.6581 & 11.2221 \\
0.8  & 19.8417 & 27.8350 & 30.8745 & 30.1438 & 29.6980 & 11.2291 \\
0.8  & 19.9166 & 27.9757 & 30.9269 & 30.2003 & 28.7380 & 11.2361\\
1.0  & \textbf{19.9924} & \textbf{28.0636} & \textbf{31.1792} & \textbf{30.6347} & \textbf{29.7782} & \textbf{11.2431}\\
\hline
\end{tabular}}
\end{table}

\par In the numerical experiments in this paper, the regularization parameter $\alpha$ is derived utilizing Hansen's Matlab tools (\cite{H07}) specifically through the mechanism designed to satisfy the discrepancy principle, as outlined in \cite{AR10,WLHC19}. It is important to note that the parameter $\alpha$ obtained via the discrepancy principle serves merely as an estimate of the optimal regularization parameter. To test the sensitivity of Algorithm \ref{alg2} concerning $\alpha$, we investigate several different regularization parameters presented in Table \ref{table:1} and \ref{table:3}.


\begin{figure}[htbp]
    \centering
    \subfigure[True signal]{\includegraphics[width=0.45\columnwidth,height=0.23\linewidth]{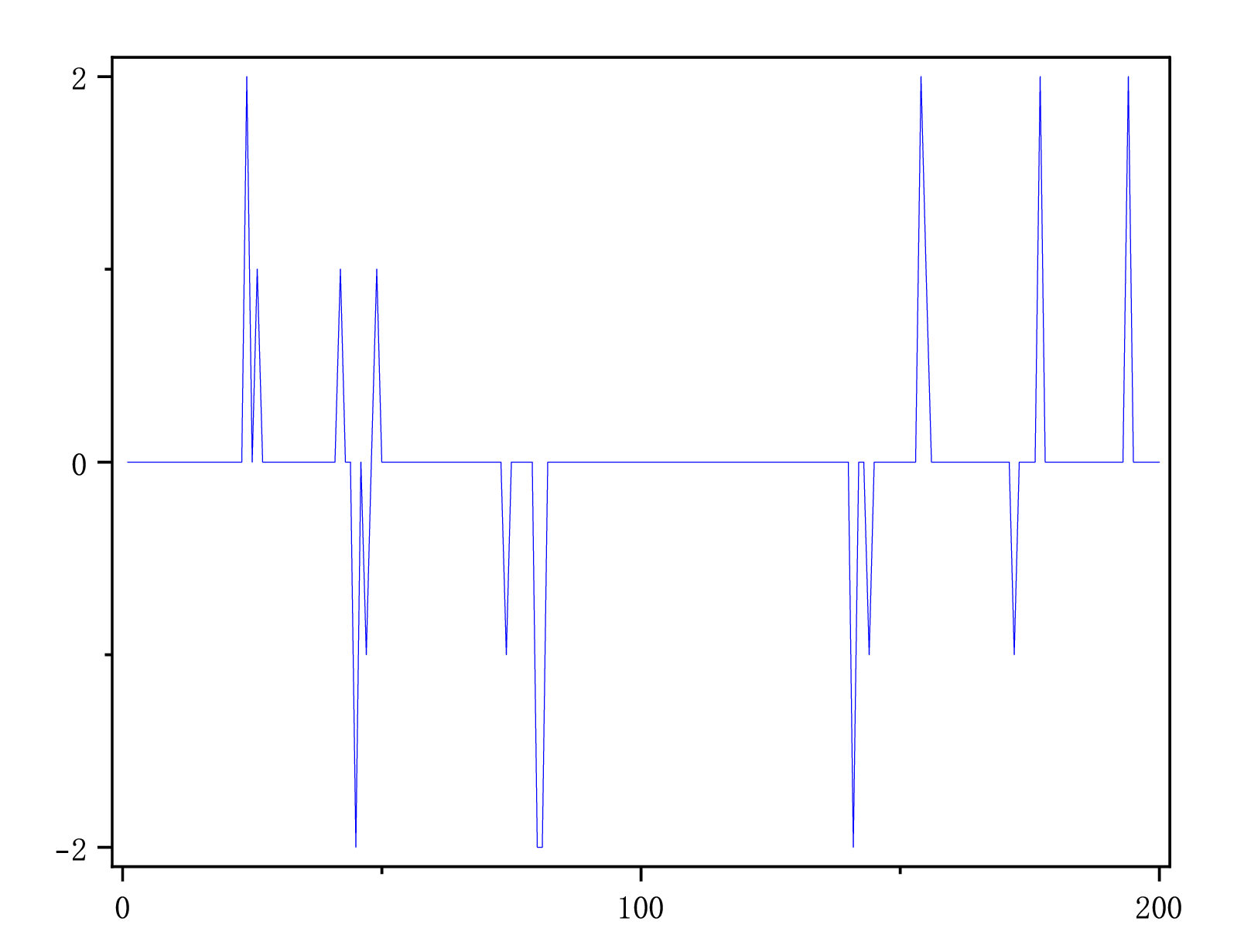}}
    \subfigure[True data and Observed data($\sigma$=40dB)]{\includegraphics[width=0.45\columnwidth,height=0.23\linewidth]{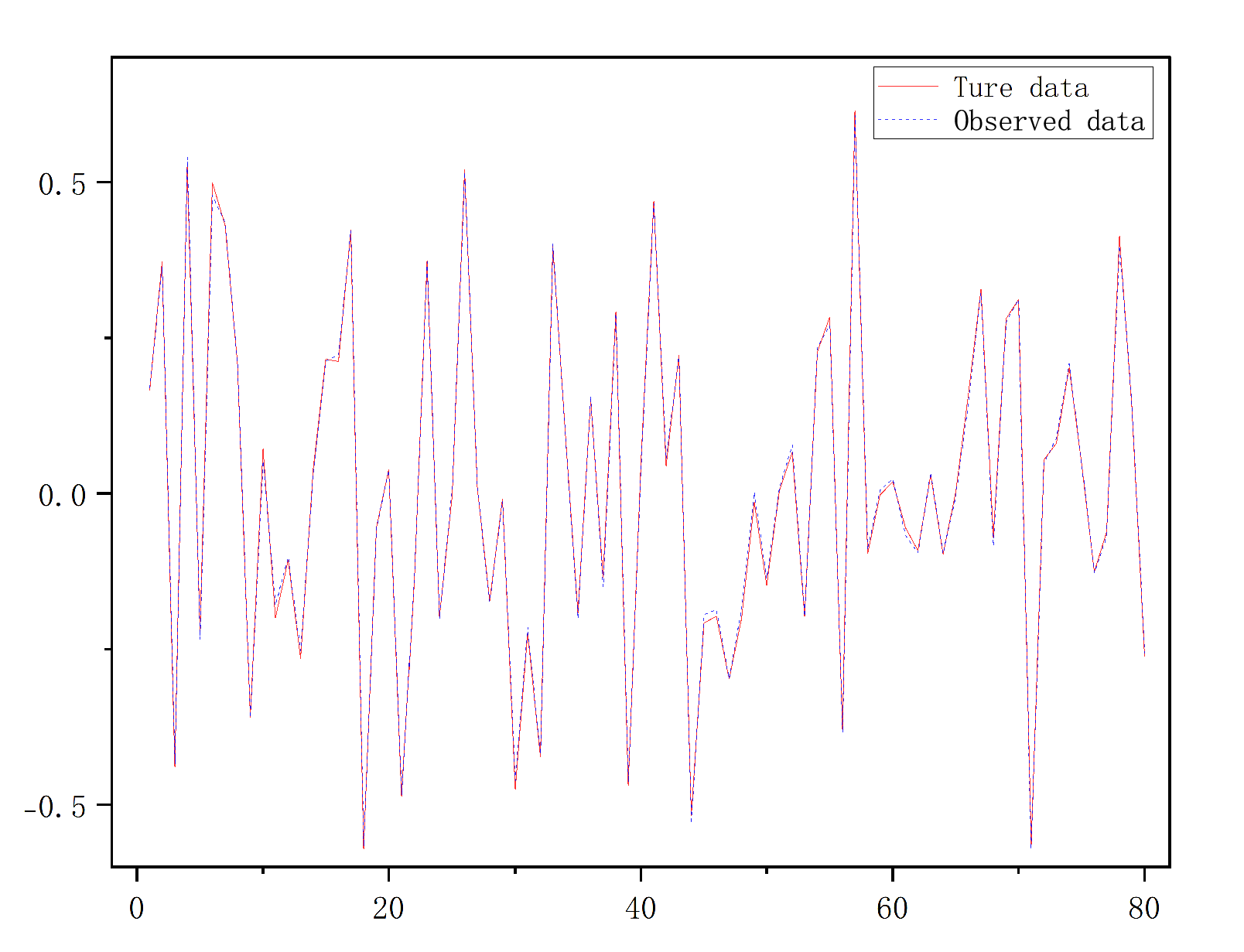}}\\
    \subfigure[ISTA, {\rm SNR} = 26.4883, Time = 13.73s ]{\includegraphics[width=0.45\columnwidth,height=0.22\linewidth]{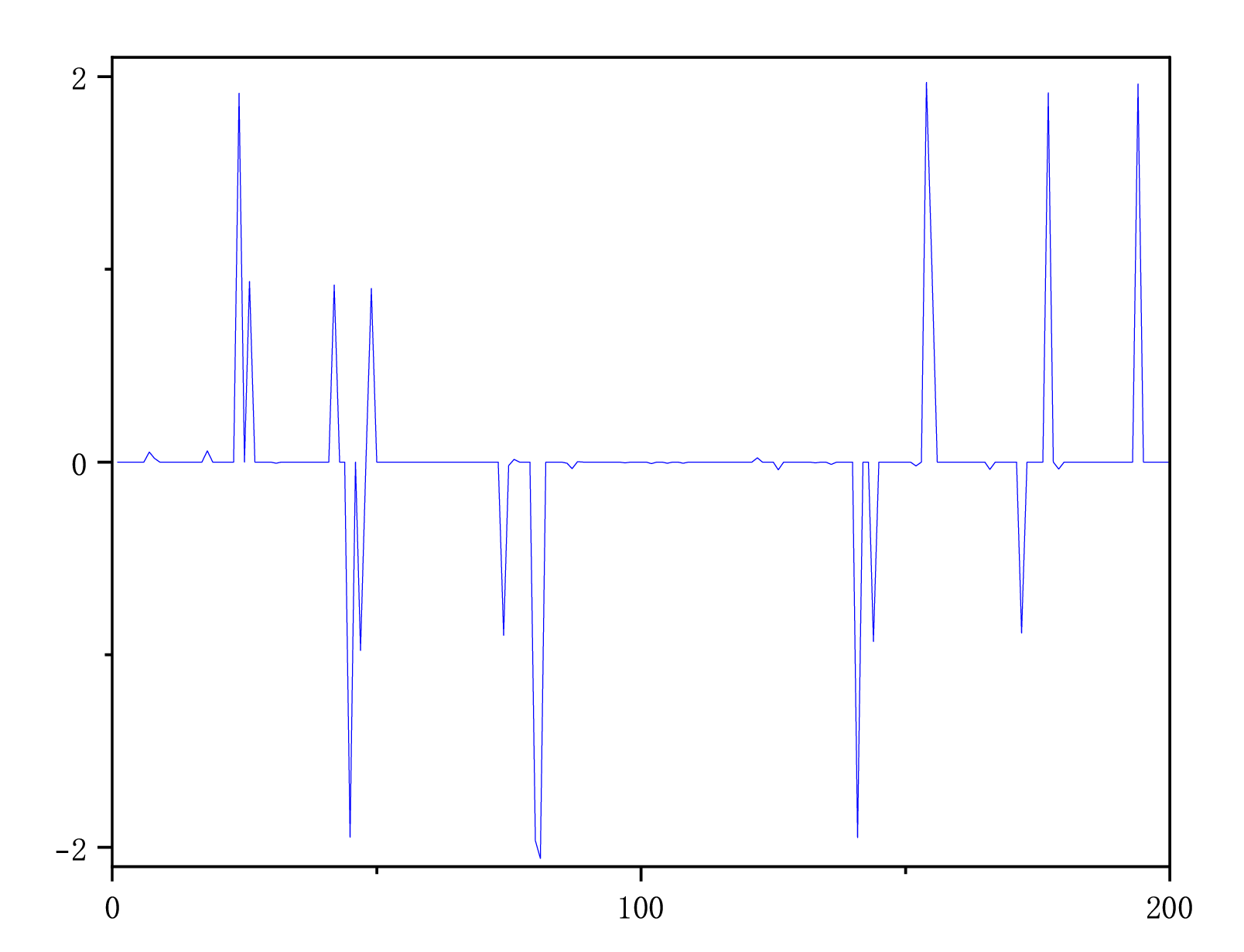}}
    \subfigure[FISTA, {\rm SNR} = 28.6235, Time = 0.50s ]{\includegraphics[width=0.45\columnwidth,height=0.22\linewidth]{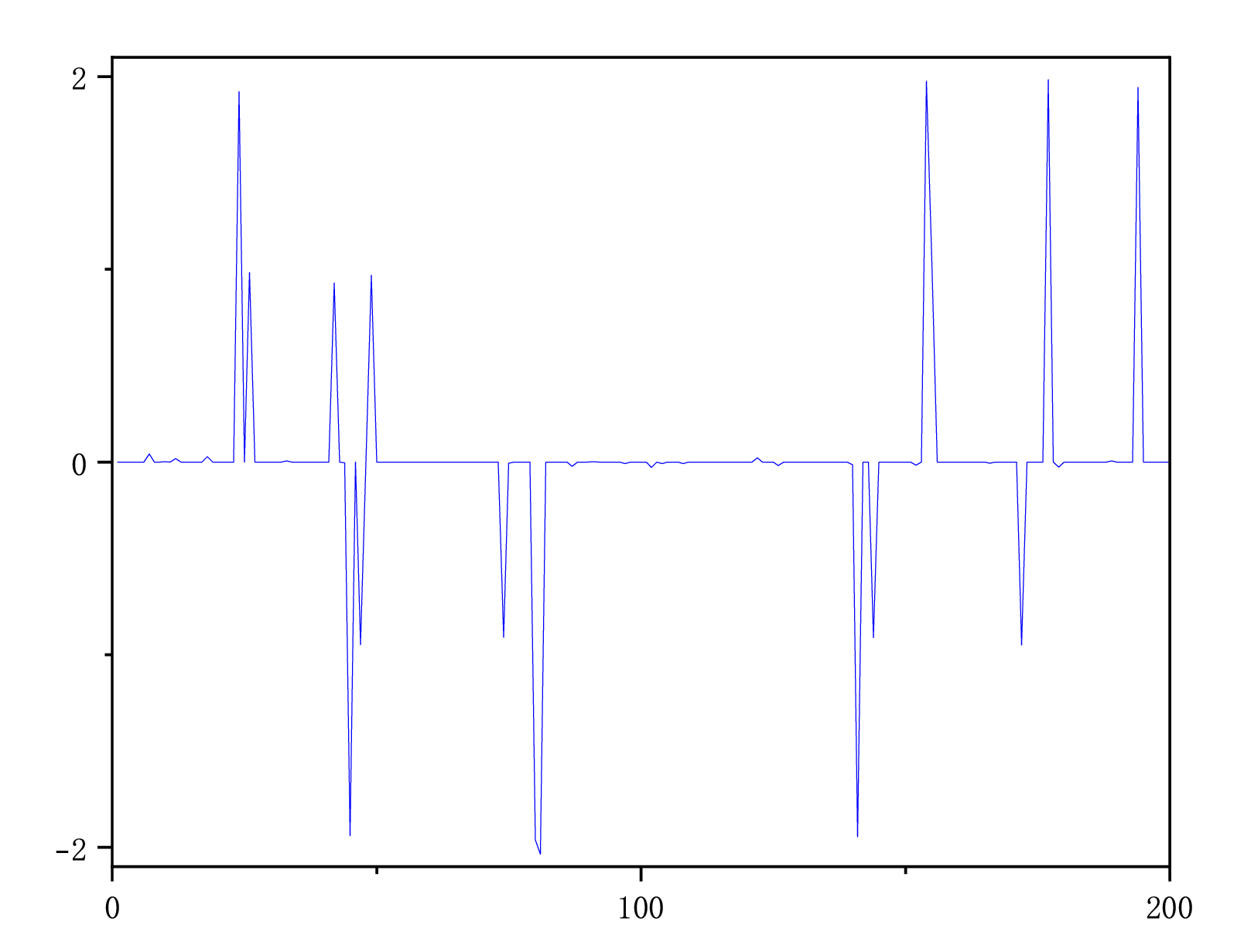}}\\
    \subfigure[ST-$\left(\alpha\ell_{1}-\beta\ell_{2}\right)$ algorithm, {\rm SNR} = 29.0615, Time = 14.20s ]{\includegraphics[width=0.45\columnwidth,height=0.22\linewidth]{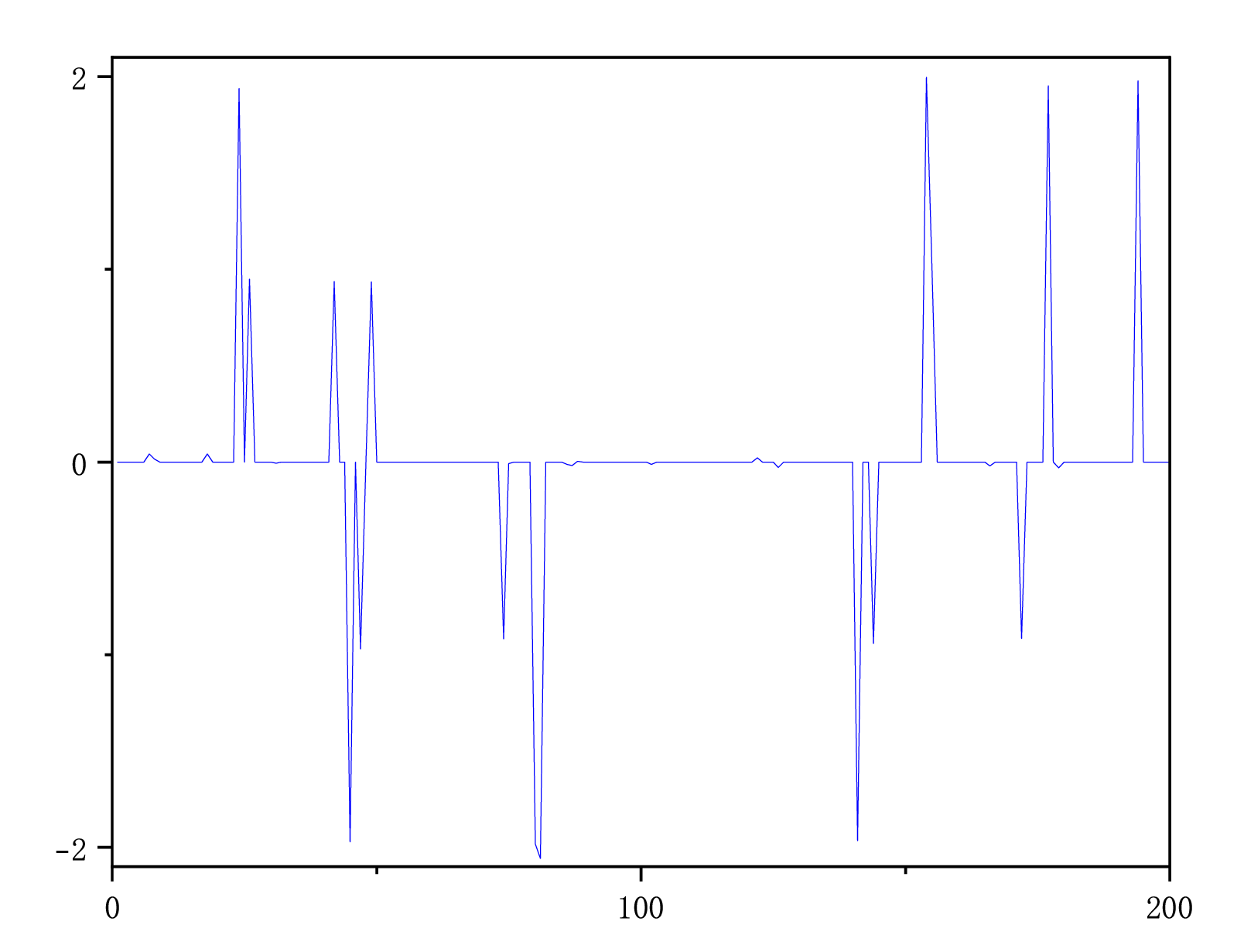}}
    \subfigure[HV-$\left(\ell_{1}^{2}-\eta\ell_{2}^{2}\right)$ algorithm, {\rm SNR} = 31.1791, Time = 282.73s ]{\includegraphics[width=0.45\columnwidth,height=0.22\linewidth]{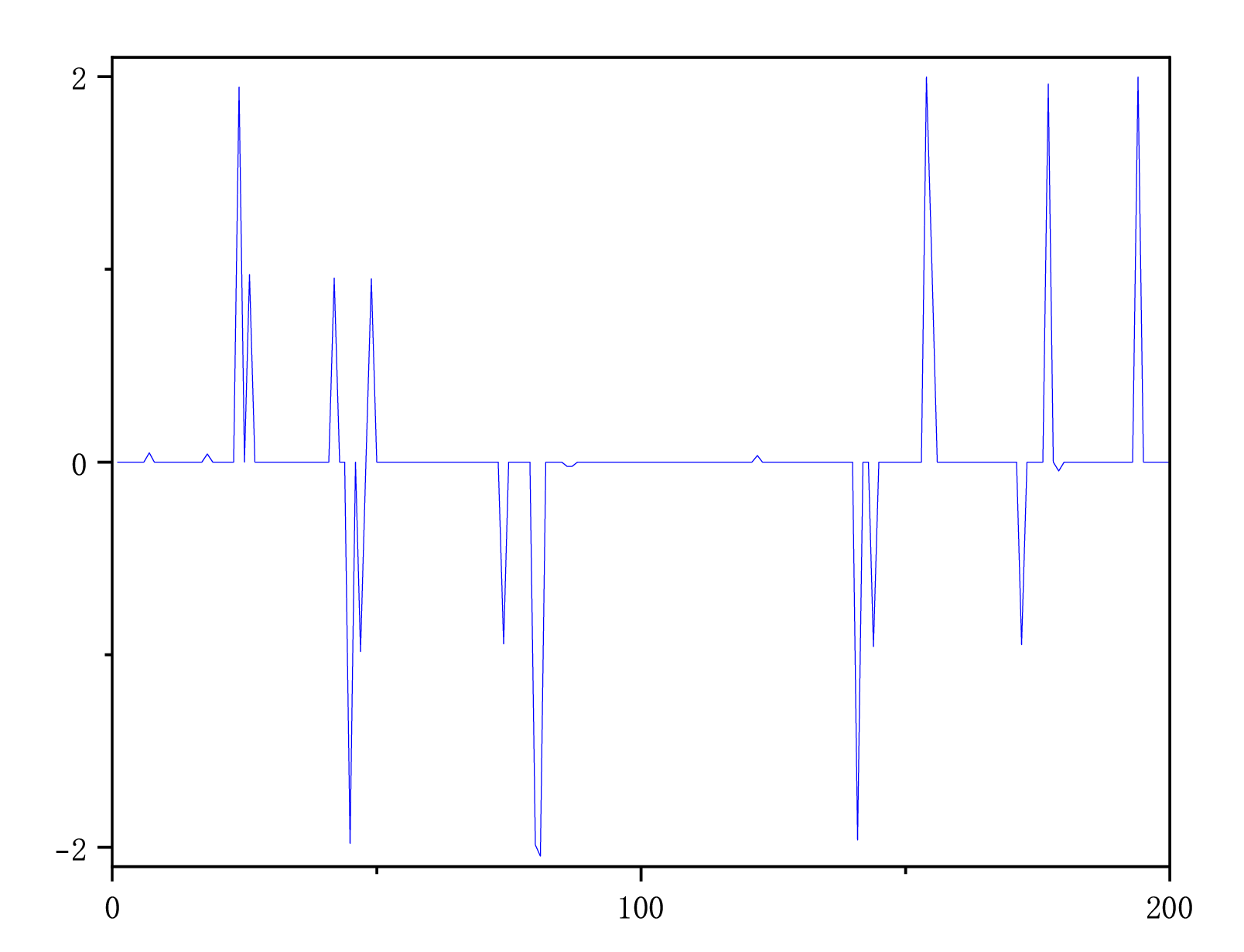}}\\
    \subfigure[HT-$\ell_{1/2}$ algorithm, {\rm SNR} = 33.9424, Time = 0.20s ]{\includegraphics[width=0.45\columnwidth,height=0.22\linewidth]{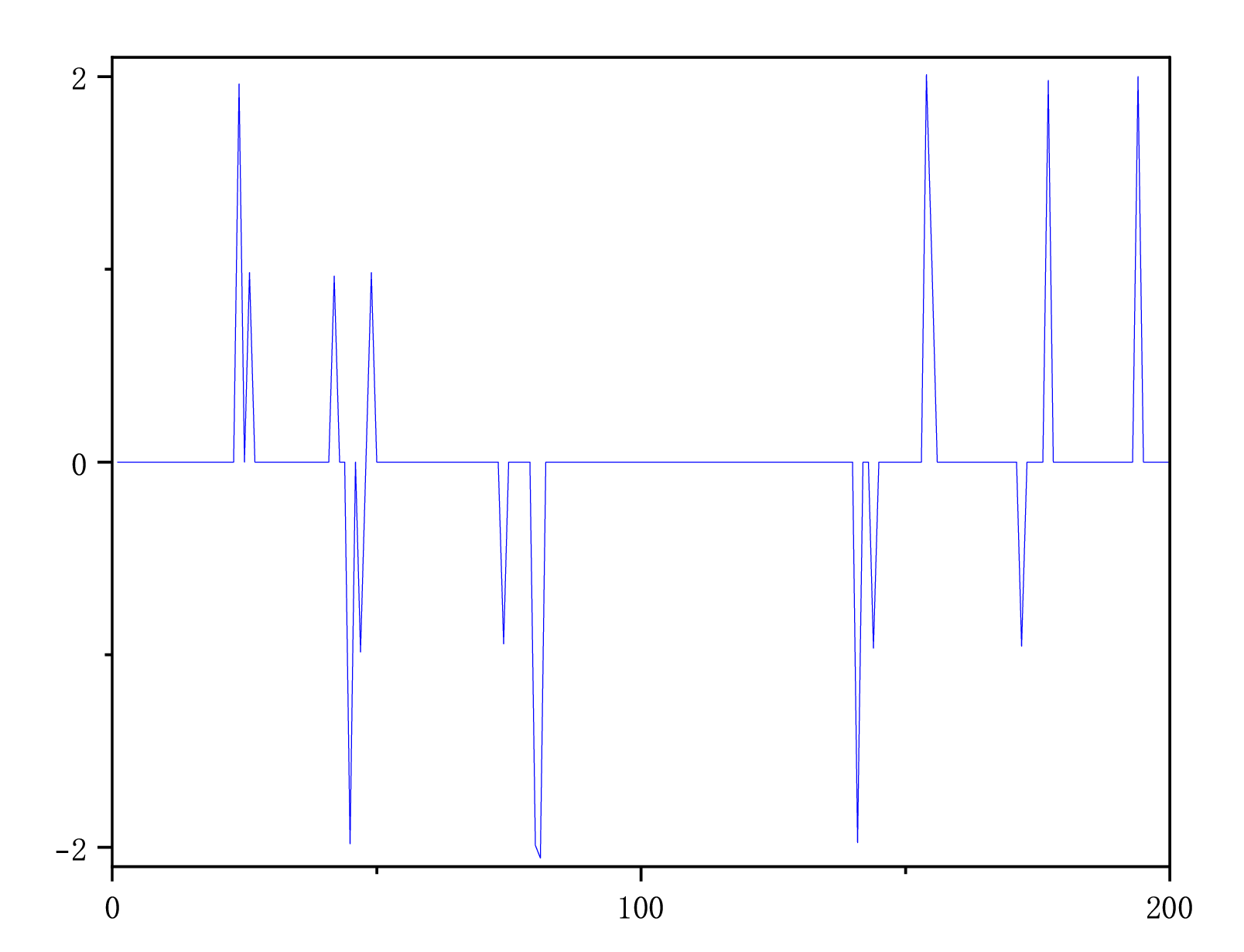}}
    \subfigure[PG-$\left(\ell_{1}^{2}-\eta\ell_{2}^{2}\right)$ algorithm, {\rm SNR} = 33.1090, Time = 0.04s ]{\includegraphics[width=0.45\columnwidth,height=0.22\linewidth]{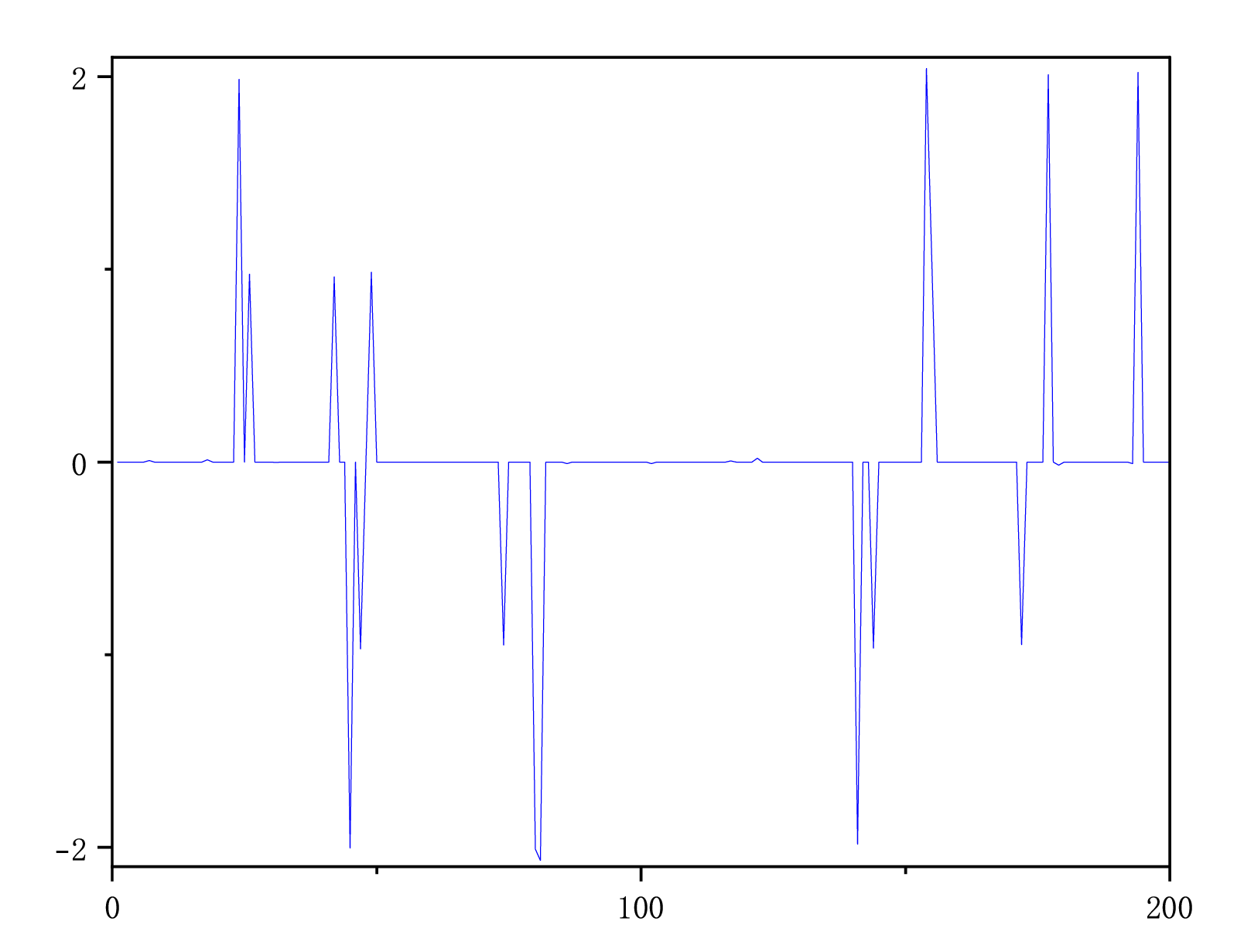}}
    \caption{(a) True signal.\ (b) True data and Observed data with $\delta$=40dB.\ (c)-(h) The reconstructed signal of ISTA, FISTA, ST-$\left(\alpha\ell_{1}-\beta\ell_{2}\right)$, HV-$\left(\ell_{1}^{2}-\eta\ell_{2}^{2}\right)$, HT-$\ell_{1/2}$ and PG-$\left(\ell_{1}^{2}-\eta\ell_{2}^{2}\right)$ algorithms with their computing times.}
    \label{figure:4}
\end{figure}

\par In our initial test, Gaussian noise with a level of $\sigma=40$dB $\left(\delta\approx 0.0775\right)$ is introduced to the exact data $y$ by calling $y^{\delta}={\rm awgn}(Ax^{\dagger},\delta)$. To investigate the impact of the parameter $\eta$, we evaluated different values for the parameters $\eta$ and $\alpha$. As shown in Table \ref{table:1}, the proposed algorithm demonstrates satisfactory performance when appropriate regularization parameters are chosen. The results indicate that the reconstruction quality improves as $\eta$ increases for a fixed $\alpha$.

\begin{table}[H]
\caption{{\rm SNR} of reconstruction $x^{*}$  utilizing HV-$\left(\ell_{1}^{2}-\eta\ell_{2}^{2}\right)$ algorithm with various noise level.}
\label{table:2}
\centering
\tabcolsep=0.35cm
\tabcolsep=0.01\linewidth
\begin{tabular}{ccccccccccc}
\hline
$\delta$ & $\alpha$ & \multicolumn{1}{c}{$\eta$=0} & \multicolumn{1}{c}{$\eta$=0.2} & \multicolumn{1}{c}{$\eta$=0.4} & \multicolumn{1}{c}{$\eta$=0.6} & \multicolumn{1}{c}{$\eta$=0.8} & \multicolumn{1}{c}{$\eta$=1.0} \\
\hline
Noise free & 1.0$\times 10^{-5}$ & 49.7239 & 49.8394 & 49.9561 & 50.0741 & 50.1934 & \textbf{50.3140} \\
$\sigma=50$dB & 2.0$\times 10^{-5}$ & 38.4811 & 38.6175 & 38.7649 & 39.1086 & 39.2265 & \textbf{39.3472} \\
$\sigma=40$dB & 6.0$\times 10^{-5}$ & 30.6048 & 30.6679 & 30.7327 & 30.8002 & 30.8745 & \textbf{31.1792} \\
$\sigma=30$dB & 3.2$\times 10^{-5}$ & 18.1063 & 18.1860 & 18.2707 & 18.6270 & 18.7666 & \textbf{18.8277} \\
$\sigma=20$dB & 9.0$\times 10^{-4}$ & 8.9087 & 8.9447 & 8.9812 & 9.0181 & 9.0554 & \textbf{9.0930} \\
\hline
\end{tabular}
\end{table}

\par Subsequently, we evaluate the stability of the proposed algorithm for various noise levels $\delta$, which are introduced to the precise data $y^{\dagger}$. The optimal regularization parameters are determined based on the discrepancy principle. The numerical outcomes are detailed in Table \ref{table:2}. A discernible trend indicates that the ${\rm SNR}$ of the reconstructed output $x^{*}$ decreases as the noise levels increase. In the absence of noise, the algorithm exhibits enhanced performance with elevated values of $\eta$, with $\eta=1$ as the optimal configuration. These findings align with the theoretical predictions inherent in the proposed non-convex regularization framework. Additionally, we conducted a comparative analysis of the reconstruction efficacy and computational efficiency of various algorithms. Fig.\ \ref{figure:4}(a)-(b) illustrate the true signal and observed data with $\sigma$=40dB $\left(\delta\approx 0.0775\right)$. Fig.\ \ref{figure:4}(c)-(h) present the reconstructed signals derived from the ISTA, FISTA, and ST-$\left(\alpha\ell_{1}-\beta\ell_{2}\right)$, HV-$\left(\ell_{1}^{2}-\eta\ell_{2}^{2}\right)$, HT-$\ell_{1/2}$, and PG-$\left(\ell_{1}^{2}-\eta\ell_{2}^{2}\right)$ algorithms, accompanied by their respective computational times. The HV-$\left(\ell_{1}^{2}-\eta\ell_{2}^{2}\right)$ and PG-$\left(\ell_{1}^{2}-\eta\ell_{2}^{2}\right)$ algorithms demonstrate superior reconstruction performance compared to the other methodologies, except the HT-$\ell_{1/2}$ algorithm. 

\begin{figure}[htbp]
    \centering
    \includegraphics[scale=0.3]{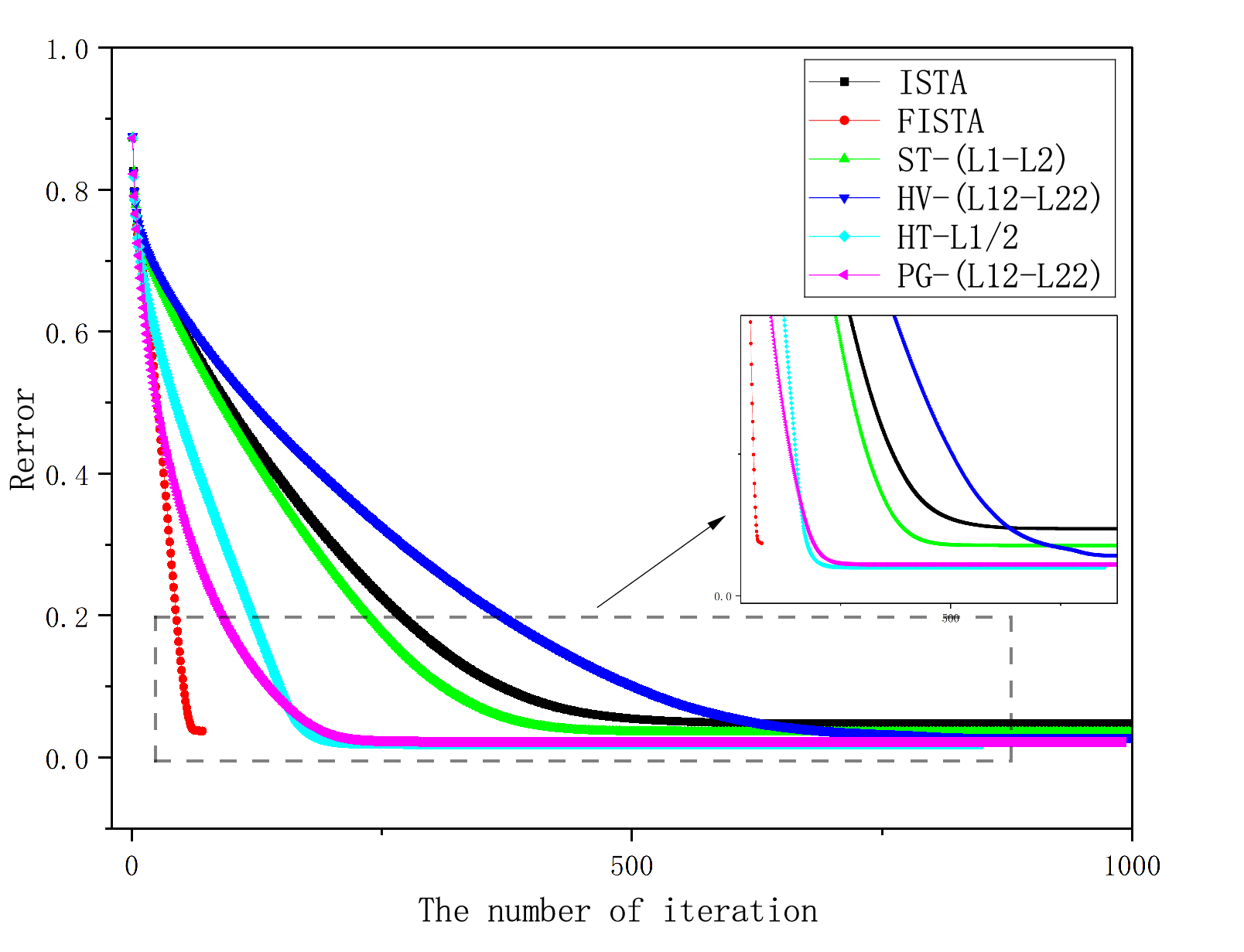}
    \caption{The change curve of Rerror for different algorithms with the increase of iteration.}
    \label{figure:5}
\end{figure}

\par To further elucidate the reconstruction performance of the various algorithms, we present the curve illustrating the change in Rerror as a function of increasing iterations in Fig.\ \ref{figure:5}. Although the Rerror of FISTA experiences a rapid decline, the adaptive increase in the step size results in a suboptimal reconstruction effect, positioning it only marginally above ISTA. Conversely, the Rerror associated with the PG-$\left(\ell_{1}^{2}-\eta\ell_{2}^{2}\right)$ algorithm decreases swiftly, yielding a reconstruction quality that ranks just below that of the HT-$\ell_{1/2}$ algorithm. Notably, the Rerror for HV-$\left(\ell_{1}^{2}-\eta\ell_{2}^{2}\right)$ algorithm declines at a slower rate. However, it still outperforms the ISTA, FISTA, and ST-$\left(\alpha\ell_{1}-\beta\ell_{2}\right)$ algorithms in terms of reconstruction quality. The change curve of the Rerror for different algorithms with the increase of computing time is shown in Fig.\ \ref{figure:6}. Regarding computational time, the HV-$\left(\ell_{1}^{2}-\eta\ell_{2}^{2}\right)$ algorithm exhibits the most extended calculation duration, attributed to its requirement for solving a multivariate linear equation. In contrast, the PG-$\left(\ell_{1}^{2}-\eta\ell_{2}^{2}\right)$ algorithm significantly outperforms the others in speed.

\begin{figure}[htbp]
    \centering
    \subfigure[ISTA and ST-$\left(\alpha\ell_{1}-\beta\ell_{2}\right)$ algorithm]{\includegraphics[scale=0.3]{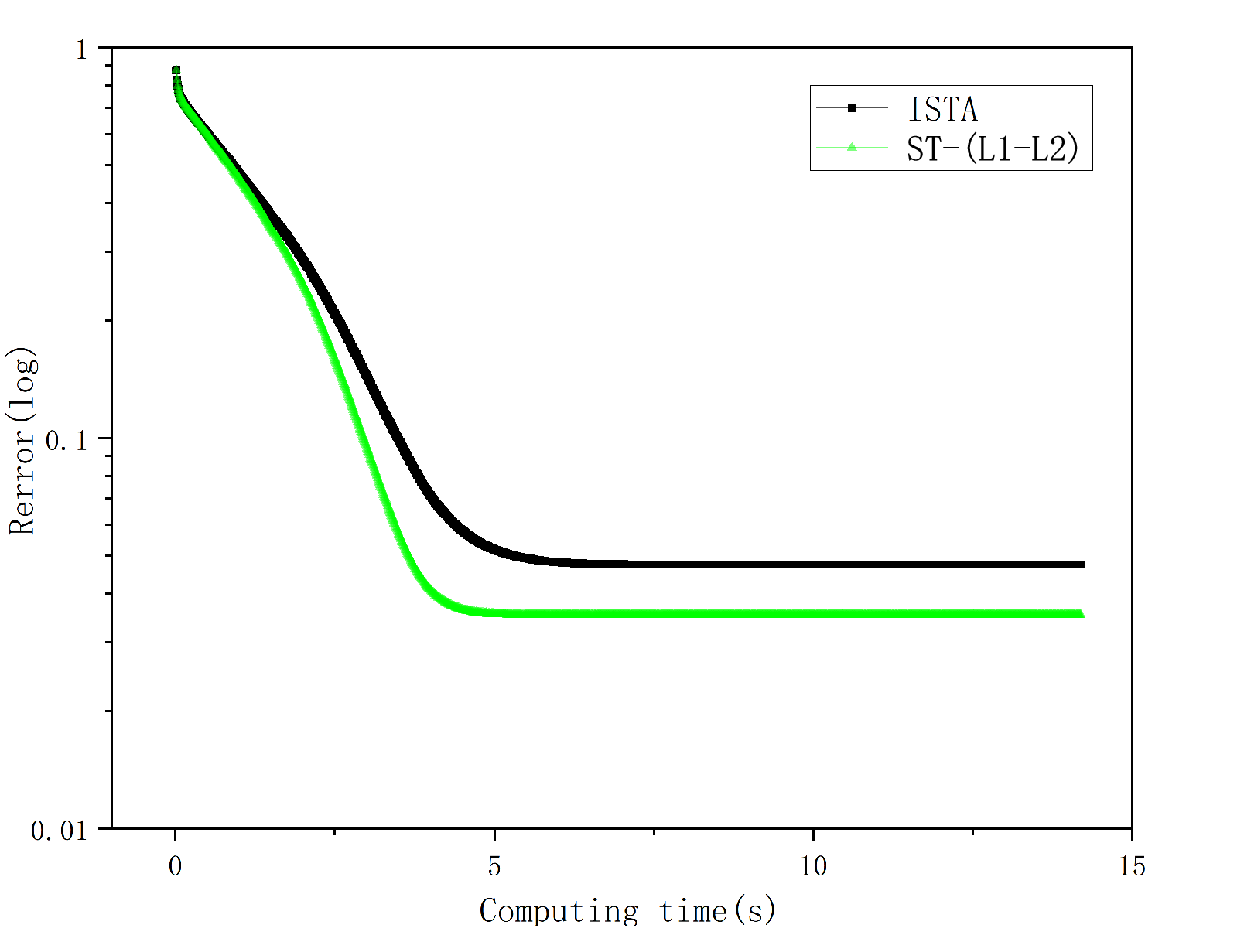}}
    \subfigure[FISTA, HT-$\ell_{1/2}$ and PG-$\left(\ell_{1}^{2}-\eta\ell_{2}^{2}\right)$ algorithms ]{\includegraphics[scale=0.3]{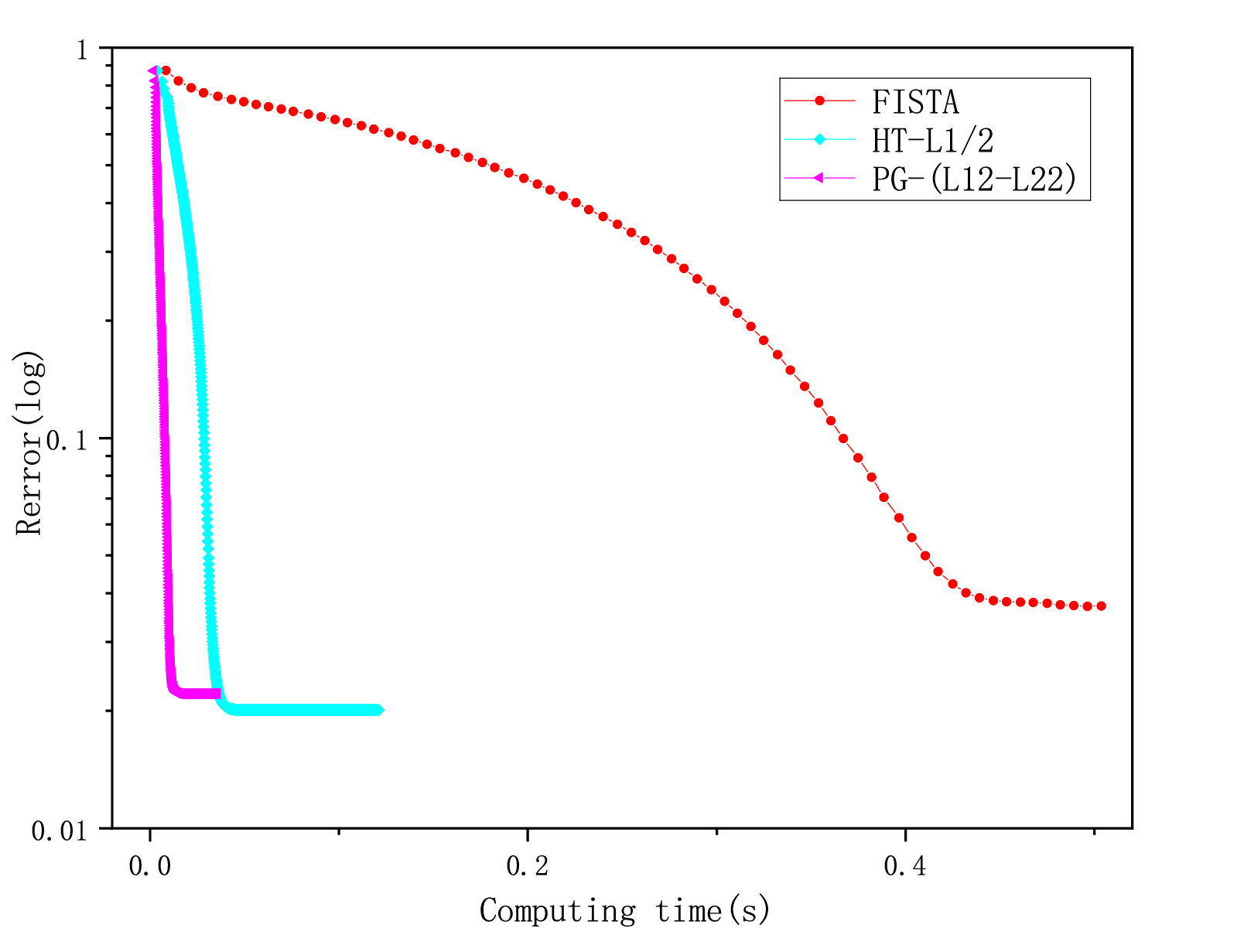}}
    \caption{The change curve of the Rerror for different algorithms with the increase of computing time.}
    \label{figure:6}
\end{figure}

\subsection{Ill-conditioned image deblurring problem}\label{sec6.2}

\par In the second example,  we investigate the ill-conditioned image deblurring problem, which entails removing blurring artifacts from images, commonly arising from factors such as defocus aberration or motion blur. The blurring effect is typically modeled by a Fredholm integral equation of the first kind, expressed as
\begin{equation*}
    \int_{a}^{b}K\left(s,t\right)f\left(t\right)dt = g\left(s\right),
\end{equation*}
where $K\left(s,t\right)$ represents the kernel function, $g\left(s\right)$ denotes the observed image, and $f\left(t\right)$ signifies the true image. We leverage the blur problem implemented in Matlab regularization tools \cite{H07} by calling $[A,b,x^{\dagger}]={\rm blur}(n,{\rm band},\sigma)$, where a Gaussian point-spread function serves as the kernel function defined by
\begin{equation*}
    K\left(s,t\right)=\frac{1}{\pi \sigma^{2}}\exp\left(-\frac{s^{2}+t^{2}}{2\sigma^{2}}\right).
\end{equation*}
The matrix $A$ is characterized as a symmetric $n^{2}\times n^{2}$ Toeplitz matrix, articulated as $A=\left(2\pi\sigma^{2}\right)^{-1}T\otimes T$, with $T$ being an $n\times n$ symmetric banded Toeplitz matrix. The first row of $T$  can be formulated by calling
\begin{equation*}
    z=\left[\exp\left(-\left([0:{\rm band}-1].^{2}\right)/\left(2\sigma^{2}\right)\right); {\rm zeros}\left(1:N-{\rm band}\right)\right].
\end{equation*}
Notably, the parameter $\sigma$ modulates the shape of the Gaussian point spread function, thereby influencing the degree of smoothing; specifically, a larger value for $\sigma$ entails a wider function and renders the problem less ill-posed. In our investigation, we select $n=16$, ${\rm band}=3$, and $\sigma=0.7$. The computed $\left\|A\right\|_{\ell_{2}}$ approximates 1, with a condition number near 30. The adopted settings closely mirror those outlined in Section \ref{sec6.1}. In order to compare the reconstruction effects of different iterative algorithms fairly, we set the maximum number of iterations to ${\rm maxiter} = 1500$ and $L^{k}=1$ for the HV-$\left(\ell_{1}^{2}-\eta\ell_{2}^{2}\right)$ algorithm, $\gamma=1$ for the PG-$\left(\ell_{1}^{2}-\eta\ell_{2}^{2}\right)$ algorithm, while $\lambda=1$ for the ST-$\left(\alpha\ell_{1}-\beta\ell_{2}\right)$ algorithm. The stopping criterion is that the error of the adjacent two iterative solutions is less than $10^{-5}$ or the maximum number of iterations is reached. The initial value $x^{0}$ is generated by calling $x^{0}=0.01{\rm ones}(n,1)$.

\begin{table}[H]
\caption{{\rm SNR} of reconstruction $x^{*}$  utilizing HV-$\left(\ell_{1}^{2}-\eta\ell_{2}^{2}\right)$ algorithm for different values of $\eta$ and $\alpha$ with $\sigma=40$dB.}
\label{table:3}
\centering
\tabcolsep=0.35cm
\tabcolsep=0.01\linewidth
\scalebox{0.9}{
\begin{tabular}{ccccccc}
\hline
$\eta$ & \multicolumn{1}{c}{$\alpha=5.0\times10^{-5}$} & \multicolumn{1}{c}{$\alpha=8.0\times10^{-5}$} & \multicolumn{1}{c}{$\alpha=1.1\times10^{-4}$} & \multicolumn{1}{c}{$\alpha=1.4\times10^{-4}$} & \multicolumn{1}{c}{$\alpha=1.7\times10^{-4}$} & \multicolumn{1}{c}{$\alpha=2.0\times10^{-4}$} \\
\hline
0.0  & 26.6327 & 27.9979 & 29.2200 & 29.8741 &29.8108 & 28.9320 \\
0.1  & 26.7356 & 28.5705 & 29.5722 & 30.1109 &30.0594 & 29.2123 \\
0.2  & 26.8115 & 28.7923 & 29.7950 & 30.3511 &30.3093 & 29.4968 \\
0.3  & 26.8720 & 29.0215 & 29.9938 & 30.5950 &30.5595 & 29.7849 \\
0.4  & 27.1370 & 29.2561 & 30.0614 & 30.8492 &30.8092 & 30.0755 \\
0.5  & 27.2688 & 29.4985 & 31.2287 & 31.0965 &31.0568 & 30.3674 \\
0.6  & 27.4016 & 29.7557 & 31.3932 & 31.3616 &31.3008 & 30.6589 \\
0.7  & 27.5354 & 30.0958 & 31.5541 & 32.1046 &31.5394 & 30.9418 \\
0.8  & 27.6702 & 30.3300 & 31.7107 & 32.2781 &31.7704 & 31.2325 \\
0.9  & 27.8509 & 30.5713 & 31.8620 & 32.4401 &31.9913 & 31.5089 \\
1.0  & \textbf{27.9425} & \textbf{30.8360} & \textbf{32.0071}  & \textbf{32.5916} & \textbf{32.1994} & \textbf{31.7739} \\
\hline
\end{tabular}}
\end{table}

\begin{figure}[htbp]
    \centering
    \subfigure[True image.]{\includegraphics[width=0.45\columnwidth,height=0.3\linewidth]{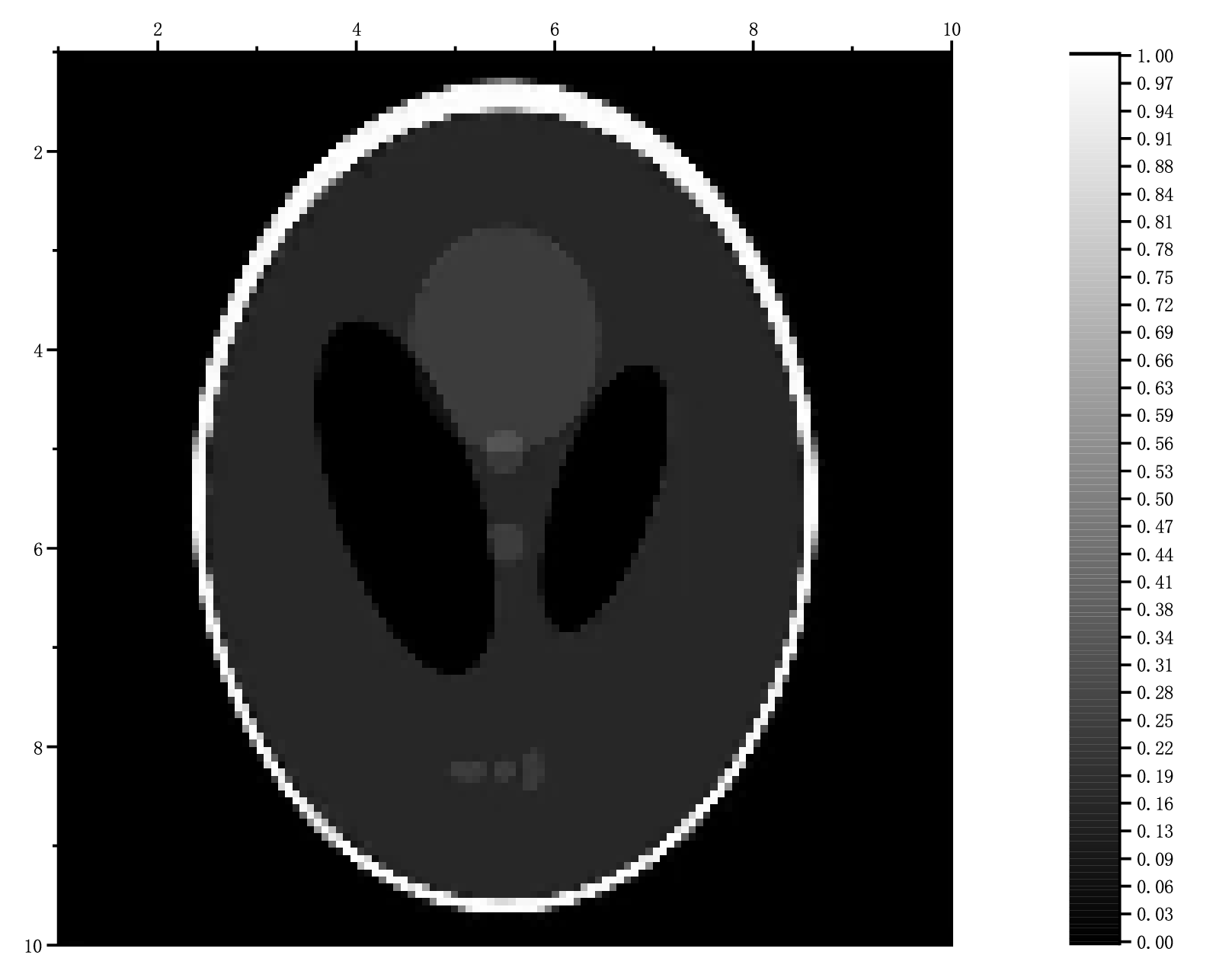}}
    \subfigure[Blurring image($\sigma$=60dB).]{\includegraphics[width=0.45\columnwidth,height=0.3\linewidth]{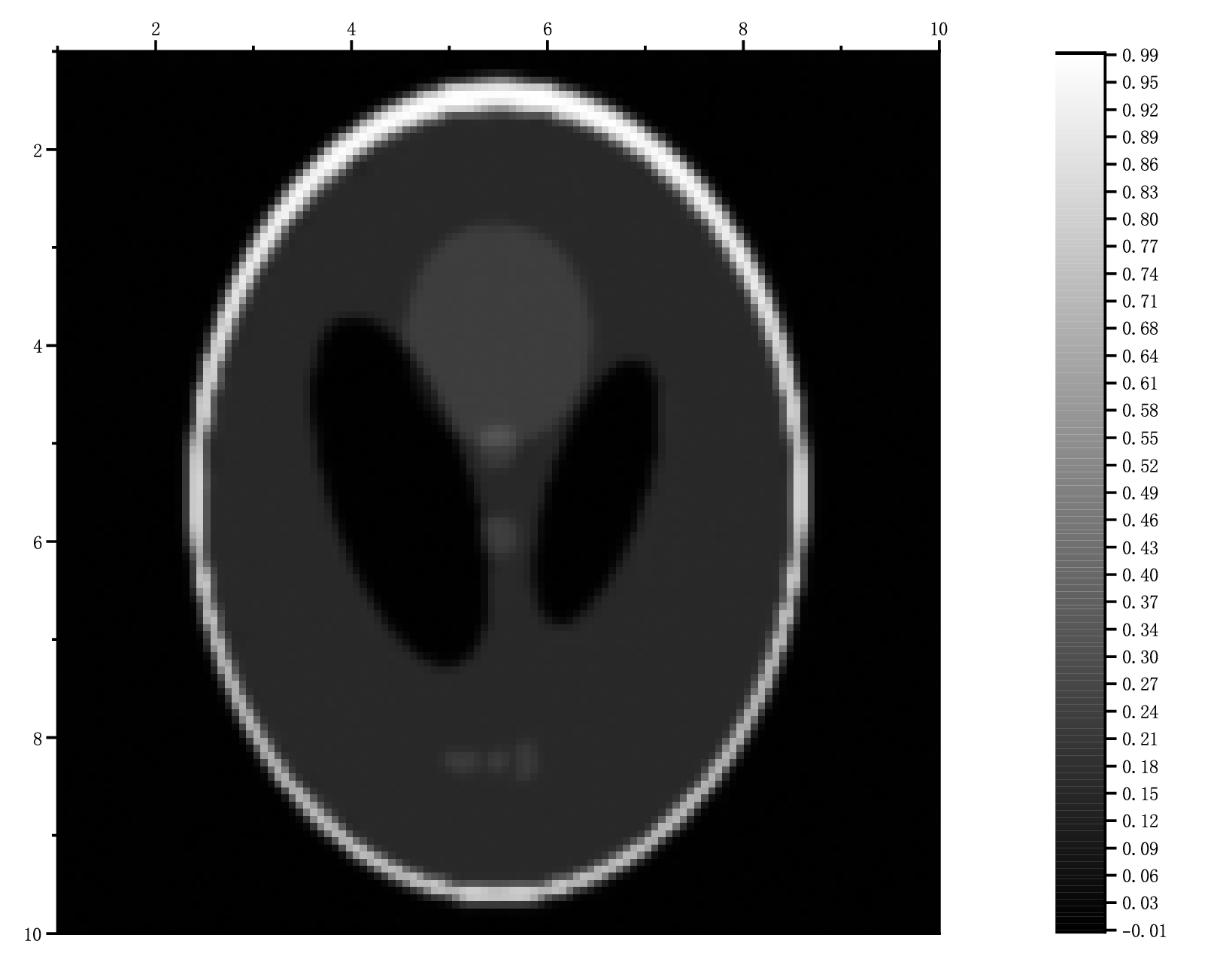}}\\
    \subfigure[HV-$\left(\ell_{1}^{2}-\eta\ell_{2}^{2}\right)$ algorithm, {\rm SNR} = 32.5106, Time = 641.62s ]{\includegraphics[width=0.45\columnwidth,height=0.3\linewidth]{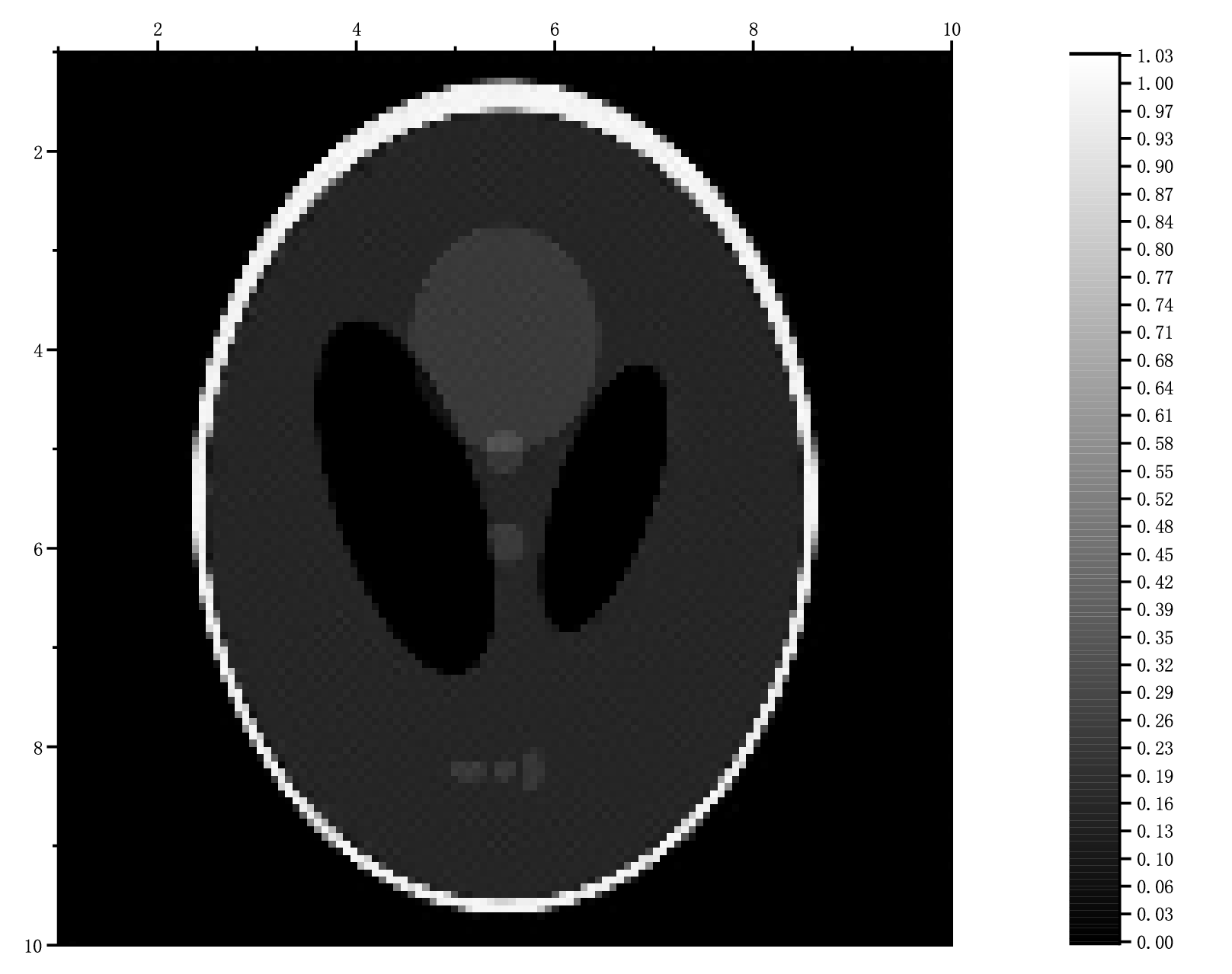}}
    \subfigure[HT-$\ell_{1/2}$ algorithm, {\rm SNR} = 30.6295, Time = 290.40s ]{\includegraphics[width=0.45\columnwidth,height=0.3\linewidth]{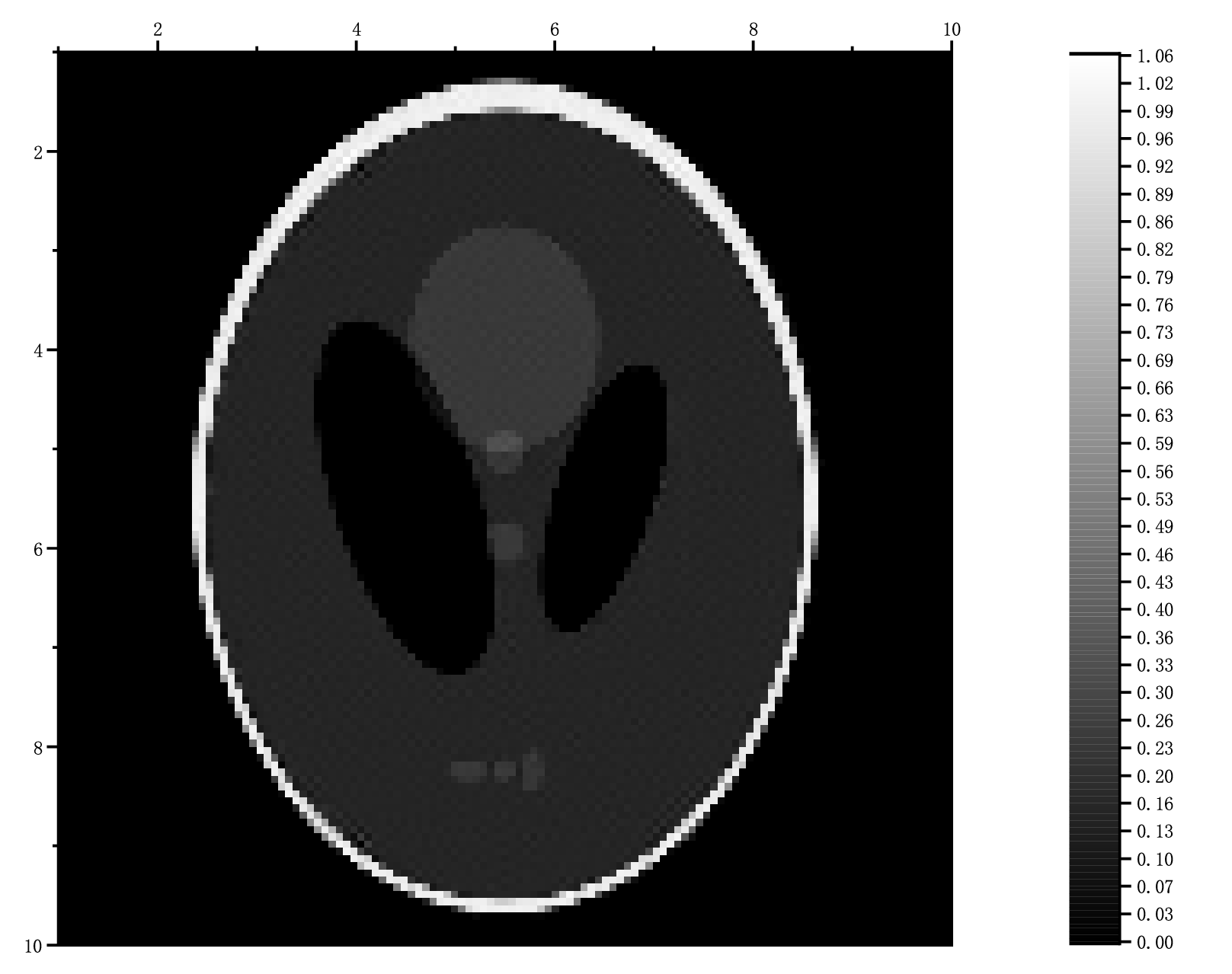}}\\
    \subfigure[PG-$\left(\ell_{1}^{2}-\eta\ell_{2}^{2}\right)$ algorithm, {\rm SNR} = 31.6984, Time = 238.39s ]{\includegraphics[width=0.45\columnwidth,height=0.3\linewidth]{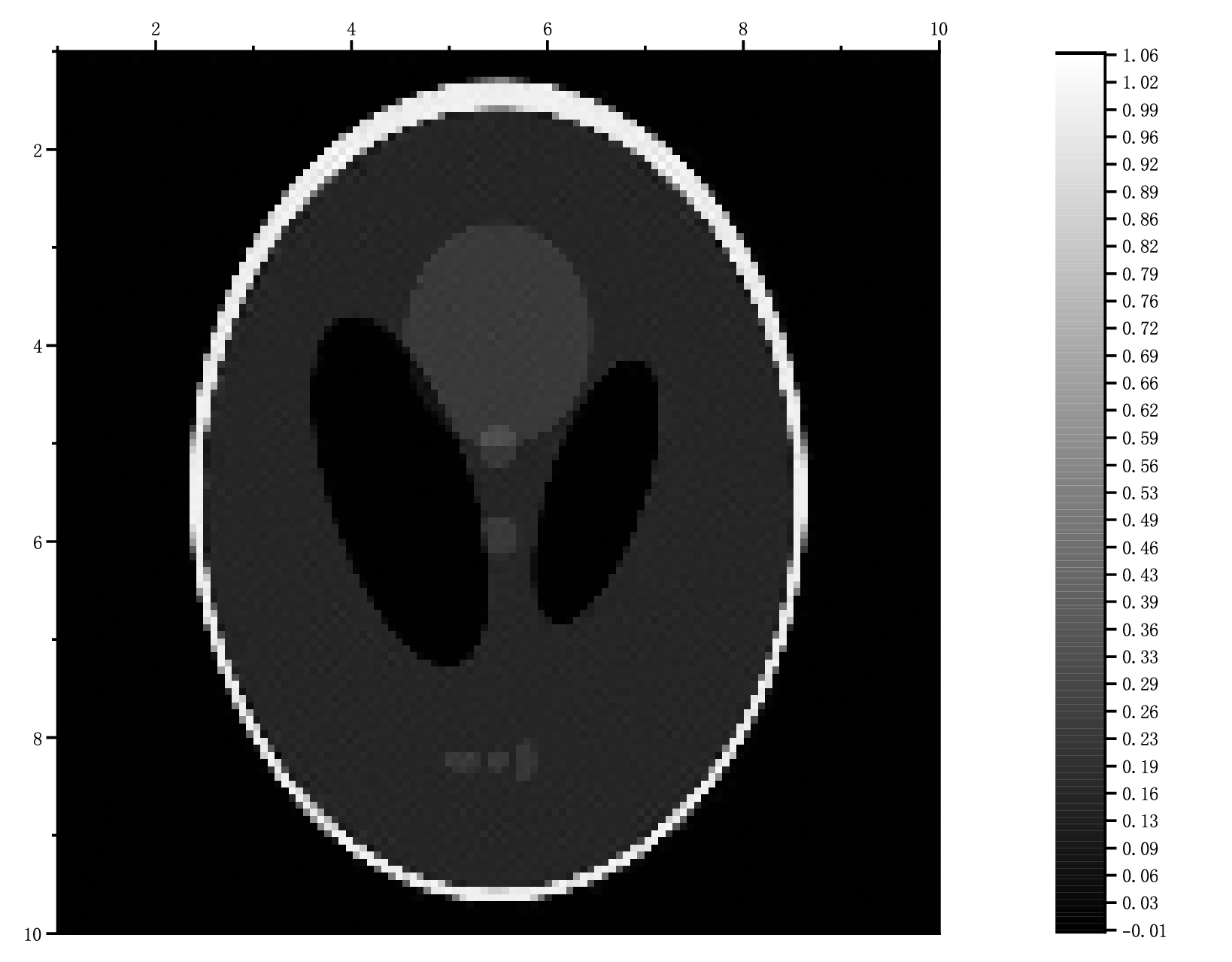}}
    \caption{(a) True image.\ (b) Blurring image with $\delta$=60dB.\ (c)-(e) The The deblurring image of HV-$\left(\ell_{1}^{2}-\eta\ell_{2}^{2}\right)$, HT-$\ell_{1/2}$ and PG-$\left(\ell_{1}^{2}-\eta\ell_{2}^{2}\right)$ algorithms with their computing times.}
    \label{figure:7}
\end{figure}

\par Table \ref{table:3} delineates the performance of the deblurring process under different regularization parameters $\alpha$ and values of $\eta$. The results indicate that reconstruction quality does not consistently improve with increasing $\eta$. Notably, the choice of $\eta=1$ is not optimal. Aligning with the findings from Section \ref{sec6.1}, we observe that performance stability is more pronounced for the case where $\eta=1$ in conjunction with the regularization parameter $\alpha$. Moreover, achieving satisfactory results with larger and smaller values of $\alpha$ is feasible. The stability of deblurring under various noise levels $\delta$ is illustrated in Table \ref{table:4}. We observe that accuracy diminishes as the noise level increases. Notably, the stability in this case is superior to that presented in Section \ref{sec6.1}. 

\begin{table}[H]
\caption{{\rm SNR} of reconstruction $x^{*}$  utilizing HV-$\left(\ell_{1}^{2}-\eta\ell_{2}^{2}\right)$ algorithm with various noise level.}
\label{table:4}
\centering
\tabcolsep=0.35cm
\tabcolsep=0.01\linewidth
\begin{tabular}{ccccccccccc}
\hline
$\delta$ & $\alpha$ & \multicolumn{1}{c}{$\eta$=0} & \multicolumn{1}{c}{$\eta$=0.2} & \multicolumn{1}{c}{$\eta$=0.4} & \multicolumn{1}{c}{$\eta$=0.6} & \multicolumn{1}{c}{$\eta$=0.8} & \multicolumn{1}{c}{$\eta$=1.0} \\
\hline
Noise free & 1.0$\times 10^{-5}$ & 56.3851 & 57.3481 & 58.4322 & 59.6720 & 61.1199 & \textbf{62.8595} \\
50dB & 4.0$\times 10^{-5}$ & 37.4118 & 37.9523 & 38.5171 & 39.1230 & 40.2378 & \textbf{40.5599} \\
40dB & 1.4$\times 10^{-4}$ & 29.8741  & 30.3511 & 30.8492 & 31.3616 & 32.2781 & \textbf{32.5916} \\
30dB & 4.5$\times 10^{-4}$ & 20.6638 & 20.8338 & 20.9669 & 21.0579 & \textbf{21.1598} & 21.1047 \\
20dB & 1.6$\times 10^{-3}$ & 9.0261 & 9.0318 & 9.0383 & \textbf{9.0457} & 9.0429 & 9.0318 \\
\hline
\end{tabular}
\end{table}

\par Additionally, we change the setting $n=125$ with other settings as before and conduct a comparative analysis of the deblurring and computational efficiency associated with different algorithms. In this case, the computing times of ISTA, FISTA and ST-$\left(\alpha\ell_{1}-\beta\ell_{2}\right)$ algorithm are too long to obtain the deblurring results, so we only consider the deblurring results of HV-$\left(\ell_{1}^{2}-\eta\ell_{2}^{2}\right)$, HT-$\ell_{1/2}$ and PG-$\left(\ell_{1}^{2}-\eta\ell_{2}^{2}\right)$ algorithms. Fig.\ \ref{figure:7}(a)-(b) depict the original image and blurred image with $\sigma$=60dB $(\delta\approx 0.1246)$. Fig.\ \ref{figure:7}(c)-(e) showcase the deblurred outputs generated by the HV-$\left(\ell_{1}^{2}-\eta\ell_{2}^{2}\right)$, HT-$\ell_{1/2}$, and PG-$\left(\ell_{1}^{2}-\eta\ell_{2}^{2}\right)$ algorithms, alongside their respective computational times. It is evident that the HV-$\left(\ell_{1}^{2}-\eta\ell_{2}^{2}\right)$ and PG-$\left(\ell_{1}^{2}-\eta\ell_{2}^{2}\right)$ algorithms outperform their counterparts regarding deblurring quality.

\begin{figure}[htbp]
    \centering
    \includegraphics[scale=0.3]{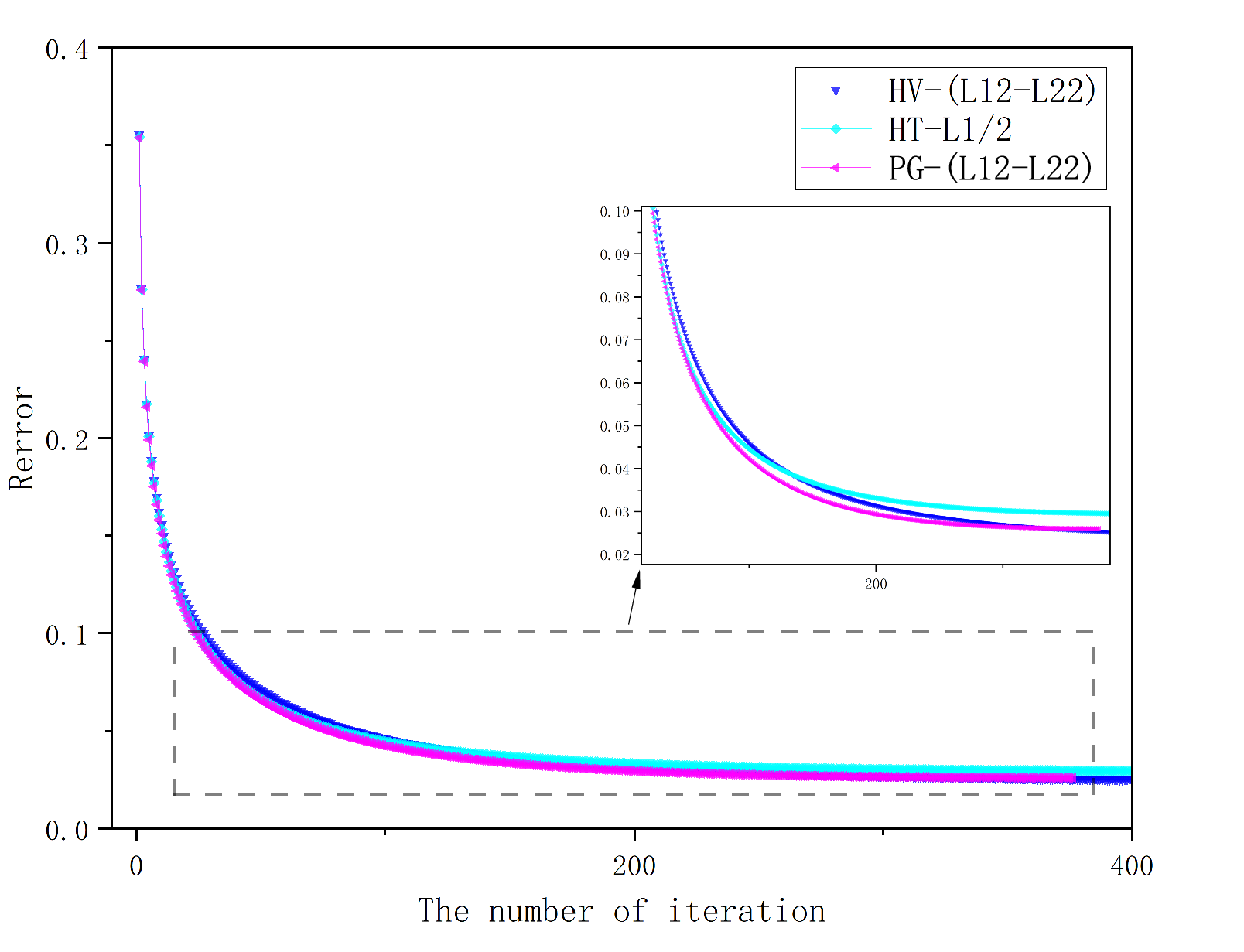}
    \caption{The change curve of Rerror for different algorithms with the increase of interation.}
    \label{figure:8}
\end{figure}

\par To better compare the deblurring performances across different algorithms, we present the Rerror change curve about increasing iterations in Fig.\ \ref{figure:8}. The PG-$\left(\ell_{1}^{2}-\eta\ell_{2}^{2}\right)$ algorithm takes the shortest time to achieve a better deblurring effect than HT-$\ell_{1/2}$, and HV-$\left(\ell_{1}^{2}-\eta\ell_{2}^{2}\right)$ algorithm takes a longer time to achieve the best deblurring effect. The change curve of the Rerror for different algorithms with the increase of computing time is shown in Fig.\ \ref{figure:9}. In terms of computational efficiency, the HV-$\left(\ell_{1}^{2}-\eta\ell_{2}^{2}\right)$ algorithm necessitates the most substantial computational resources due to the complexity of solving multivariate linear equations. However, as the scale increases, the calculation time increases within an acceptable range. In contrast, the PG-$\left(\ell_{1}^{2}-\eta\ell_{2}^{2}\right)$ algorithm demonstrates a significantly reduced computational duration compared to the others.

\begin{figure}[htbp]
    \centering
    \includegraphics[scale=0.3]{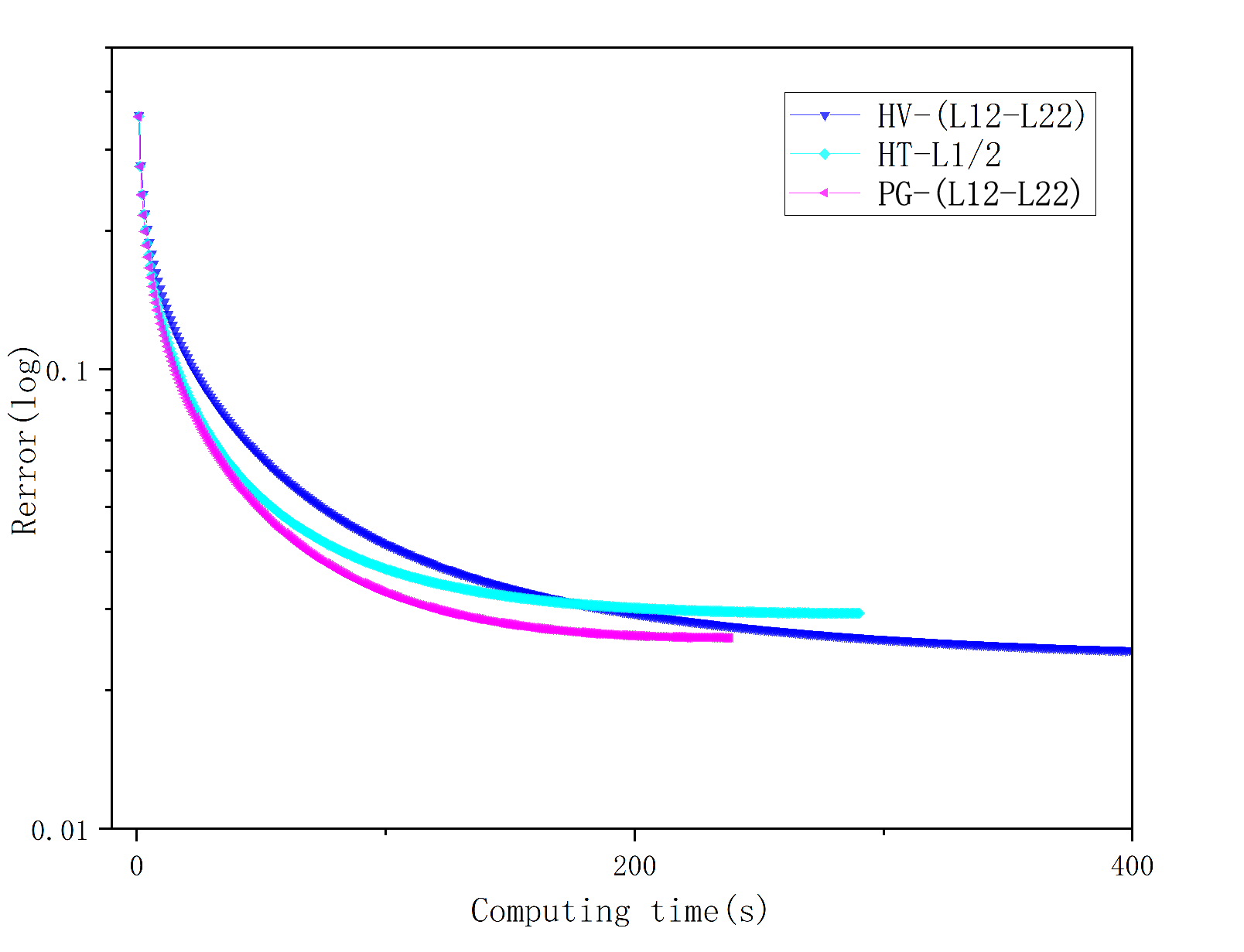}
    \caption{The change curve of Rerror for different algorithms with the increase of computing time.}
    \label{figure:9}
\end{figure}

\par Finally, to assess the effectiveness of the constraint radius selection method, we employ Algorithm \ref{alg4} to address the problems of compressive sensing (CS) and image deblurring (DB) as detailed in Section \ref{sec6}. Fig.\ \ref{figure:3} illustrates the relationship between the constraint radius and relative error throughout iterations. For CS, this evidence underscores the necessity of selecting $\|x^{\dagger}\|_{\ell_{1}}^{2}=576$, facilitating the recovery of $R= 578.1250$ through Algorithm \ref{alg4}, thereby substantiating the efficacy of the radius selection method employed. For DB, this evidence underscores the necessity of selecting $\|x^{\dagger}\|_{\ell_{1}}^{2}=2.6897{\rm E}+6$, facilitating the recovery of $R=2.6855{\rm E}+6$ through Algorithm \ref{alg4}, thereby validating the efficacy of the radius selection methodology employed.

\begin{figure}[htbp]
    \centering
    \subfigure[PG-$\left(\ell_{1}^{2}-\eta\ell_{2}^{2}\right)$ algorithm for CS]{\includegraphics[scale=0.3]{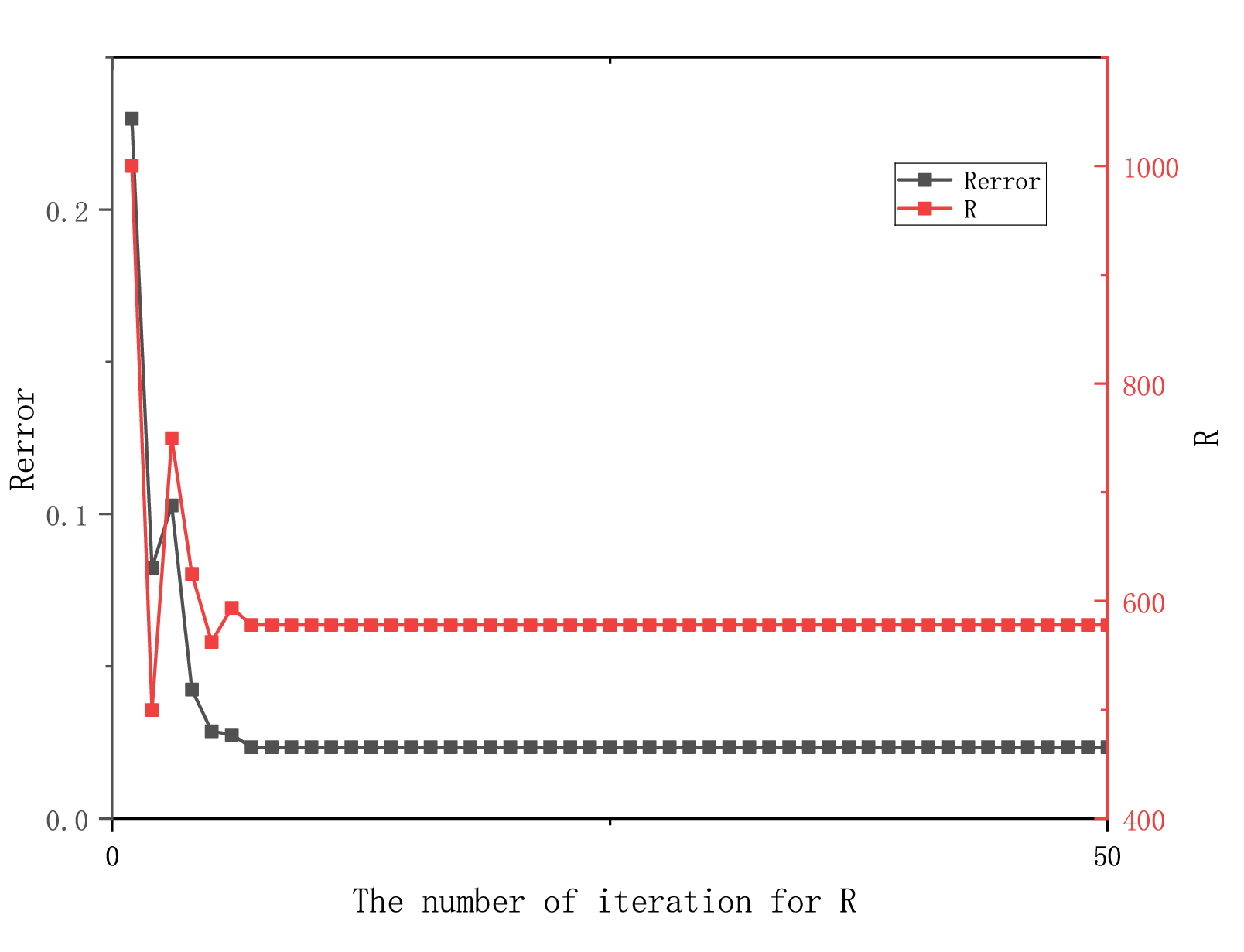}}
    \subfigure[PG-$\left(\ell_{1}^{2}-\eta\ell_{2}^{2}\right)$ algorithm for DB]{\includegraphics[scale=0.3]{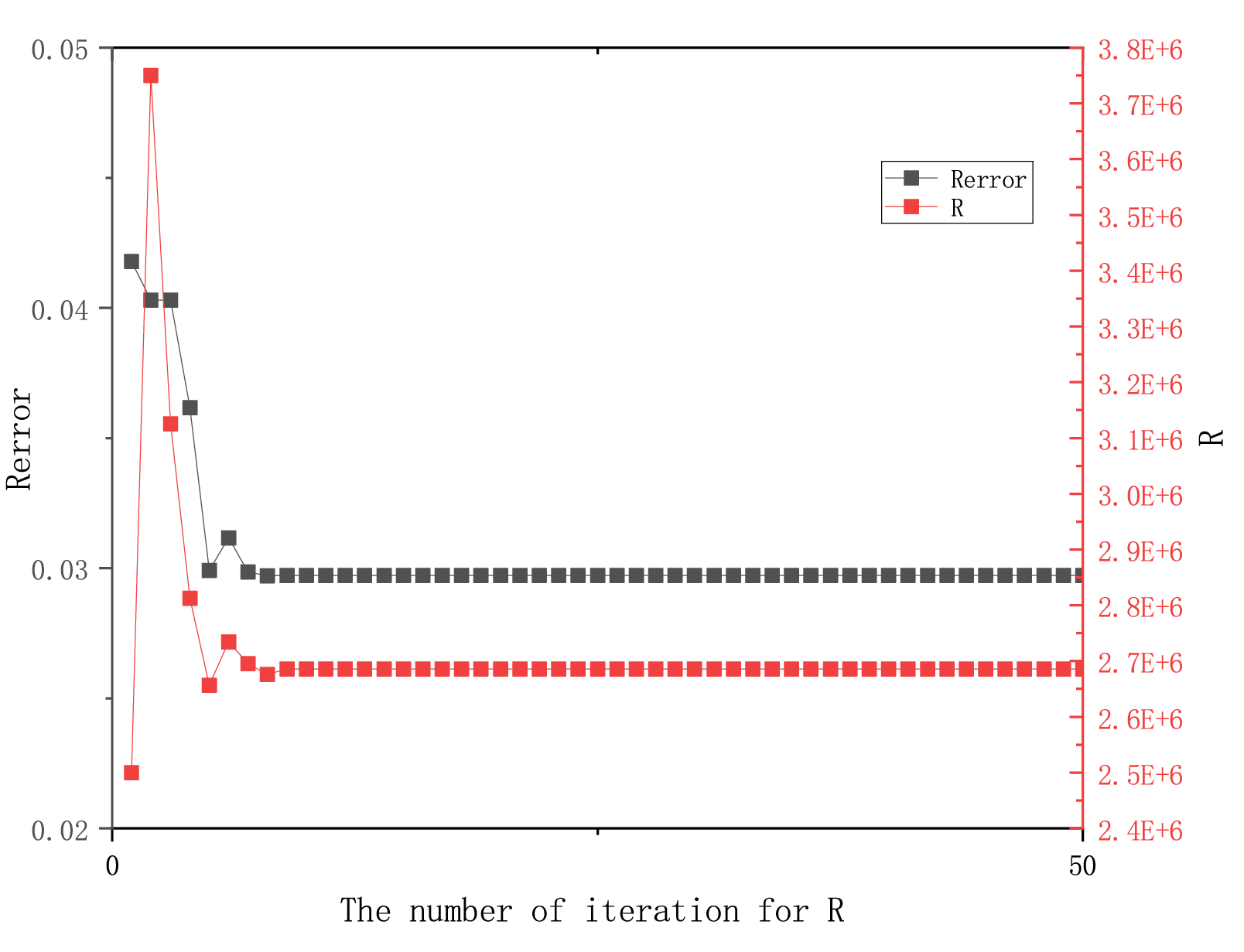}}\\
    \caption{The constraint radius and relative error vary with the number of iterations using Algorithm \ref{alg4}. (a) For compressive sensing (CS), the signal dimension is 200$\times$1 with a noise level of $\sigma = 40$dB. (b) For the image deblurring (DB), the image dimension is 125$\times$125 with a noise level of $\sigma = 60$dB.}
    \label{figure:3}
\end{figure}

\section{Conclusion}\label{sec7}

\par We have proposed and conducted a comprehensive analysis of a novel non-convex $\ell_{1}^{2}-\eta\ell_{2}^{2}$ $\left(0<\eta\leq 1\right)$ regularization technique, aimed at enhancing sparse recovery. The theoretical framework establishes a convergence rate of $\mathcal{O}\left(\delta\right)$ under a specified source condition, which applies to both a priori and a posteriori parameter selection strategies. To implement this methodology, we introduced an HV-$\left(\ell_{1}^{2}-\eta\ell_{2}^{2}\right)$ algorithm derived from the proximal gradient method, designed to effectively address the $\ell_{1}^{2}-\eta\ell_{2}^{2}$ sparse regularization problem. Additionally, we developed a PG-$\left(\ell_{1}^{2}-\eta\ell_{2}^{2}\right)$ algorithm that leverages a combination of projected gradient techniques and surrogate function methods to accelerate the solution process for the stated optimization problem. Numerical experiments demonstrate that our proposed algorithm significantly outperforms traditional $\ell_{1}$ and $\alpha\ell_{1}-\beta\ell_{2}$ regularization, regardless of whether the operator $A$ is well-conditioned or ill-conditioned. Furthermore, it performs comparably to the $\ell_{1/2}$ regularization. These findings suggest that the $\ell_{1}^{2}-\eta\ell_{2}^{2}$ regularization framework not only surpasses conventional sparsity-promoting methods but also represents a compelling alternative to $\ell_{p}\left(0<p\leq 1\right)$  regularization techniques.

\section{Conflict of interest statement}
There are no conflicts of interest regarding this submission, and all authors have approved the manuscript for publication.

\end{document}